\documentclass{article}

 \usepackage[preprint]{neurips_2026}

\usepackage[utf8]{inputenc} 
\usepackage[T1]{fontenc}    
\usepackage{url}            
\usepackage{booktabs}       
\usepackage{amsfonts}       
\usepackage{nicefrac}       
\usepackage{microtype}      
\usepackage{xcolor}         
\usepackage{comment}
\usepackage{csquotes}

\usepackage{amsmath}
\usepackage{algorithm}
\usepackage{algpseudocode}
\usepackage{graphicx}
\usepackage{amsthm}
\usepackage{amssymb}
\usepackage{thmtools}
\usepackage{thm-restate}
\usepackage[colorlinks=true, allcolors=blue]{hyperref}
\usepackage{subcaption}

\usepackage{enumitem} 

\newtheorem{theorem}{Theorem}
\newtheorem{lemma}{Lemma}

\newtheorem{remark}{Remark}
\newtheorem*{theorem*}{Theorem}
\newtheorem*{lemma*}{Lemma}

\usepackage[colorinlistoftodos,bordercolor=orange,backgroundcolor=orange!20,linecolor=orange,textsize=scriptsize]{todonotes}

\newcommand{\bi}{\begin{itemize}}
\newcommand{\ei}{\end{itemize}}

\title{FlexGrad: \\ A Root-Free Approach to Adaptive Step Sizes}

\author{%
  Pablo Barros \\
  Fundação Getulio Vargas \\
  \And
  Aaron Defazio \\
  Fundamental AI Research Team, Meta \\
  \And
  Vincent Guigues \\
  Fundação Getulio Vargas
}

\begin{document}

\maketitle

\begin{abstract}
We introduce the FlexGrad optimizer, an adaptive gradient method that addresses a key limitation of classical AdaGrad-type methods: their inability to exploit correlations between gradients. By scaling steps using both cumulative gradients and cumulative squared gradients, FlexGrad preserves AdaGrad-like behavior in noisy regimes while allowing substantially larger steps when gradient directions are persistent. The method is fully online, horizon-free, and avoids any auxiliary search procedure during optimization. For convex Lipschitz objectives, we prove that FlexGrad attains the optimal asymptotic convergence rate for nonsmooth convex optimization, with guarantees for both global and coordinate-wise variants. FlexGrad is straightforward to implement and is closely connected to practical gradient-normalized methods. Experiments on a range of convex and nonconvex problems show that FlexGrad performs competitively across a range of optimization settings.
\end{abstract}

\section{Introduction}

Adaptive optimization methods form the foundation on which modern machine learning is built. Understanding and addressing their limitations is crucial to putting the field on more solid ground. This work focuses on one key limitation of classical AdaGrad-family \citep{adagrad} optimizers: \emph{their blindness to correlations between gradients.} AdaGrad adapts the global learning rate $\gamma$ to each coordinate $i$ by using the sum of the squares of gradients for that coordinate:
\[
    \gamma_{k+1,i} = \frac{\gamma}{\sqrt{\sum_{j=0}^{k} g^2_{j,i}}}.
\]
This update scales the step size down when gradients are large, even when far from the solution, which can slow down training. We propose the use of a more flexible step size that does not use a square-root on the denominator:
\[
\gamma_{k+1,i}=\frac{\gamma\left|\sum_{j=0}^{k}g_{j,i}\right|}{\sum_{j=0}^{k}g_{j,i}^{2}}.
\]
This scaling maintains the desirable properties of the AdaGrad step size, decreasing at a $1/\sqrt{k}$ rate when the gradients are noisy or near the optimum, while maintaining larger step sizes when gradient signal is strong and consistent.
\begin{description}
    \item[Uncorrelated Gradients] When consecutive gradients are uncorrelated, the absolute value in the numerator behaves like a random-walk, growing as $\sqrt{\sum_{j=0}^{k}g_{j,i}^{2}}$, resulting in a similar step size to conventional AdaGrad. \[
\gamma_{k+1,i}=\frac{\gamma\left|\sum_{j=0}^{k}g_{j,i}\right|}{\sum_{j=0}^{k}g_{j,i}^{2}}\approx\frac{\gamma\sqrt{\sum_{j=0}^{k}g_{j,i}^{2}}}{\sum_{j=0}^{k}g_{j,i}^{2}}=\frac{\gamma}{\sqrt{\sum_{j=0}^{k}g_{j,i}^{2}}}
    \]
    \item[Correlated Gradients] When gradients are correlated, the numerator becomes larger, resulting in much larger steps (Figure~\ref{fig:flex-example}) and a non-decreasing step size:
\[
\gamma_{k+1,i}=\frac{\gamma\left|\sum_{j=0}^{k}g_{j,i}\right|}{\sum_{j=0}^{k}g_{j,i}^{2}}\approx\frac{\gamma(k+1)G}{(k+1)G^{2}}=\frac{\gamma}{G}.
\]
\end{description}

This approach is not merely heuristic; we establish a strong convergence theory for both coordinate-wise and global versions of this method.

The paper is organized as follows. Section~\ref{sec:literature} reviews the related literature, and Section~\ref{sec:motivation} motivates FlexGrad. Section~\ref{sec:convergence-theory} introduces FlexGrad and its coordinate-wise variant, and proves their convergence. Section~\ref{sec:ema} develops Exponential-Moving-Average (EMA) variants. Finally, Section~\ref{sec:numsim} reports numerical experiments on standard machine learning problems, comparing FlexGrad, coordinate-wise FlexGrad, their EMA variants, AdaGrad \cite{adagrad}, Adam \cite{kingma2014adam}, BCos , and GradaGrad \cite{defazio2022grad}.

\begin{figure}
    \centering
    \includegraphics[width=1\linewidth]{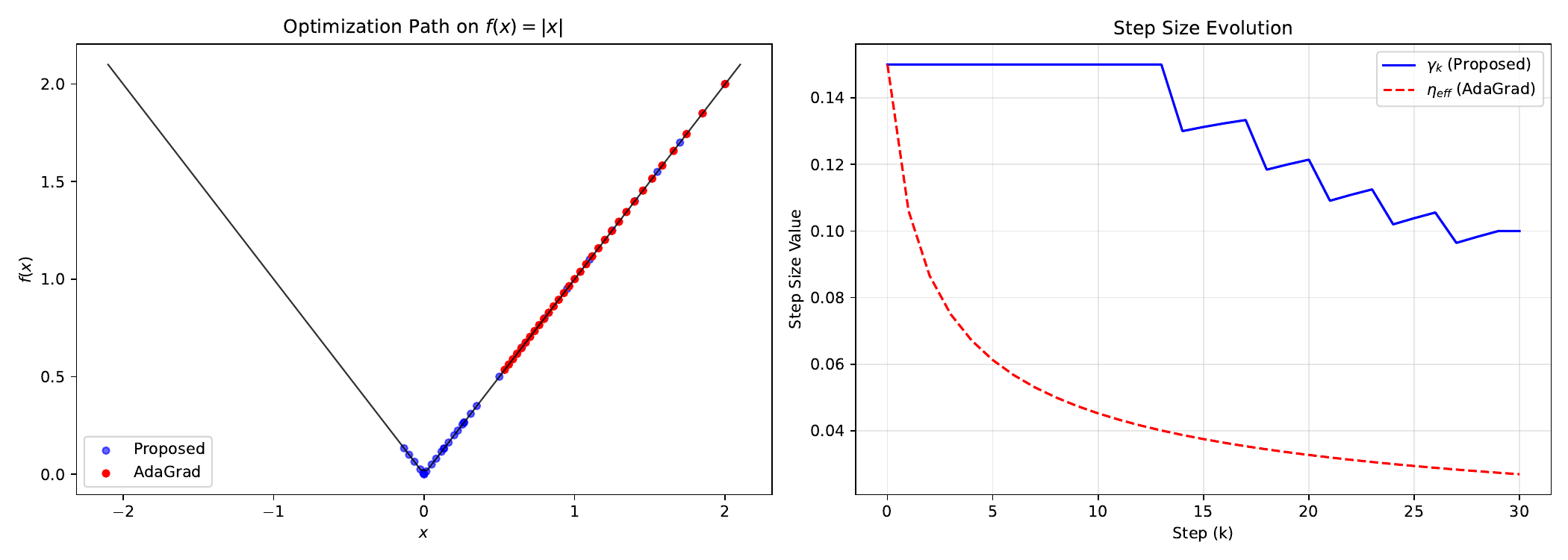}
    \caption{The FlexGrad optimizer maintains a larger step size when gradients are correlated.}
    \label{fig:flex-example}
\end{figure}

\section{Related Work}\label{sec:literature}

Since the development of the AdaGrad method \citep{adagrad}, a number of directions have been explored under the broad scope of adaptive optimization. These methods differ in the degree of adaptivity, the quantity adapted to, and the practicality of the approach. 

Coin-betting approaches \citep{orabona2017training} attempt to adapt to the initial distance-to-solution $D$, but the resulting methods lose adaptivity to the gradient norm sequence, becoming sensitive to outliers and significantly varying gradient norm sequences, which hurts their practical performance. Coin-betting techniques have intriguing theoretical properties. The resulting step sizes are non-monotonic, and can grow and shrink more rapidly than other known adaptive methods. Recent work introduces exponential-moving-averages into the formulation, and makes the methods more modular, resulting in better practical performance \citep{cutkosky2023mechanic}.

Since fully non-monotonic estimation of $D$ can be unstable, another family of approaches uses monotonically increasing estimates of $D$, coupled with standard gradient-norm-adaptivity techniques. D-Adaptation \citep{defazio2023dadapt} maintains a lower bound on $D$ that is updated at every step, asymptotically approaching within a constant factor of the true $D$. The weighted Prodigy \citep{mishchenko2023prodigy} variant of D-Adaptation has demonstrated state-of-the-art learning rate adaptation across a large number of problem classes. However, its estimate is only for the global learning rate, and no coordinate-wise version is known to work. Other methods in this class include DoG \citep{ivgi2023dog} and DoWG \citep{khaled2023dowg}, which estimate $D$ using a maximum of the iterate distances seen so far $\displaystyle \max_{i<k}\|x_i - x_0\|$.

Coin-betting and D-Adaptation methods share a property: their adaptivity makes use of inner products of the current gradient with sums of past gradients. Many methods have been developed that try to make explicit use of such quantities to define a step size. Historically, the inner-product of two consecutive gradients was the main approach, defining the class of \emph{hyper-gradient methods}. The name refers to the fact that for a step like $x_{k+1}=x_k - \eta_k g_k$ the gradient with respect to the step size $\eta_k$ is given by this inner-product \citep{mirzoakhmedov1983adaptive}:
\[
\frac{\partial f(x_{k+1})}{\partial\eta_{k}}=\left\langle g_{k+1},g_{k}\right\rangle.
\]
Recent explorations of hyper-gradient methods in the context of machine learning include \citet{baydin2017online} and \citet{mahmood2012tuning}. Like the methods that adapt to $D$, these approaches do not have practical coordinate-wise variants.

The GradaGrad method \citep{defazio2022grad} incorporates this gradient product into the denominator of an AdaGrad-like step size, yielding a degree of non-monotonic flexibility:
\[
\frac{D}{\sqrt{\sum_{t=0}^{k}\left(g_{i,t}^{2}-\rho g_{i,j}g_{i,t-1}\right)}},
\]
for coordinate $i$. This approach potentially yields negative values under the square-root, and a number of tricks are needed to handle small or potentially negative values to prevent blow-up. Our approach avoids this issue by not using a square-root in the denominator, and introducing non-monontone behavior by an increasing numerator instead. The EMA variant of our approach is related to the BCOS method \cite{jiang2025bcos}. Their proposed method is developed as a simplified variant of the update we suggest.

\section{Motivation}\label{sec:motivation}

Our proposed approach arises from the goal of adapting D-Adaptation's global learning rate estimation approach to the coordinate-wise case. D-Adaptation uses the inner product of the current gradient and the sum of past gradients $s_{k+1} := \sum_{j=0}^{k} g_j$ to estimate $D$:
\[
\hat{d}_{k+1}=\frac{\sum_{i=0}^{k}d_{i}\gamma_{i}\left\langle g_{i},s_{i}\right\rangle }{\left\Vert s_{k+1}\right\Vert }
\]
where $d_i$ is the maximum of the previous estimates at step $i$. The denominator of the step size (here folded into $\gamma$) is taken to be the AdaGrad denominator.

This $D$ estimate can grow exponentially fast, which is beneficial for global learning rate estimation, but potentially problematic when applied coordinate-wise, as such rapid changes can be fragile. Additionally, D-Adaptation's theory relies crucially on the use of function-value gap positivity ($f(x_k) - f_* > 0$) which has no obvious coordinate-wise analog, and the method fails almost immediately when applied coordinate-wise.

The FlexGrad optimizer's global version uses the same inner product not to estimate $D$, but to modify the gradient-squared accumulation performed in AdaGrad to bring about a more flexible adaptation. The $g$-to-$s$ inner product arises if we consider the expansion of the numerator:
\begin{align*}
\gamma_{k+1}
&=\frac{D\left\Vert \sum_{i=0}^{k}g_{i}\right\Vert}{\sum_{i=0}^{k}\left\Vert g_{i}\right\Vert^{2}}
=\frac{D\sqrt{\left\Vert \sum_{i=0}^{k-1}g_{i}+g_{k}\right\Vert^{2}}}{\sum_{i=0}^{k}\left\Vert g_{i}\right\Vert^{2}}\\
&=\frac{D\sqrt{\left\Vert \sum_{i=0}^{k-1}g_{i}\right\Vert^{2}
+2\left\langle s_{k},g_{k}\right\rangle+\left\Vert g_{k}\right\Vert^{2}}}{\sum_{i=0}^{k}\left\Vert g_{i}\right\Vert^{2}}\\
&\dots\\
&=\frac{D\sqrt{2\sum_{i=0}^{k}\left\langle s_i,g_i\right\rangle+\sum_{i=0}^{k}\|g_i\|^2}}{\sum_{i=0}^{k}\left\Vert g_{i}\right\Vert^{2}}.
\end{align*}
When $\sum_{i=0}^{k}\left\langle s_{i},g_{i}\right\rangle=0$ the step size collapses to the AdaGrad step, so the sign of this inner product quantity determines whether the step size is smaller or larger than AdaGrad. The step size is scaled up when this inner product is positive, and down when it is negative.

The D-Adaptation theory directly motivates this choice of step size if we consider the simplified setting of fixed-in-advance step sizes. Suppose we use a step size of the form
\[
\gamma_{\text{FlexGrad}}=\frac{D\|s_{k+1}\|}{\sum_{i=0}^{k}\|g_{i}\|^{2}}.
\]
D-Adaptation theory arises from the following key inequality \citep[Lemma 6]{defazio2023dadapt}:
\[
\sum_{i=0}^{k}(f(x_{i})-f_{\star})\le D\|s_{k+1}\|-\frac{\gamma}{2}\|s_{k+1}\|^{2}+\frac{\gamma}{2}\sum_{i=0}^{k}\|g_{i}\|^{2}.
\]
Substituting in $\gamma_{\text{FlexGrad}}$ gives:
\begin{equation}\label{dadapflexgad}
\sum_{i=0}^{k}(f(x_{i})-f_{\star})\le\frac{3}{2}D\|s_{k+1}\|-\frac{1}{2}D\frac{\|s_{k+1}\|^{3}}{\sum_{i=1}^{k}\|g_{i}\|^{2}}.
\end{equation}
We can maximize the right-hand size over $u=\|s_{k+1}\|\ge0$
to find the worse-case upper bound. So let us take the derivative with
respect to $u$ and set it equal to zero:
\[
\frac{\partial}{\partial u}\left(\frac{3}{2}Du-\frac{1}{2}D\frac{u^{3}}{\sum_{i=1}^{k}\|g_{i}\|^{2}}\right)=\frac{3}{2}D-\frac{3}{2}D\frac{u^{2}}{\sum_{i=1}^{k}\|g_{i}\|^{2}}=0
\]
\[
\therefore u=\sqrt{\sum_{i=1}^{k}\|g_{i}\|^{2}}.
\]
Plugging this in \eqref{dadapflexgad}  gives the final convergence rate bound:
\[
\sum_{i=0}^{k}(f(x_{i})-f_{\star})\le D\sqrt{\sum_{i=0}^{k}\|g_{i}\|^{2}}.
\]
This bound exactly matches the convergence rate of adaptive methods for non-smooth convex optimization. Note that this worst-case performance happens when $u=\|s_{k+1}\|= \sqrt{\sum_{i=1}^{k}\|g_{i}\|^{2}}$, which is precisely when the 
$g$-to-$s$ inner product term is $0$. Outside of the worst-case, the convergence rate bound can potentially be much better than the AdaGrad bound. In Section~\ref{sec:convergence-theory} we establish full convergence theory when the step size is not fixed in advance, for both the global variant and coordinate-wise versions of the method.

\section{Convergence Theory}
\label{sec:convergence-theory}

We now turn to the theoretical analysis of FlexGrad in the convex nonsmooth setting. Our goal is to show that, despite its more flexible and potentially non-monotone scaling, the method retains the same worst-case order as standard adaptive subgradient methods. At the same time, the bounds make explicit how FlexGrad departs from AdaGrad by responding to cumulative gradient alignment through the quantity \(s_{k+1}\).

We work in a weighted framework with coefficients \((\lambda_k)_{k\ge0}\), which includes the basic method as the special case \(\lambda_k\equiv 1\). This formulation is convenient both for the analysis and for later comparisons with other adaptive schemes.
Our FlexGrad optimizer is given in Algorithm \ref{algo1}
below.

\begin{algorithm}[h!]
\caption{FlexGrad}
\label{algo1}
\begin{flushleft}
\textbf{Input:} $x_0\in\mathbb R^d$, \, $(\lambda_k)_{k\ge0}$ with $\lambda_k\ge0$ and $\lambda_0>0$, \, $\widehat D \ge \|x_0-x_*\|.$ \\
$s_0=0$, \, $g_0\in\partial f(x_0)$ \\
If $g_0=0$, exit with $\hat x_n=x_0$ \\
\textbf{for} $k=0$ \textbf{to} $n-1$ \textbf{do} \\
\quad $g_k\in\partial f(x_k)$ \\
\quad $s_{k+1}=s_k+\lambda_k g_k$ \\
\quad $\displaystyle \gamma_{k+1}=\frac{\widehat D\|s_{k+1}\|}{\sum_{i=0}^{k}\lambda_i^2\|g_i\|^2}$ \\
\quad $x_{k+1}=x_0-\gamma_{k+1}s_{k+1}$ \\
\textbf{end for} \\
Return $\displaystyle \hat x_n=\frac{1}{\sum_{k=0}^n \lambda_k}\sum_{k=0}^n \lambda_k x_k$
\end{flushleft}
\end{algorithm}

Our first result shows that FlexGrad satisfies an adaptive convergence guarantee of the same overall form as those known for AdaGrad-type methods, with an additional factor that quantifies the irregularity of the gradient sequence under the FlexGrad scaling.

\begin{theorem}\label{thm:weighted}
In Algorithm \ref{algo1}, for every \(n\ge1\),
\[
f(\hat x_n)-f_*
\le
14\widehat D(1+\rho_n)
\frac{\sqrt{\sum_{k=0}^{n}\lambda_k^2 \|g_k\|^2}}{\sum_{k=0}^n \lambda_k},
\]
where
\[
\rho_n:=
\max_{1\le k\le n}
\frac{\lambda_k^2\|g_k\|^2}{\sum_{i=0}^{k-1}\lambda_i^2\|g_i\|^2}.
\]
\end{theorem}

The following lemma is the key technical ingredient in the proof of Theorem~\ref{thm:weighted}.

\begin{lemma}\label{keybound2}
Let \((u_k)_{k\ge1}\) and \((v_k)_{k\ge1}\) be sequences such that \(v_1>0\) and
\[
u_k\ge 0,\qquad d_k:=u_{k+1}-u_k,\qquad \Delta_k:=v_{k+1}-v_k,
\]
with
\[
d_k^2\le \Delta_k \qquad (k\ge1).
\]
Then
\begin{align*}
u_{n+1}+ \frac{1}{2} \sum_{k=1}^n
\frac{u_k}{v_k}\left(u_k^2-u_{k+1}^2+\Delta_k\right)
&\le
\frac{2}{3} \sqrt{v_{n+1}}
+\frac13\,\frac{u_1^3}{v_1}
+2\sum_{k=1}^n \frac{\Delta_k}{\sqrt{v_k}}
+\frac{13}{3} \sum_{k=1}^n \frac{\Delta_k^{3/2}}{v_k}.
\end{align*}
\end{lemma}

In our setting, the left-hand side of the inequality above is proportional to the convexity upper bound for
\(
\sum_{k=0}^{n}\lambda_k\bigl(f(x_k)-f_*\bigr)
\)
after the substitutions
\(
u_k:=\|s_k\|,
\,
v_k:=\sum_{i=0}^{k-1}\lambda_i^2\|g_i\|^2.
\)
See the proof of Theorem~\ref{thm:weighted} for the corresponding derivation.

\begin{remark}
The conditions of Lemma~\ref{keybound2} are invariant under adding a nonnegative constant to the sequence \((v_k)\). More precisely, for any \(c\ge 0\), define
\[
\tilde v_k:=v_k+c \qquad (k\ge 1).
\]
Then
\[
(u_{k+1}-u_k)^2\le \tilde v_{k+1}-\tilde v_k=v_{k+1}-v_k,
\]
so the assumptions remain unchanged, except that \(v_1\) is replaced by \(\tilde v_1=v_1+c\). In particular, in Algorithm \ref{algo1} one may define
\[
\gamma_{k+1}=\frac{\widehat D\|s_{k+1}\|}{c+\sum_{i=0}^{k}\lambda_i^2\|g_i\|^2},
\qquad k\ge 0,
\]
and Theorem~\ref{thm:weighted} gives
\[
f(\hat x_n)-f_*
\le
14\widehat D(1+\rho_n)
\frac{\sqrt{c+\sum_{k=0}^{n}\lambda_k^2 \|g_k\|^2}}{\sum_{k=0}^n \lambda_k}
\]
for
\[
\rho_n:=
\max_{1\le k\le n} \,
\frac{\lambda_k^2\|g_k\|^2}{c+\sum_{i=0}^{k-1}\lambda_i^2\|g_i\|^2}.
\]
\end{remark}

The parameter \(c\) plays the same role as the \(\varepsilon\)-stabilization commonly used in AdaGrad, RMSProp, and Adam-type methods. Its presence does not alter the qualitative form of the guarantee, but improves robustness when the accumulated squared-gradient term is still small. In particular, the regularized formulation makes clear that FlexGrad can be implemented with the same kind of denominator smoothing that is standard in practical adaptive optimizers.

We next specialize the general bound to polynomially growing weights, under the standard additional assumption that \(f\) is Lipschitz continuous. This yields a more explicit asymptotic statement and clarifies how the regularity factor behaves for a common class of schedules.

\begin{theorem}\label{thm:weighted_poly}
Assume that, in Algorithm~\ref{algo1}, the weights satisfy
\[
\lambda_k \asymp k^\alpha \qquad (k\ge 1)
\]
for some \(\alpha>-\tfrac12\). Define
\[
\delta_n:=
\begin{cases}
1, & \text{if} \ \ \alpha\le 0,\\[4pt]
\displaystyle \min_{1\le k\le n}\sqrt{k}\,\|g_k\|, & \text{if} \ \ \alpha>0.
\end{cases}
\]
Then
\[
f(\hat x_n)-f_*
=
\mathcal O\!\left(\frac1{\delta_n^2 \, \sqrt n}\right).
\]
\end{theorem}

The conclusion shows that, within the polynomial-weight regime \(\lambda_k\asymp k^\alpha\), FlexGrad still achieves the standard \(\mathcal O(n^{-1/2})\) rate unless unbounded weights are accompanied by excessively fast gradient decay. The quantity \(\delta_n\) captures exactly this issue. For \(\alpha\le 0\), no further condition is required; for \(\alpha>0\), it suffices that \(\sqrt{k}\,\|g_k\|\) stay bounded away from zero. In this sense, Theorem~\ref{thm:weighted_poly} recovers the usual \(\mathcal O(n^{-1/2})\) rate throughout the natural nondegenerate regimes.

One of the main motivations for FlexGrad is that the same root-free scaling extends naturally to the coordinate-wise setting. This is especially important from a theoretical perspective, since many globally adaptive methods do not admit equally clean coordinate-wise analogues, and many practical coordinate-wise methods are backed primarily by smooth or stochastic analyses rather than nonsmooth convex guarantees of this form.

\begin{algorithm}[h!]
\caption{Coordinate-wise FlexGrad}
\label{algo1_coord}
\begin{flushleft}
\textbf{Input:} $x_0\in\mathbb R^d$, \, $(\lambda_k)_{k\ge0}$ with $\lambda_k\ge0$ and $\lambda_0>0$, \, $\vec D\in\mathbb R^d$ such that $\vec D\ge |x_0-x_*|$ \\
$s_0=0$, \, $g_0\in\partial f(x_0)$, \, $v_0=0$\\
\textbf{for} $k=0$ \textbf{to} $n-1$ \textbf{do} \\
\quad $g_k\in\partial f(x_k)$ \\
\quad $s_{k+1}=s_k+\lambda_k g_k$ \\
\quad $v_{k+1}=v_k+\lambda_k^2 g_k \odot g_k$ \\
\quad $\displaystyle \gamma_{k+1}= \vec D\odot |s_{k+1}| \oslash v_{k+1} \qquad \text{(using convention } 0/0=0)$ \\
\quad $x_{k+1}=x_0-\gamma_{k+1} \odot s_{k+1}$ \\
\textbf{end for} \\
Return $\displaystyle \hat x_n=\frac{1}{\sum_{k=0}^n \lambda_k}\sum_{k=0}^n \lambda_k x_k$
\end{flushleft}
\end{algorithm}

The next theorem shows that the coordinate-wise variant 
of FlexGrad  given in Algorithm \ref{algo1_coord}
below admits essentially the same convergence theory as the global variant.

\begin{theorem}\label{thm:weighted_coord}
In Algorithm \ref{algo1_coord}, for every \(n\ge1\),
\[
f(\hat x_n)-f_*
\le
14\|\vec D\|(1+\rho_n)
\frac{\sqrt{\sum_{k=0}^{n}\lambda_k^2 \|g_k\|^2}}{\sum_{k=0}^n \lambda_k},
\]
where
\[
\rho_n:=
\max_{\substack{1\le k\le n\\ 1\le j\le d}} \,
\frac{\lambda_k^2 g_{k,j}^2}{\sum_{i=0}^{k-1}\lambda_i^2 g_{i,j}^2}.
\]
\end{theorem}

The structure of the bound is the same as in the global case, with \(\|\vec D\|\) replacing \(\widehat D\) and a coordinate-wise version of the irregularity factor \(\rho_n\). Thus the root-free scaling remains theoretically stable even when applied independently across coordinates. This is one of the main conceptual differences between FlexGrad and several other adaptive mechanisms based on cumulative gradient information: the same basic principle extends directly to a coordinate-wise rule without requiring square-root corrections, sign restrictions, or auxiliary clipping devices in the analysis. As in the global case, one can combine Theorem~\ref{thm:weighted_coord} with polynomially growing weights \(\lambda_k\asymp k^\alpha\) to obtain the corresponding \(\mathcal O(n^{-1/2})\)-type asymptotic estimate under the analogous nondegeneracy condition on the coordinate gradients.

Taken together, these results show that FlexGrad preserves the standard adaptive guarantees associated with AdaGrad-type methods, while incorporating cumulative-gradient information in a manner more reminiscent of D-Adaptation and coin-betting ideas. The resulting theory makes the comparison precise: in the worst case, FlexGrad behaves like a classical adaptive method; in more favorable regimes, it can respond more aggressively to persistent gradient signal; and this same principle extends cleanly from the global setting to the coordinate-wise one.

\section{Exponential-Moving-Average Variants}\label{sec:ema}

Exponential-Moving-Average variants of AdaGrad such as Adam \citep{kingma2014adam} and AdamW \citep{loshchilov2018decoupled} are overwhelmingly used in practice due to their superior practical performance over vanilla AdaGrad. Our approach can be adapted to use EMAs for the $s$ and $v$ sequences as follows (where each update is applied per-coordinate): 
\begin{align}
s_{k+1} & =\beta s_{k}+(1-\beta)g_{k}, \nonumber\\
v_{k+1} & =\beta v_{k}+\left(1-\beta\right)g_{k}^{2}. \label{eq:ema}
\end{align}
When using a mirror-descent-style step for consistency with Adam, the update becomes:
\[
x_{k+1}=x_{k}-\gamma\frac{\left|s_{k+1}\right|}{v_{k+1}}m_{k+1}.
\]
Here, $m_{k+1}$ is the momentum buffer also used by Adam. This approach is closely related to Adam as we may rewrite the step as:
\[
\frac{\left|s_{k+1}\right|}{v_{k+1}}m_{k+1} = \left(\frac{\left|s_{k+1}\right|}{\sqrt{v_{k+1}}}\right)\cdot\frac{m_{k+1}}{\sqrt{v_{k+1}}}.
\]
We have rewritten the step to match the Adam update $m_{k+1}/\sqrt{v_{k+1}}$ modulated by a quantity related to the signal-to-noise ratio $|\mu|/\sigma$ of the gradient. This modulator is non-expansive, as we show in Lemma~\ref{lem:ema_ineq}.
\begin{lemma}
\label{lem:ema_ineq} The EMA-based updates in Equation~\ref{eq:ema}, when initialized with $s_0=0$ and $v_0=0$, satisfy:
$$|s_{k+1}| \leq \sqrt{v_{k+1}}.$$
\end{lemma}
\begin{proof}
We proceed by induction. The base case is trivial, so assume by induction that $|s_{k}| \leq \sqrt{v_{k}}$. Consider $s_{k+1}^2$. Since our weights $\beta$ and $(1-\beta)$ are positive and sum to 1, we can apply Jensen's inequality:
\[
\left(\beta s_k + (1-\beta)g_k\right)^2 \leq \beta s_k^2 + (1-\beta)g_k^2
\]
Now, apply our inductive hypothesis ($s_k^2 \leq v_k$) to the right side of the inequality:$$\beta s_k^2 + (1-\beta)g_k^2 \leq \beta v_k + (1-\beta)g_k^2$$
The right side of this new inequality is exactly the update rule for $v_{k+1}$. Therefore:
$$s_{k+1}^2 \leq \beta v_k + (1-\beta)g_k^2 = v_{k+1}.$$
Taking the square-root gives the result.
\end{proof}
This EMA version captures the same desirable properties of the non-EMA version: it takes larger steps in directions with high gradient signal, and smaller steps in directions dominated by noise. We use it as a practical Adam-style variant in the experiments, while the convergence theory in Section~\ref{sec:convergence-theory} applies to the cumulative non-EMA versions of FlexGrad. 

\section{Experiments}\label{sec:numsim}

\subsection{Convex Problems}

We validated FlexGrad on standard logistic-regression benchmarks from the LIBSVM repository, available at
\url{https://www.csie.ntu.edu.tw/~cjlin/libsvmtools/datasets/}.
Our code
is available
on Github at
\url{https://github.com/vguigues/Adaptive-Optimizer}.
We compare FlexGrad, coordinate-wise FlexGrad, EMA, and coordinate-wise EMA  against AdaGrad, Adam, BCOS-g, and GradaGrad. For each dataset and method, we performed an independent hyperparameter search and selected the configuration achieving the best accuracy. The grid was
\(\{10^{-3},10^{-2},10^{-1},1,10\}\)
for the parameter \(D\) in FlexGrad, the learning rate in AdaGrad, parameter
$\alpha_t=\alpha$ for BCOS-g,
and the parameter \(\gamma\) in the EMA variants. For the latter, we additionally used
\(\beta\in\{0.9,0.96,0.98\}\),
while for Adam we used
\(\beta_1,\beta_2\in\{0.9,0.99,0.999\}\) and
\(\alpha\in\{10^{-3},10^{-2},10^{-1},1,10\}\).
For BCOS-g, we also take 
\(\beta\in\{0.9,0.96,0.98\}\).
For GradaGrad, we take
$\gamma_0 \in \{10^{-3},10^{-2},10^{-1},1,10\}$, 
$\rho \in \{1,2,4\}$, and $r \in \{0.25,0.5, 1\}$.

All methods were implemented in Julia and run on a machine with an AMD Ryzen 5 2600 CPU and 16 GB of RAM. Each experiment used 50 epochs and regularization parameter \(10^{-4}\). For binary classification problems, we minimized the \(\ell_1\)-regularized logistic objective; for multiclass problems, we minimized the cross-entropy loss.

After hyperparameter selection, each method was run five times from independently generated initial points \(x_0\). 

In a first set of experiments, we compare variant BCOS-g of BCOS,
GradaGrad, Flexgrad coordinate wise, Adagrad,
EMA coordinate-wise, and ADAM.
For these methods, Figure~\ref{empprobmodel26} reports the mean accuracy along the iterations while Figure~\ref{empprobmodels6} reports the corresponding standard deviations. Figure~\ref{empprobmodelo6} reports the objective value along the iterations.

In this experiment, methods BCOS-g and GradaGrad
in general do not perform as well
as the remaining methods and in a second set of experiments
we compare our methods Flexgrad, Flexgrad coordinate-wise, 
EMA, and EMA coordinate-wise
against Adagrad and ADAM.

For this second set of experiments, Figure~\ref{empprobmodel2}  reports the mean accuracy along the iterations while Figure~ \ref{empprobmodels} reports the corresponding standard deviations. Figure~\ref{empprobmodelo} reports the objective value along the iterations.

Except on \texttt{vowel}, all methods achieve high accuracy after 50 iterations. The fastest methods to reach the best observed accuracy vary across datasets: coordinate-wise FlexGrad is fastest on several binary benchmarks, AdaGrad is fastest on others, and the EMA variants perform best on some multiclass problems. On \texttt{dna} and \texttt{usps}, all methods obtain similar accuracies.

The objective-value curves show that, on the binary benchmarks, all optimizers reach near-optimal values within fewer than 10 iterations. On the remaining datasets, the relative ordering varies: EMA and Adam are fastest on \texttt{wine}, EMA is fastest on \texttt{vehicle}, coordinate-wise EMA is best on \texttt{letter}, and the methods perform similarly on several other tasks. The standard deviations generally decrease along the iterations and are small at the final iteration for all methods.

These results indicate that the benefits of the FlexGrad variants are not restricted to nonconvex settings and also extend to standard convex benchmarks.

\subsection{LLM Pretraining}
We tested the EMA variant of FlexGrad for language model pretraining. We used a LLama 3 architecture 1B parameter model with 16 layers and a layer width of 2048, and pretrained on the FineWeb-EDU dataset for 19,074 steps at a batch-size of 1M, which equates to approximately 20B training tokens. This gives a tokens-per-parameter of 20, following the Chinchilla-ratio which is quickly becoming standard for comparisons in the literature. For both the baseline AdamW optimizer and the FlexGrad optimizer, we tuned both learning rate, weight-decay and $1-\beta_2$ on a log-spaced grid. Figure~\ref{fig:llm-pretrain-1b} shows that the FlexGrad optimizer is competitive with AdamW, giving essentially the same loss throughout training. This demonstrates that it can be used as a drop-in replacement for Adam.

\begin{figure}
    \centering
    \includegraphics[scale=0.5]{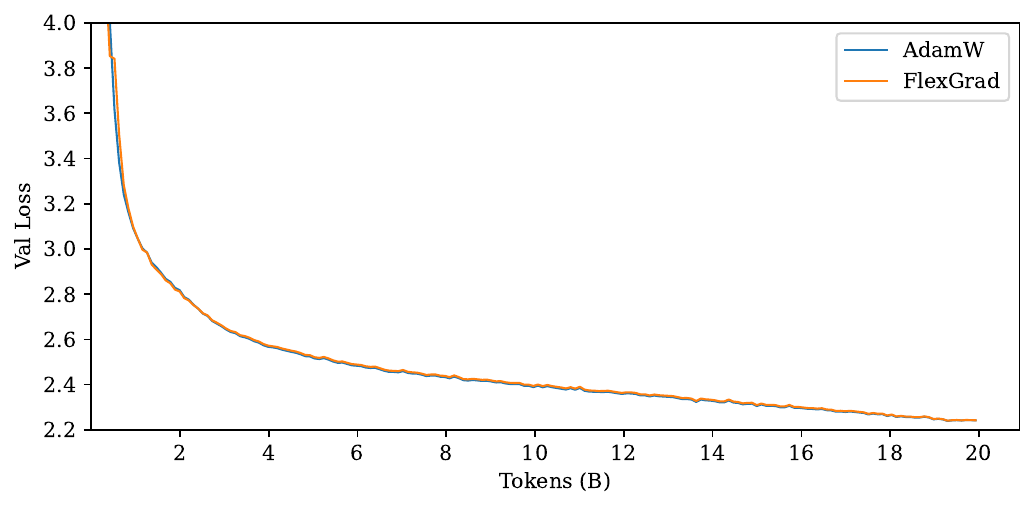}
    \caption{1B parameter LLM Pretraining}
    \label{fig:llm-pretrain-1b}
\end{figure}

\subsection{Learning Rate Robustness}
The EMA variant of the FlexGrad optimizer's biggest advantage is it's increased stability and robustness to larger step sizes. To investigate this we ran a learning rate sweep for AdamW and the FlexGrad optimizer on a 120M parameter model, trained for 12.5B tokens, with 0.25B batch-size. Figure~\ref{fig:lr_sweep} shows that FlexGrad can converge with step sizes much larger than AdamW without divergence. 

\begin{figure}
    \centering
    \includegraphics[scale=0.5]{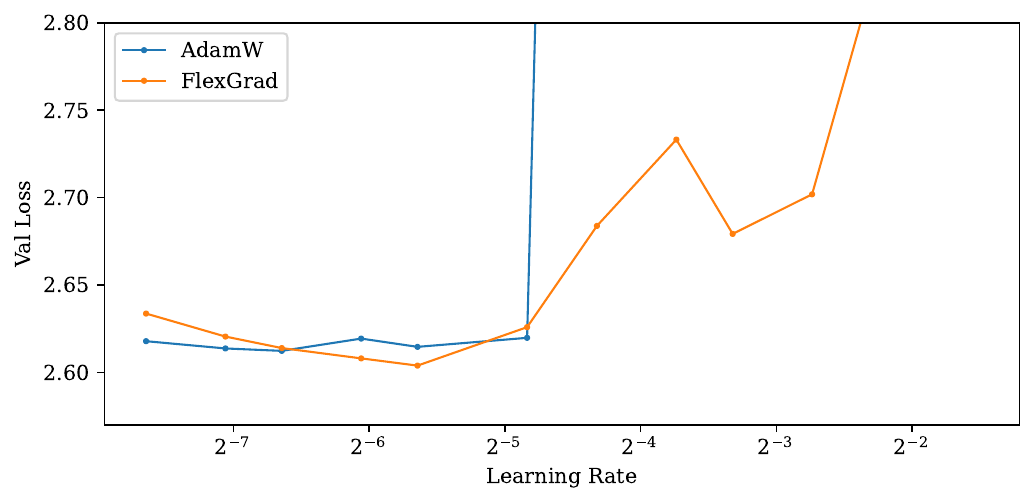}
    \caption{EMA optimizer converges stably with larger learning rates than AdamW}
    \label{fig:lr_sweep}
\end{figure}

\section{Conclusion}

We have presented FlexGrad, a root-free adaptive method that uses cumulative gradient information to modulate the stepsize more directly than classical square-root-normalized rules. The main limitations of the present work are the scope of the theory and the current empirical coverage. Our convergence guarantees apply to the non-EMA FlexGrad variants in the convex Lipschitz nonsmooth setting, while the EMA variants used in large-scale neural-network experiments are currently supported by stability observations and empirical evidence rather than a full convergence theory. Empirically, FlexGrad is competitive across the settings considered here, but its practical advantage is expected to be most visible when gradients exhibit persistent directional structure; in highly noisy regimes, the method is designed to behave similarly to AdaGrad-type normalization rather than to yield a large improvement. Its appeal is not only algorithmic simplicity, but also the fact that the same principle captures two desirable behaviors: conservative AdaGrad-type scaling in noisy or unstable regimes, and larger, correlation-sensitive updates when gradient directions remain coherent over time. Our analysis shows that this added responsiveness is compatible with the standard worst-case theory for nonsmooth Lipschitz convex optimization, both in the global and coordinate-wise settings.

\clearpage
\bibliography{neurips_2026}
\bibliographystyle{apalike}
\newpage


\appendix

\section{Appendix section}

The next lemma controls the two auxiliary sums that remain after the main scalar reduction. Both are bounded on the natural \(\sqrt{v_{n+1}}\) scale, up to the irregularity factor \(\rho_n\).

\begin{lemma}\label{sumbounds}
Let \(v_1>0\). Define the sequences \((\Delta_k)_{k\ge1}\), with \(\Delta_k\ge0\), and
\[
v_{k+1}=v_k+\Delta_k \qquad (k\ge1).
\]
For \(n\ge1\), let
\[
\rho_n:=\max_{1\le k\le n}\frac{\Delta_k}{v_k}.
\]
Then
\bi
\item[(i)]
\[
\sum_{k=1}^n \frac{\Delta_k}{\sqrt{v_k}}
\le
\left(1+\sqrt{1+\rho_n}\right)(\sqrt{v_{n+1}}-\sqrt{v_1}).
\]

\item[(ii)]
\[
\sum_{k=1}^n \frac{\Delta_k^{3/2}}{v_k}
\le
\sqrt{\rho_n}\,\left(1+\sqrt{1+\rho_n}\right)(\sqrt{v_{n+1}}-\sqrt{v_1}).
\]
\ei
\end{lemma}

\begin{proof}
\bi
\item[(i)]
Let \(R_k:=\sqrt{v_k}\). Since \(v_{k+1}-v_k=\Delta_k\),
\[
\frac{\Delta_k}{\sqrt{v_k}}
=
\frac{R_{k+1}^2-R_k^2}{R_k}
=
(R_{k+1}-R_k)\left(\frac{R_{k+1}}{R_k}+1\right)
=
(R_{k+1}-R_k)\left(\sqrt{1+\frac{\Delta_k}{v_k}}+1\right).
\]
Because \(\Delta_k/v_k\le \rho_n\),
\[
\frac{\Delta_k}{\sqrt{v_k}}
\le
(1+\sqrt{1+\rho_n})(R_{k+1}-R_k).
\]
Summing over \(k=1,\ldots,n\) yields
\[
\sum_{k=1}^n \frac{\Delta_k}{\sqrt{v_k}}
\le
(1+\sqrt{1+\rho_n})
\sum_{k=1}^n (R_{k+1}-R_k)=
(1+\sqrt{1+\rho_n})(\sqrt{v_{n+1}}-\sqrt{v_1}).
\]
\item[(ii)]
Using \(\Delta_k\le \rho_n v_k\), we get
\[
\frac{\Delta_k^{3/2}}{v_k}
=
\frac{\Delta_k}{\sqrt{v_k}}\sqrt{\frac{\Delta_k}{v_k}}
\le
\sqrt{\rho_n}\,\frac{\Delta_k}{\sqrt{v_k}}.
\]
Now it suffices to use item (ii) above.
\ei
\end{proof}

The next lemma controls the main nonlinear contribution appearing in the scalar estimate. It reduces this term to the two auxiliary sums estimated above.

\begin{lemma}\label{keybound1}
Let \((u_k)_{k\ge1}\) and \((v_k)_{k\ge1}\) be sequences such that \(v_1>0\) and
\[
u_k\ge 0,\qquad d_k:=u_{k+1}-u_k,\qquad \Delta_k:=v_{k+1}-v_k,
\]
with
\[
d_k^2\le \Delta_k \qquad (k\ge1).
\]
Then, for every \(n\ge1\),
\[
\sum_{k=1}^n \left[\frac{3u_k\Delta_k}{v_k}
-\frac{u_{k+1}^3\Delta_k}{v_kv_{k+1}}\right]
\le
6\sum_{k=1}^n \frac{\Delta_k}{\sqrt{v_k}}
+
12\sum_{k=1}^n \frac{\Delta_k^{3/2}}{v_k}.
\]
\end{lemma}

\begin{proof}
Set
\[
A_k:=\frac{3u_k\Delta_k}{v_k}
-\frac{u_{k+1}^3\Delta_k}{v_kv_{k+1}}.
\]
Split the indices into
\[
I:=\{k:\ u_k\le 2\sqrt{v_k}+4\sqrt{\Delta_k}\},
\qquad
J:=\{k:\ u_k>2\sqrt{v_k}+4\sqrt{\Delta_k}\}.
\]
We first show that \(A_k\le0\) for every \(k\in J\).
Since \(d_k^2\le \Delta_k\), we have \(|d_k|\le \sqrt{\Delta_k}\), hence
\[
u_{k+1}\ge u_k-\sqrt{\Delta_k}>2\sqrt{v_k}+3\sqrt{\Delta_k}.
\]
Using \(\sqrt{v_k+\Delta_k}\le \sqrt{v_k}+\sqrt{\Delta_k}\), we get
\[
u_{k+1}>2\sqrt{v_k}+2\sqrt{\Delta_k}\ge 2\sqrt{v_k+\Delta_k}=2\sqrt{v_{k+1}}.
\]
Thus
\[
\frac{u_{k+1}^2}{v_{k+1}}>4.
\]
Also, since \(k\in J\),
\[
u_k>4\sqrt{\Delta_k},
\]
so
\[
u_{k+1}\ge u_k-\sqrt{\Delta_k}\ge \frac34\,u_k.
\]
Combining the last two inequalities,
\[
\frac{u_{k+1}^3}{v_{k+1}}
=
u_{k+1}\frac{u_{k+1}^2}{v_{k+1}}
>
4u_{k+1}
\ge
3u_k.
\]
Hence
\[
A_k
=
\frac{\Delta_k}{v_k}\left(3u_k-\frac{u_{k+1}^3}{v_{k+1}}\right)
\le 0.
\]
Now let \(k\in I\). Then
\[
A_k\le \frac{3u_k\Delta_k}{v_k}
\le
3\,\frac{(2\sqrt{v_k}+4\sqrt{\Delta_k})\Delta_k}{v_k}
=
6\frac{\Delta_k}{\sqrt{v_k}}
+
12\frac{\Delta_k^{3/2}}{v_k}.
\]
Since \(A_k\le0\) for every \(k\in J\), we conclude that
\[
\sum_{k=1}^n A_k
=
\sum_{\substack{k\le n\\k\in I}} A_k
+
\sum_{\substack{k\le n\\k\in J}} A_k
\le
\sum_{\substack{k\le n\\k\in I}} A_k
\le
6\sum_{k=1}^n \frac{\Delta_k}{\sqrt{v_k}}
+
12\sum_{k=1}^n \frac{\Delta_k^{3/2}}{v_k}.
\]
\end{proof}

The following lemma is the core scalar estimate behind the proof of Theorem~\ref{thm:weighted}. It reduces the nonlinear contribution of the FlexGrad scaling to standard adaptive quantities.

\begin{lemma*}[Lemma \ref{keybound2}]
Let \((u_k)_{k\ge1}\) and \((v_k)_{k\ge1}\) be sequences such that \(v_1>0\) and
\[
u_k\ge 0,\qquad d_k:=u_{k+1}-u_k,\qquad \Delta_k:=v_{k+1}-v_k,
\]
with
\[
d_k^2\le \Delta_k \qquad (k\ge1).
\]
Then
\begin{align*}
u_{n+1}+ \frac{1}{2} \sum_{k=1}^n
\frac{u_k}{v_k}\left(u_k^2-u_{k+1}^2+\Delta_k\right)
&\le
\frac{2}{3} \sqrt{v_{n+1}}
+\frac13\,\frac{u_1^3}{v_1}
+2\sum_{k=1}^n \frac{\Delta_k}{\sqrt{v_k}}
+\frac{13}{3} \sum_{k=1}^n \frac{\Delta_k^{3/2}}{v_k}.
\end{align*}
\end{lemma*}

\begin{proof}
Set
\[
t_k:=\frac{u_k}{v_k}\left(u_k^2-u_{k+1}^2+\Delta_k\right).
\]
Since
\[
u_k^2-u_{k+1}^2=-2u_kd_k-d_k^2,
\]
we have
\[
t_k=\frac{u_k\Delta_k-2u_k^2d_k-u_kd_k^2}{v_k}.
\]
Also,
\[
u_k^3-u_{k+1}^3=-3u_k^2d_k-3u_kd_k^2-d_k^3,
\]
hence
\[
\frac23\,\frac{u_k^3-u_{k+1}^3}{v_k}
=
\frac{-2u_k^2d_k-2u_kd_k^2-\frac23 d_k^3}{v_k}.
\]
Subtracting,
\[
t_k-\frac23\,\frac{u_k^3-u_{k+1}^3}{v_k}
=
\frac{u_k\Delta_k+u_kd_k^2+\frac23 d_k^3}{v_k}.
\]
Furthermore, 
\[
\frac{u_k^3-u_{k+1}^3}{v_k}
=
\left[\frac{u_k^3}{v_k}-\frac{u_{k+1}^3}{v_{k+1}}\right] - \left[\frac{u_{k+1}^3}{v_k}-\frac{u_{k+1}^3}{v_{k+1}}\right]
=
\left[\frac{u_k^3}{v_k}-\frac{u_{k+1}^3}{v_{k+1}}\right] - \frac{u_{k+1}^3 (v_{k+1}-v_k)}{v_k v_{k+1}}.
\]
Therefore
\begin{align*}
    \sum_{k=1}^n t_k
&=
\sum_{k=1}^n \frac{2}{3} \left(\left[\frac{u_k^3}{v_k}-\frac{u_{k+1}^3}{v_{k+1}}\right] - \frac{u_{k+1}^3 \Delta_k}{v_k v_{k+1}}\right)+\sum_{k=1}^n \frac{u_k\Delta_k+u_kd_k^2+\frac23 d_k^3}{v_k}
\\&=\frac23\,\frac{u_1^3}{v_1}-\frac23\,\frac{u_{n+1}^3}{v_{n+1}}
+
\sum_{k=1}^n \left[ \frac{u_k\Delta_k+u_kd_k^2+\frac23 d_k^3}{v_k} - \frac23\frac{u_{k+1}^3 \Delta_k}{v_k v_{k+1}} \right].
\end{align*}
From $d_k^2\le \Delta_k$ we get $u_kd_k^2\le u_k\Delta_k$ and $d_k^3 \le \Delta_k^{3/2}$, thus
\begin{align*}
    2u_{n+1}+\sum_{k=1}^n t_k
&\le
2u_{n+1}+\frac23\,\frac{u_1^3}{v_1}-\frac23\,\frac{u_{n+1}^3}{v_{n+1}}
+
\sum_{k=1}^n \left[ \frac{u_k\Delta_k+u_k\Delta_k+\frac23 \Delta_k^{3/2}}{v_k} - \frac23\frac{u_{k+1}^3 \Delta_k}{v_k v_{k+1}} \right]
\\&= 2u_{n+1}-\frac23\,\frac{u_{n+1}^3}{v_{n+1}}+\frac23\,\frac{u_1^3}{v_1}
+\frac23\sum_{k=1}^n \left[\frac{3u_k\Delta_k}{v_k} - \frac{u_{k+1}^3 \Delta_k}{v_k v_{k+1}}\right]
+\frac{2}{3} \sum_{k=1}^n \frac{\Delta_k^{3/2}}{v_k}.
\end{align*}
Maximization in $u_{n+1}\ge 0$ gives
\[
2u_{n+1}-\frac23\,\frac{u_{n+1}^3}{v_{n+1}} \le \frac{4}{3} \sqrt{v_{n+1}}.
\]
By Lemma \ref{keybound1},
\[
\frac23\sum_{k=1}^n \left[\frac{3u_k\Delta_k}{v_k}
-\frac{u_{k+1}^3\Delta_k}{v_kv_{k+1}}\right]
\le
4\sum_{k=1}^n \frac{\Delta_k}{\sqrt{v_k}}
+
8\sum_{k=1}^n \frac{\Delta_k^{3/2}}{v_k}.
\]
Therefore
\begin{align*}
2u_{n+1}+ \sum_{k=1}^n t_k
&\le
\frac{4}{3} \sqrt{v_{n+1}}
+\frac23\,\frac{u_1^3}{v_1}
+4\sum_{k=1}^n \frac{\Delta_k}{\sqrt{v_k}}
+\frac{26}{3} \sum_{k=1}^n \frac{\Delta_k^{3/2}}{v_k}.
\end{align*}
Dividing by $2$ and substituting back $t_k$ gives the lemma.
\end{proof}

We now apply the scalar estimates with \(u_k=\|s_k\|\) and \(v_k=\sum_{i=0}^{k-1}\lambda_i^2\|g_i\|^2\). The resulting bound has the same basic form as in AdaGrad, up to the factor \(1+\rho_n\).

\begin{theorem*}[Theorem \ref{thm:weighted}]
In Algorithm \ref{algo1}, for every \(n\ge1\),
\[
f(\hat x_n)-f_*
\le
14\widehat D(1+\rho_n)
\frac{\sqrt{\sum_{k=0}^{n}\lambda_k^2 \|g_k\|^2}}{\sum_{k=0}^n \lambda_k},
\]
where
\[
\rho_n:=
\max_{1\le k\le n}
\frac{\lambda_k^2\|g_k\|^2}{\sum_{i=0}^{k-1}\lambda_i^2\|g_i\|^2}.
\]
\end{theorem*}

\begin{proof}
In the case where \(\lambda_0 g_0 = 0\), the estimate is trivial. Thus we assume that \(v_1=\lambda_0^2\|g_0\|^2>0\).

For each \(k\ge 0\), convexity of \(f\) and the fact that \(g_k\in\partial f(x_k)\) give
\[
f(x_k)-f_*
\le
\langle g_k,x_k-x_*\rangle.
\]
Multiplying by \(\lambda_k\),
\[
\lambda_k\left(f(x_k)-f_*\right)
\le
\langle \lambda_k g_k,x_k-x_*\rangle
=
\langle s_{k+1}-s_k,x_k-x_*\rangle.
\]
Since \(x_k=x_0-\gamma_k s_k\), we get
\[
\lambda_k\left(f(x_k)-f_*\right)
\le
\langle s_{k+1}-s_k,x_0-x_*\rangle-\gamma_k\langle s_{k+1}-s_k,s_k\rangle.
\]
Summing from \(k=0\) to \(n\),
\[
\sum_{k=0}^{n}\lambda_k\left(f(x_k)-f_*\right)
\le
\left\langle \sum_{k=0}^{n} (s_{k+1}-s_k),\; x_0-x_*\right\rangle
-
\sum_{k=0}^{n}\gamma_k\langle s_{k+1}-s_k,s_k\rangle.
\]
Since \(s_0=0\), we obtain
\begin{equation}\label{convexitysum}
\begin{aligned}
        \sum_{k=0}^{n}\lambda_k\left(f(x_k)-f_*\right)
&\le
\|s_{n+1}\|\,\|x_0-x_*\|
-
\sum_{k=1}^{n}\gamma_k\langle s_{k+1}-s_k,s_k\rangle
\\&\le \widehat D\|s_{n+1}\|
-
\sum_{k=1}^{n}\gamma_k\langle s_{k+1}-s_k,s_k\rangle
\end{aligned}
\end{equation}
Now define the scalar sequences
\[
u_k:=\|s_k\|,\qquad v_k:=\sum_{i=0}^{k-1}\lambda_i^2\|g_i\|^2\qquad (k\ge 1)
\]
and their successive differences
\[
d_k := u_{k+1}-u_k, \qquad \Delta_k := v_{k+1}-v_k.
\]
Then
\[
\gamma_k=\frac{\widehat Du_k}{v_k},\qquad \Delta_k=\lambda_k^2\|g_k\|^2,\qquad u_k\ge 0,\qquad v_1>0.
\]
Moreover,
\[
u_{k+1}^2
=
\|s_k+\lambda_k g_k\|^2
=
u_k^2+2\langle \lambda_k g_k,s_k\rangle+\lambda_k^2\|g_k\|^2,
\]
so
\[
2\langle -(s_{k+1}-s_k),s_k\rangle
=
u_k^2-u_{k+1}^2+\lambda_k^2\|g_k\|^2
=
u_k^2-u_{k+1}^2+\Delta_k.
\]
Therefore
\begin{equation}\label{eqgammasum}
    \frac{\widehat D}{2}\frac{u_k}{v_k}\left(u_k^2-u_{k+1}^2+\Delta_k\right)
=
-\gamma_k\langle s_{k+1}-s_k,s_k\rangle.
\end{equation}
Also,
\[
|d_k|
=
\left|\|s_k+\lambda_k g_k\|-\|s_k\|\right|
\le
\lambda_k\|g_k\|,
\]
hence
\[
d_k^2\le \lambda_k^2\|g_k\|^2=\Delta_k.
\]
We may then apply Lemma \ref{keybound2} with the scalar sequences \((u_k)\) and \((v_k)\):
\begin{align*}
u_{n+1}+ \frac{1}{2} \sum_{k=1}^n
\frac{u_k}{v_k}\left(u_k^2-u_{k+1}^2+\Delta_k\right)
&\le
\frac{2}{3} \sqrt{v_{n+1}}
+\frac13\,\frac{u_1^3}{v_1}
+2\sum_{k=1}^n \frac{\Delta_k}{\sqrt{v_k}}
+\frac{13}{3} \sum_{k=1}^n \frac{\Delta_k^{3/2}}{v_k},
\end{align*}
with \(u_1^3/v_1=\sqrt{v_1}\). Notice that
\[
\rho_n=\max_{1\le k\le n}\frac{\Delta_k}{v_k}.
\]
Applying Lemma~\ref{sumbounds},
\begin{align*}
2\sum_{k=1}^n \frac{\Delta_k}{\sqrt{v_k}}
&\le
2\left(1+\sqrt{1+\rho_n}\right)(\sqrt{v_{n+1}}-\sqrt{v_1}) \le 4 (1+\rho_n) (\sqrt{v_{n+1}}-\sqrt{v_1}), \\
\frac{13}{3}\sum_{k=1}^n \frac{\Delta_k^{3/2}}{v_k}
&\le
\frac{13}{3}\sqrt{\rho_n}\,\left(1+\sqrt{1+\rho_n}\right)(\sqrt{v_{n+1}}-\sqrt{v_1}) \le \frac{26}{3} (1+\rho_n) (\sqrt{v_{n+1}}-\sqrt{v_1}).
\end{align*}
Hence
\begin{align*}
u_{n+1}+ \frac{1}{2} \sum_{k=1}^n
\frac{u_k}{v_k}\left(u_k^2-u_{k+1}^2+\Delta_k\right)
&\le
\frac23 \sqrt{v_{n+1}}
+\frac13\sqrt{v_1} 
+ \frac{38}{3} (1+\rho_n)
(\sqrt{v_{n+1}}-\sqrt{v_1})
\\&\le 14 (1+\rho_n) \sqrt{v_{n+1}}.
\end{align*}
Multiplying by $\widehat D$ and using \eqref{convexitysum} and \eqref{eqgammasum}, we obtain
\[
\sum_{k=0}^{n}\lambda_k\left(f(x_k)-f_*\right)
\le
14 \widehat D(1+\rho_n) \sqrt{v_{n+1}}.
\]
Finally, by Jensen's inequality, we get
\[
f(\hat x_n)-f_*
\le
\frac{1}{\sum_{k=0}^n \lambda_k}14 \widehat D(1+\rho_n) \sqrt{v_{n+1}},
\]
which gives the result after substituting back $v_{n+1}$.
\end{proof}

The next corollary makes explicit when the general estimate recovers the standard \(n^{-1/2}\) rate. The only additional issue is whether the gradients vanish too quickly when the weights grow.

\begin{theorem*}[Theorem \ref{thm:weighted_poly}]
Assume that, in Algorithm~\ref{algo1}, the weights satisfy
\[
\lambda_k \asymp k^\alpha \qquad (k\ge 1)
\]
for some \(\alpha>-\tfrac12\). Define
\[
\delta_n:=
\begin{cases}
1, & \text{if} \ \ \alpha\le 0,\\[4pt]
\displaystyle \min_{1\le k\le n}\sqrt{k}\,\|g_k\|, & \text{if} \ \ \alpha>0.
\end{cases}
\]
Then
\[
f(\hat x_n)-f_*
=
\mathcal O\!\left(\frac1{\delta_n^2 \, \sqrt n}\right).
\]
\end{theorem*}

\begin{proof}
Since \(\lambda_k\asymp k^\alpha\) with \(\alpha>-\tfrac12\), we have
\[
\sum_{k=1}^n \lambda_k \asymp n^{\alpha+1},
\qquad
\sum_{k=1}^n \lambda_k^2 \asymp n^{2\alpha+1}.
\]
From that and $G$-Lipschitzness of $f$,
\[
\frac{\sqrt{\sum_{k=0}^{n}\lambda_k^2\|g_k\|^2}}{\sum_{k=0}^n \lambda_k}
\le G\frac{\sqrt{\sum_{k=0}^{n}\lambda_k^2}}{\sum_{k=0}^n \lambda_k}
\le G\frac{\sqrt{\lambda_0}+\sqrt{\sum_{k=1}^{n}\lambda_k^2}}{\lambda_0+\sum_{k=1}^n \lambda_k}
=\mathcal O\!\left(\frac1{\sqrt n}\right).
\]

If \(\alpha\le 0\), then \((\lambda_k)\) is bounded, and therefore
\[
\rho_n:=
\max_{1\le k\le n}
\frac{\lambda_k^2\|g_k\|^2}{\sum_{i=0}^{k-1}\lambda_i^2\|g_i\|^2}
\]
is bounded above, since $\|g_k\|^2 \le G^2$ and
\[
\sum_{i=0}^{k-1}\lambda_i^2\|g_i\|^2\ge \lambda_0^2\|g_0\|^2 >0.
\]

Assume now that \(\alpha>0\). Since \(\sqrt{k}\,\|g_k\|\ge \delta_n\) for \(1\le k\le n\), we have
\[
\|g_i\|\ge \frac{\delta_n}{\sqrt{i}}\ge \frac{\delta_n}{\sqrt{k}}
\qquad (1\le i\le k),
\]
hence
\[
\min_{0\le i\le k}\|g_i\|^2\ge \frac{\delta_n^2}{k}.
\]
Therefore,
\[
\sum_{i=0}^{k-1}\lambda_i^2\|g_i\|^2
\ge
\left(\min_{0\le i\le k}\|g_i\|^2\right)\sum_{i=0}^{k-1}\lambda_i^2
\ge
\frac{\delta_n^2}{k}\sum_{i=0}^{k-1}\lambda_i^2.
\]
Using also \(\|g_k\|\le G\), we obtain
\[
\rho_n
=
\max_{1\le k\le n}
\frac{\lambda_k^2\|g_k\|^2}{\sum_{i=0}^{k-1}\lambda_i^2\|g_i\|^2}
\le
\frac{G^2}{\delta_n^2}
\max_{1\le k\le n}
\frac{k\,\lambda_k^2}{\sum_{i=0}^{k-1}\lambda_i^2}.
\]
Since 
\[
\lambda_k^2\asymp k^{2\alpha},
\qquad
\sum_{i=0}^{k-1}\lambda_i^2
\asymp
k^{2\alpha+1},
\]
we have
\[
\frac{k\,\lambda_k^2}{\sum_{i=0}^{k-1}\lambda_i^2} = \mathcal O(1),
\]
and therefore
\[
\rho_n \lesssim \frac{G^2}{\delta_n^2}.
\]
The conclusion now follows straightforwardly from Theorem \ref{thm:weighted}.
\end{proof}

The same scalar mechanism extends to the coordinate-wise method. The proof proceeds componentwise and yields a bound of the same form as in the global case.

\begin{theorem*}[Theorem \ref{thm:weighted_coord}]
In Algorithm \ref{algo1_coord}, for every $n\geq 1$,
$$
f(\hat x_n)-f_*
\leq
14\|\vec D\|(1+\rho_n)
\frac{\sqrt{\sum_{k=0}^{n}\lambda_k^2\|g_k\|^2}}
{\sum_{k=0}^n\lambda_k},
$$
where
$$
\rho_n := \max\left\{ 0, \max_{\substack{1 \leq k \leq n, 1 \leq j \leq d \\ v_{k,j} > 0}} \frac{\lambda_k^2 g_{k,j}^2}{v_{k,j}} \right\}, \quad \text{where } v_{k,j} := \sum_{i=0}^{k-1} \lambda_i^2 g_{i,j}^2.
$$
\end{theorem*}

\begin{proof}
For every $k\geq 0$, convexity of $f$ and
$g_k\in\partial f(x_k)$ give
$$
f(x_k)-f_*
\leq
\langle g_k,x_k-x_*\rangle.
$$
Multiplying by $\lambda_k$ yields
$$
\lambda_k\bigl(f(x_k)-f_*\bigr)
\leq
\langle s_{k+1}-s_k,x_k-x_*\rangle.
$$
Since $x_k=x_0-\gamma_k\odot s_k$,
$$
\lambda_k\bigl(f(x_k)-f_*\bigr)
\leq
\langle s_{k+1}-s_k,x_0-x_*\rangle
-
\langle s_{k+1}-s_k,\gamma_k\odot s_k\rangle.
$$
Summing from $k=0$ to $n$ gives
\begin{align*}
\sum_{k=0}^{n}\lambda_k\bigl(f(x_k)-f_*\bigr)
&\leq
\left\langle
\sum_{k=0}^{n}(s_{k+1}-s_k),x_0-x_*
\right\rangle
-
\sum_{k=0}^{n}
\langle s_{k+1}-s_k,\gamma_k\odot s_k\rangle\\
&=
\langle s_{n+1},x_0-x_*\rangle
-
\sum_{k=1}^{n}
\langle s_{k+1}-s_k,\gamma_k\odot s_k\rangle.
\end{align*}
Therefore,
\begin{equation}\label{convexitysumcoord}
\begin{aligned}
\sum_{k=0}^{n}\lambda_k\bigl(f(x_k)-f_*\bigr)
&\leq
\sum_{j=1}^d
|s_{n+1,j}|\,|(x_0-x_*)_j|
-
\sum_{k=1}^{n}\sum_{j=1}^d
\gamma_{k,j}(s_{k+1,j}-s_{k,j})s_{k,j}\\
&\leq
\sum_{j=1}^d
\vec D_j|s_{n+1,j}|
-
\sum_{k=1}^{n}\sum_{j=1}^d
\gamma_{k,j}(s_{k+1,j}-s_{k,j})s_{k,j}.
\end{aligned}
\end{equation}

Fix a coordinate $j\in\{1,\ldots,d\}$. If $\lambda_k g_{k,j}=0$ for every $0\leq k\leq n$, then $s_{k,j}=v_{k,j}=0$ (and $\gamma_{k,j}=0$) for all $k$, so this coordinate contributes exactly zero to \eqref{convexitysumcoord}. Thus, we may assume the coordinate is active and define its first activation index:
$$
\tau_j
:=
\min\left\{
k\in\{0,\ldots,n\}:
\lambda_k g_{k,j}\neq 0
\right\}.
$$
For $k\leq \tau_j$, we have $s_{k,j}=v_{k,j}=0$,
whereas
$$
s_{\tau_j+1,j}
=
\lambda_{\tau_j}g_{\tau_j,j},
\qquad
v_{\tau_j+1,j}
=
\lambda_{\tau_j}^2g_{\tau_j,j}^2>0.
$$
In particular,
$$
|s_{\tau_j+1,j}|
=
\sqrt{v_{\tau_j+1,j}}.
$$

For $k\geq \tau_j+1$, set
$$
u_{k,j}:=|s_{k,j}|,
\qquad
v_{k,j}:=
\sum_{i=0}^{k-1}\lambda_i^2g_{i,j}^2,
$$
and
$$
d_{k,j}:=u_{k+1,j}-u_{k,j},
\qquad
\Delta_{k,j}:=v_{k+1,j}-v_{k,j}.
$$
Then
$$
\gamma_{k,j}
=
\frac{\vec D_j u_{k,j}}{v_{k,j}},
\qquad
\Delta_{k,j}
=
\lambda_k^2g_{k,j}^2.
$$
Moreover,
\begin{align*}
u_{k+1,j}^2
&=
s_{k+1,j}^2\\
&=
\bigl(s_{k,j}+\lambda_k g_{k,j}\bigr)^2\\
&=
u_{k,j}^2
+
2(s_{k+1,j}-s_{k,j})s_{k,j}
+
\Delta_{k,j}.
\end{align*}
Hence
$$
2\bigl(-(s_{k+1,j}-s_{k,j})s_{k,j}\bigr)
=
u_{k,j}^2-u_{k+1,j}^2+\Delta_{k,j},
$$
and therefore
\begin{equation}\label{eqgammasumcoord}
\frac{\vec D_j}{2}
\frac{u_{k,j}}{v_{k,j}}
\left(
u_{k,j}^2-u_{k+1,j}^2+\Delta_{k,j}
\right)
=
-\gamma_{k,j}
(s_{k+1,j}-s_{k,j})s_{k,j}.
\end{equation}
Furthermore,
$$
|d_{k,j}|
=
\bigl||s_{k+1,j}|-|s_{k,j}|\bigr|
\leq
|s_{k+1,j}-s_{k,j}|
=
\lambda_k|g_{k,j}|,
$$
so
$$
d_{k,j}^2\leq\Delta_{k,j}.
$$

Apply Lemma~\ref{keybound2} to the shifted sequences
$$
\widetilde u_{r,j}
:=
u_{\tau_j+r,j},
\qquad
\widetilde v_{r,j}
:=
v_{\tau_j+r,j},
\qquad
r\geq 1.
$$
Their initial values satisfy
$$
\widetilde v_{1,j}
=
v_{\tau_j+1,j}>0,
\qquad
\frac{\widetilde u_{1,j}^3}
{\widetilde v_{1,j}}
=
\sqrt{\widetilde v_{1,j}}.
$$
Thus,
\begin{align*}
&u_{n+1,j}
+
\frac12
\sum_{k=\tau_j+1}^{n}
\frac{u_{k,j}}{v_{k,j}}
\left(
u_{k,j}^2-u_{k+1,j}^2+\Delta_{k,j}
\right)\\
&\leq
\frac23\sqrt{v_{n+1,j}}
+
\frac13\sqrt{v_{\tau_j+1,j}}
+
2\sum_{k=\tau_j+1}^{n}
\frac{\Delta_{k,j}}{\sqrt{v_{k,j}}}
+
\frac{13}{3}
\sum_{k=\tau_j+1}^{n}
\frac{\Delta_{k,j}^{3/2}}{v_{k,j}}.
\end{align*}

Define
$$
\rho_{n,j}
:=
\max\left(
\left\{
\frac{\Delta_{k,j}}{v_{k,j}}:
\tau_j+1\leq k\leq n
\right\}
\cup\{0\}
\right).
$$
By Lemma~\ref{sumbounds},
\begin{align*}
2\sum_{k=\tau_j+1}^{n}
\frac{\Delta_{k,j}}{\sqrt{v_{k,j}}}
&\leq
4(1+\rho_{n,j})
\left(
\sqrt{v_{n+1,j}}
-
\sqrt{v_{\tau_j+1,j}}
\right),\\
\frac{13}{3}
\sum_{k=\tau_j+1}^{n}
\frac{\Delta_{k,j}^{3/2}}{v_{k,j}}
&\leq
\frac{26}{3}(1+\rho_{n,j})
\left(
\sqrt{v_{n+1,j}}
-
\sqrt{v_{\tau_j+1,j}}
\right).
\end{align*}
Consequently,
\begin{align*}
&u_{n+1,j}
+
\frac12
\sum_{k=\tau_j+1}^{n}
\frac{u_{k,j}}{v_{k,j}}
\left(
u_{k,j}^2-u_{k+1,j}^2+\Delta_{k,j}
\right)\\
&\leq
\frac23\sqrt{v_{n+1,j}}
+
\frac13\sqrt{v_{\tau_j+1,j}}
+
\frac{38}{3}(1+\rho_{n,j})
\left(
\sqrt{v_{n+1,j}}
-
\sqrt{v_{\tau_j+1,j}}
\right)\\
&\leq
14(1+\rho_{n,j})\sqrt{v_{n+1,j}}.
\end{align*}

For $k<\tau_j$, we have
$s_{k+1,j}-s_{k,j}=0$, while for $k=\tau_j$,
$\gamma_{\tau_j,j}=0$. Therefore, the omitted terms vanish, and
\eqref{eqgammasumcoord} gives
\begin{align*}
&\vec D_j|s_{n+1,j}|
-
\sum_{k=1}^{n}
\gamma_{k,j}(s_{k+1,j}-s_{k,j})s_{k,j}\\
&\leq
14\vec D_j(1+\rho_{n,j})\sqrt{v_{n+1,j}}.
\end{align*}

Summing over $j=1,\ldots,d$ and using
\eqref{convexitysumcoord}, we obtain
$$
\sum_{k=0}^{n}
\lambda_k\bigl(f(x_k)-f_*\bigr)
\leq
14\sum_{j=1}^d
\vec D_j(1+\rho_{n,j})\sqrt{v_{n+1,j}}.
$$
By definition, $\rho_{n,j}\leq\rho_n$, and hence
$$
\sum_{k=0}^{n}
\lambda_k\bigl(f(x_k)-f_*\bigr)
\leq
14(1+\rho_n)
\sum_{j=1}^d
\vec D_j\sqrt{v_{n+1,j}}.
$$
By Jensen's inequality and Cauchy--Schwarz,
\begin{align*}
f(\hat x_n)-f_*
&\leq
14(1+\rho_n)
\frac{
\sum_{j=1}^d\vec D_j\sqrt{v_{n+1,j}}
}{
\sum_{k=0}^n\lambda_k
}\\
&\leq
14(1+\rho_n)\|\vec D\|
\frac{
\sqrt{\sum_{j=1}^dv_{n+1,j}}
}{
\sum_{k=0}^n\lambda_k
}.
\end{align*}
Finally,
$$
\sum_{j=1}^dv_{n+1,j}
=
\sum_{k=0}^{n}\lambda_k^2\|g_k\|^2,
$$
which proves
$$
f(\hat x_n)-f_*
\leq
14\|\vec D\|(1+\rho_n)
\frac{
\sqrt{\sum_{k=0}^{n}\lambda_k^2\|g_k\|^2}
}{
\sum_{k=0}^n\lambda_k
}.
$$
\end{proof}

\subsection{Convex Experiments}
\newcommand{\figscale}{0.2} 

\begin{figure}[htbp]
    \centering
    \begin{tabular}{ccc}
        \includegraphics[scale=\figscale]{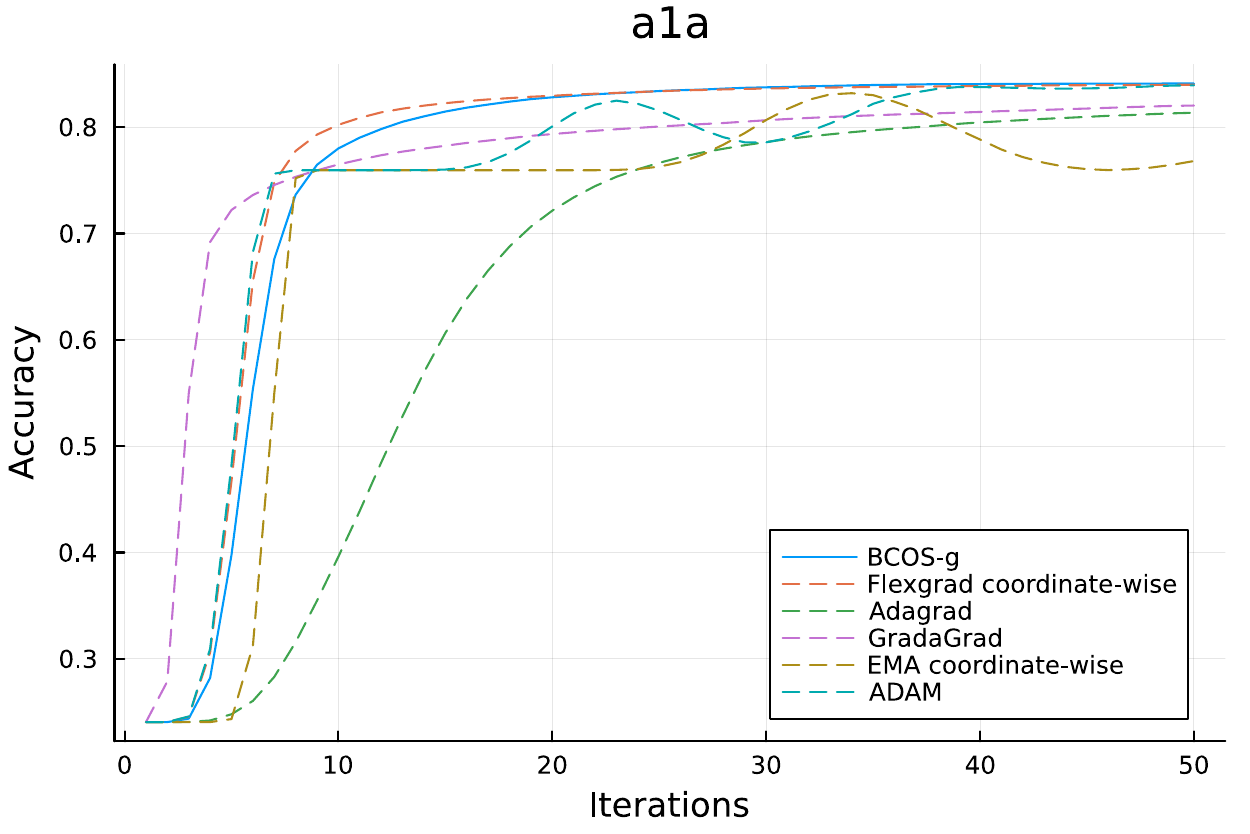}
        & \includegraphics[scale=\figscale]{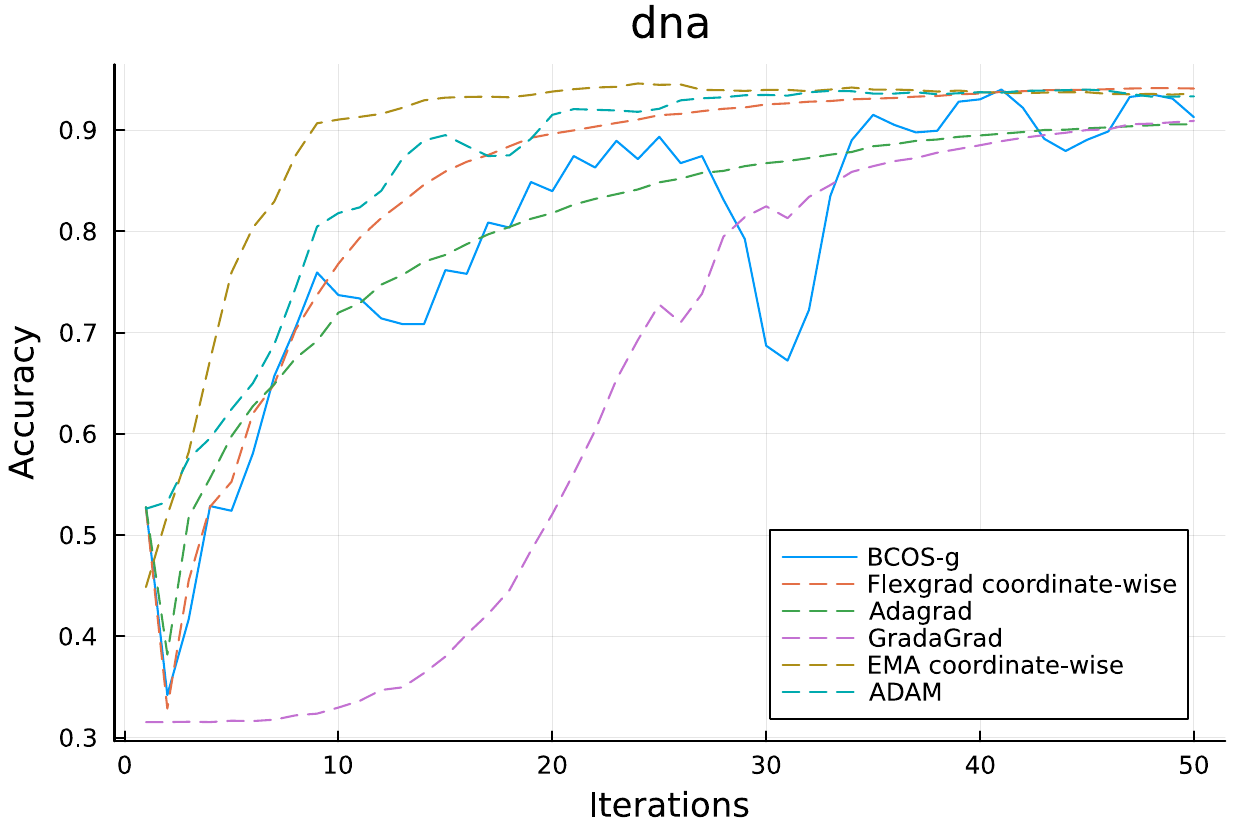}&
        \includegraphics[scale=\figscale]{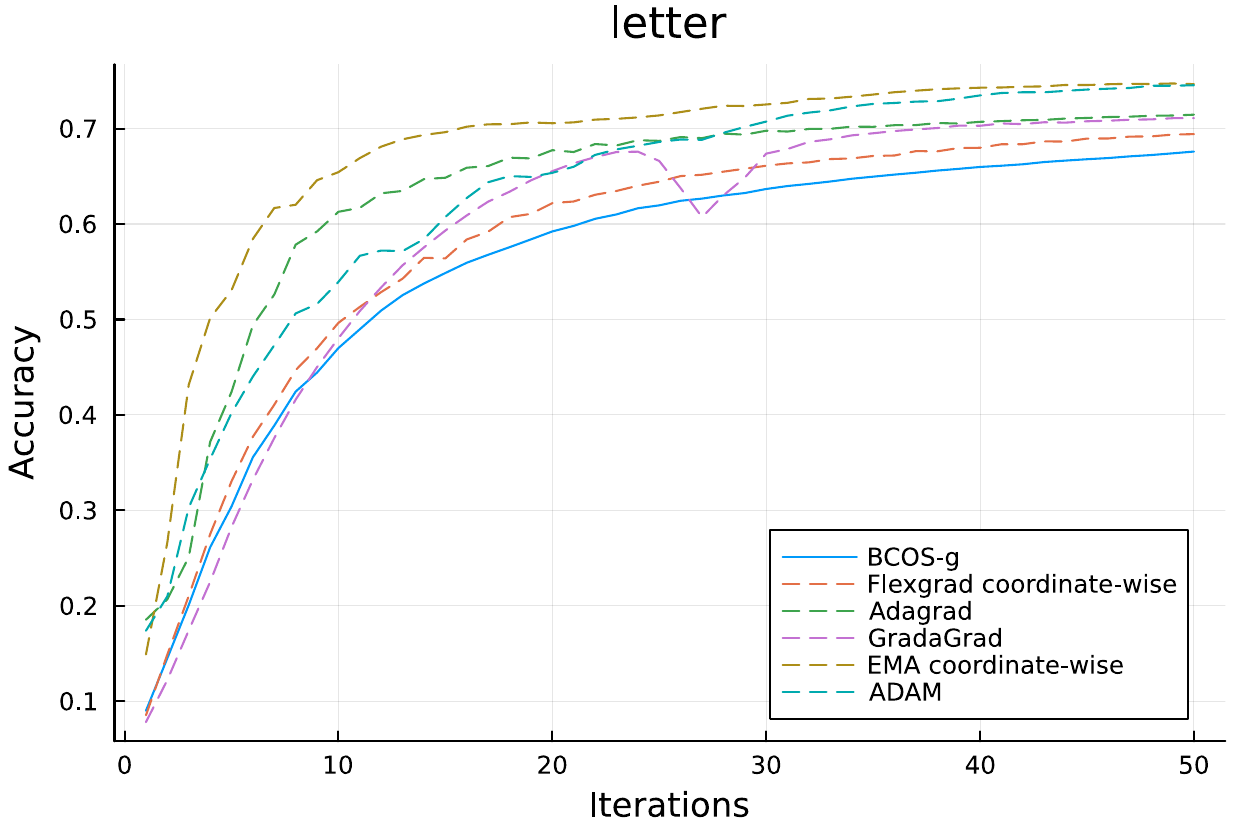}\\ 
        \includegraphics[scale=\figscale]{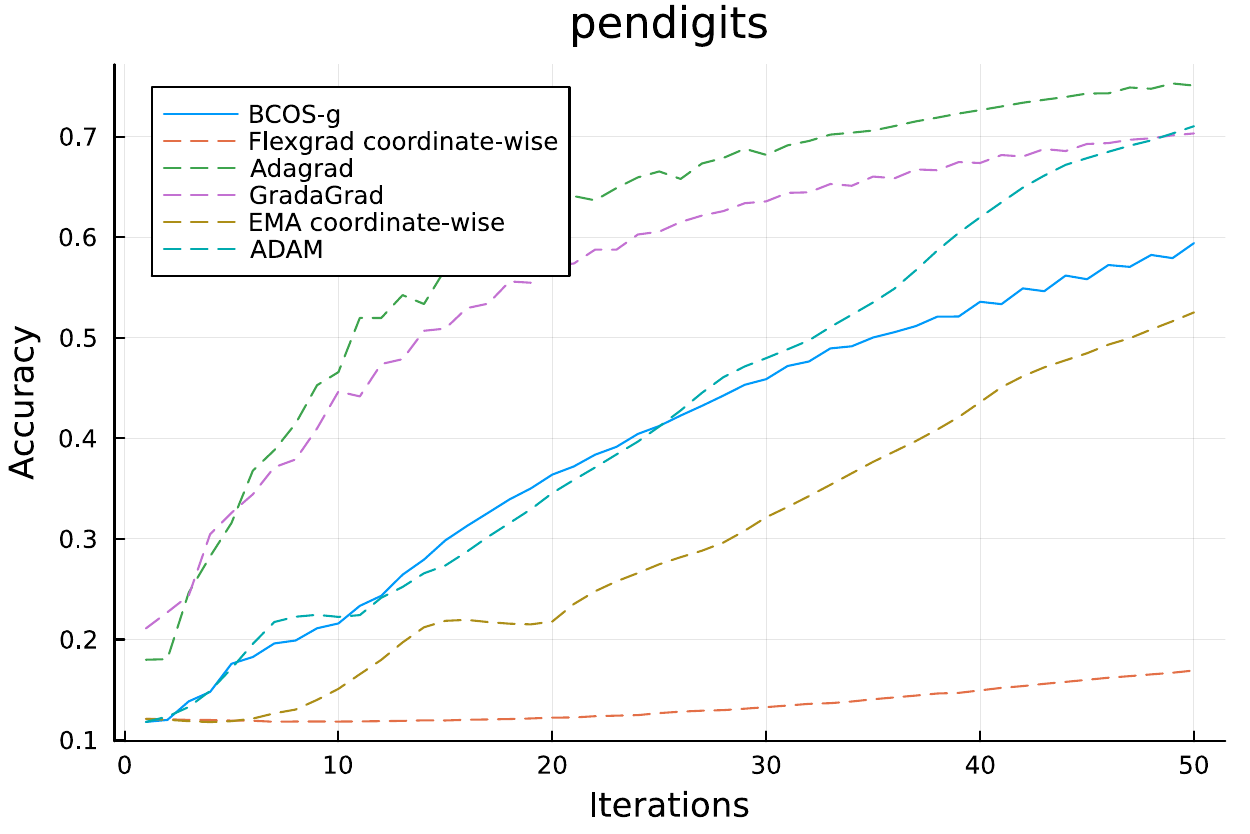}
        & \includegraphics[scale=\figscale]{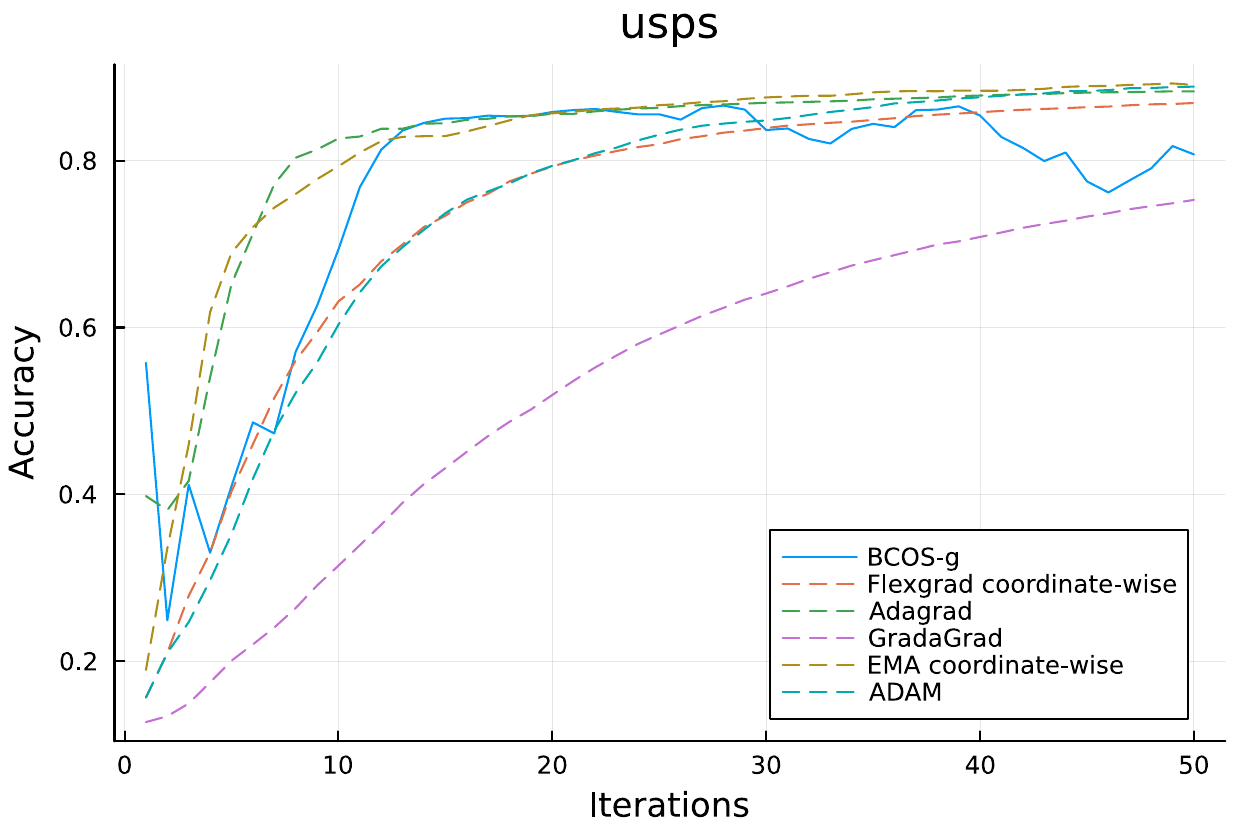}&
        \includegraphics[scale=\figscale]{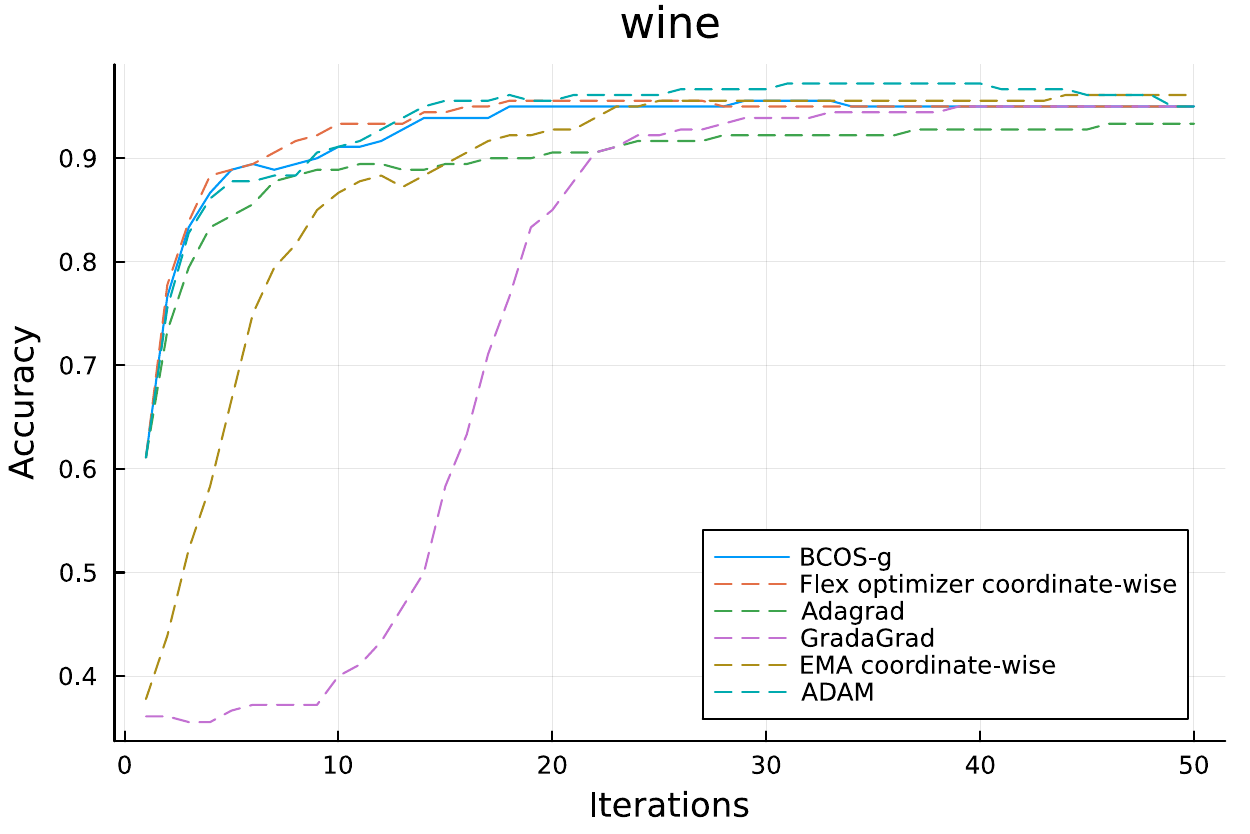}
     \end{tabular}
 \caption{Evolution of the accuracy along iterations on 6 data sets for BCOS, Flexgrad coordinate-wise, Adagrad, Gradagrad, EMA coordinate-wise, ADAM.}
\label{empprobmodel26}
\end{figure}

\begin{figure}[htbp]
    \centering
    \begin{tabular}{ccc}
  \includegraphics[scale=\figscale]{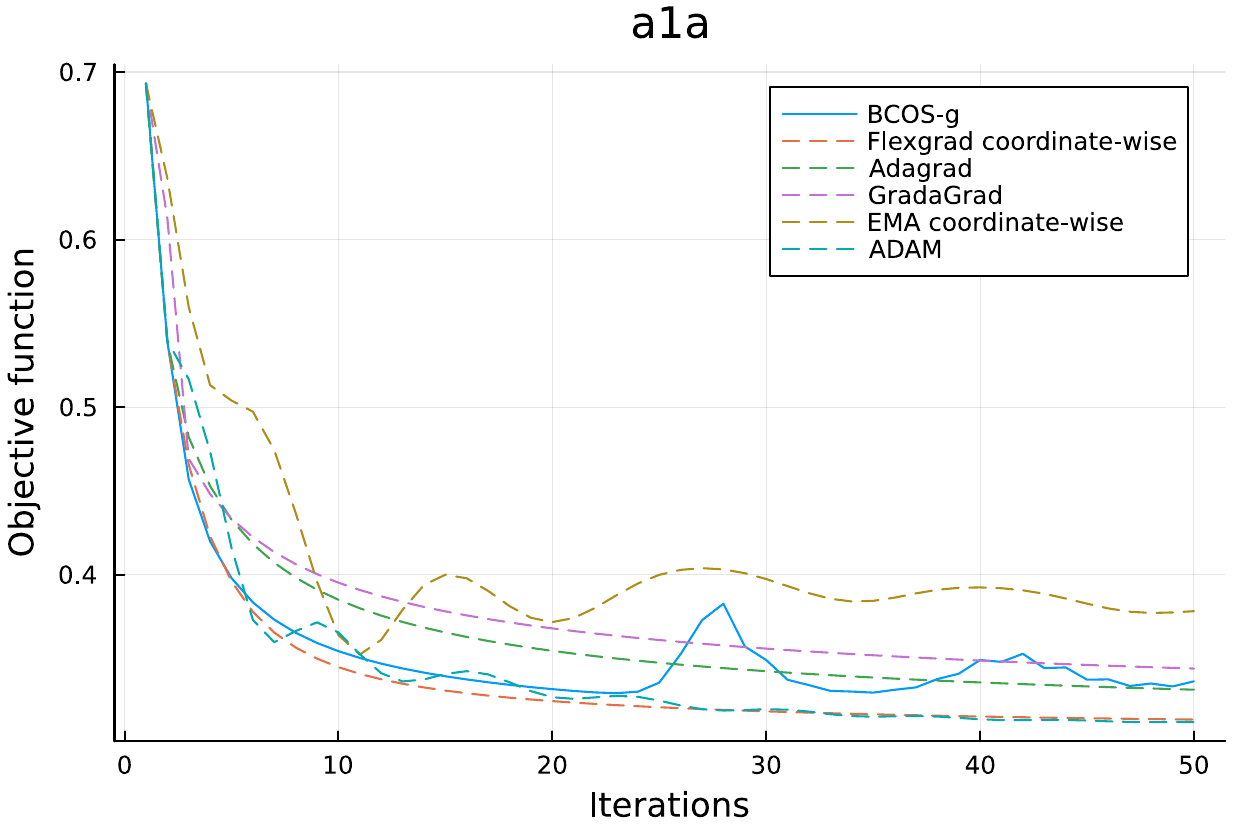}
        & \includegraphics[scale=\figscale]{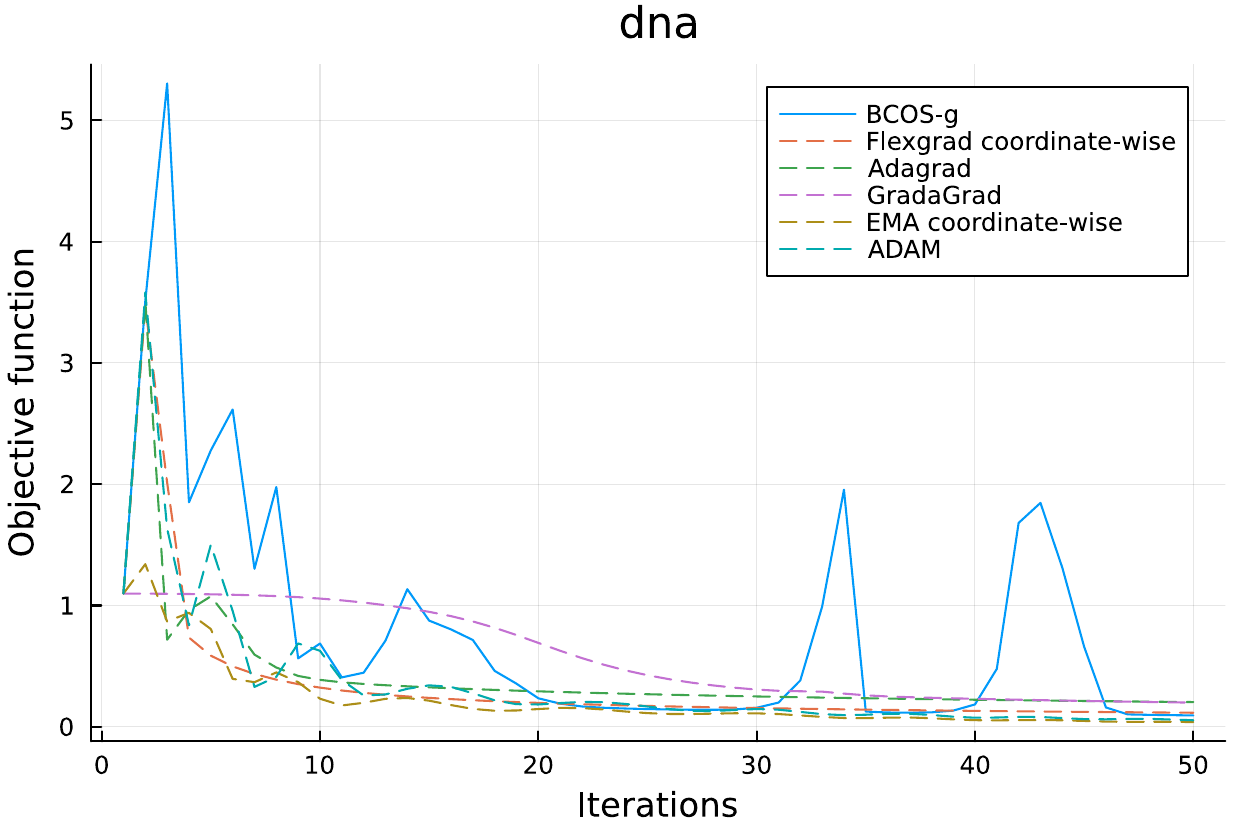}&
        \includegraphics[scale=\figscale]{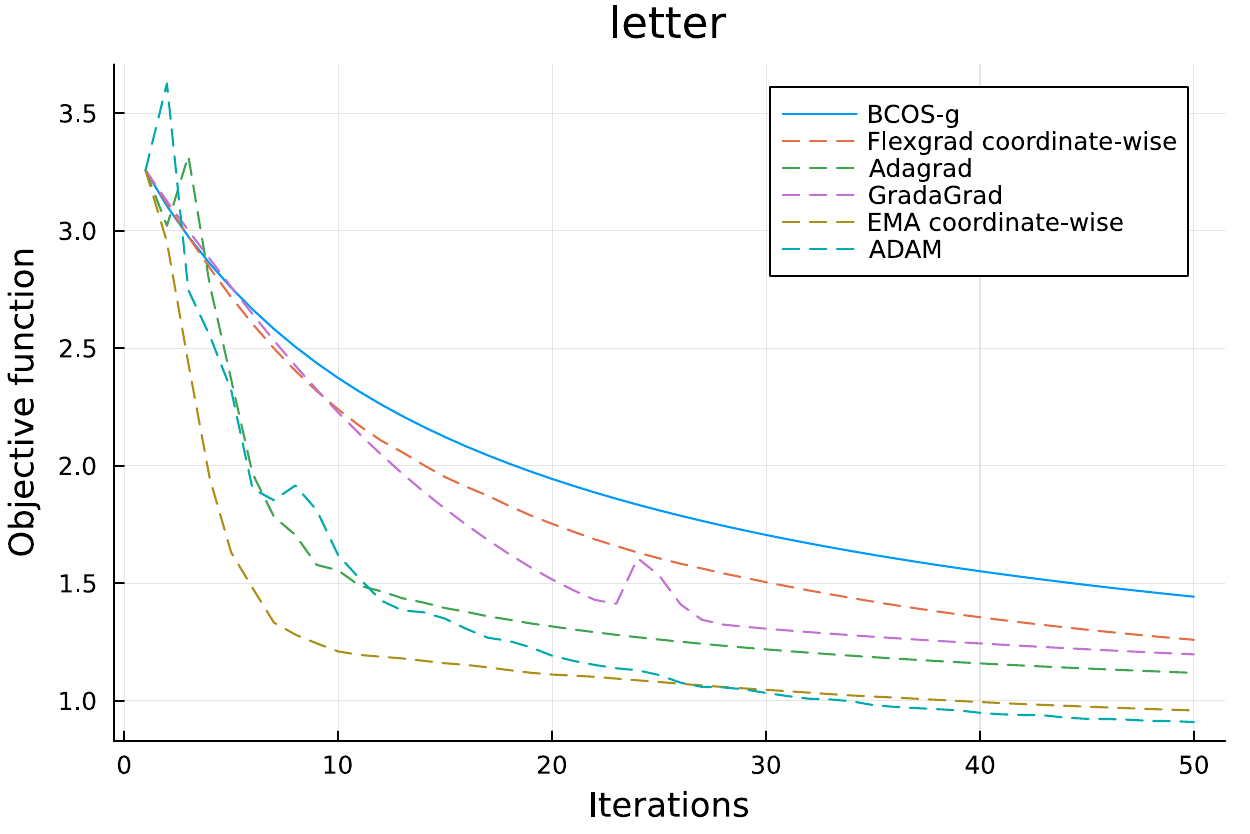}\\ 
        \includegraphics[scale=\figscale]{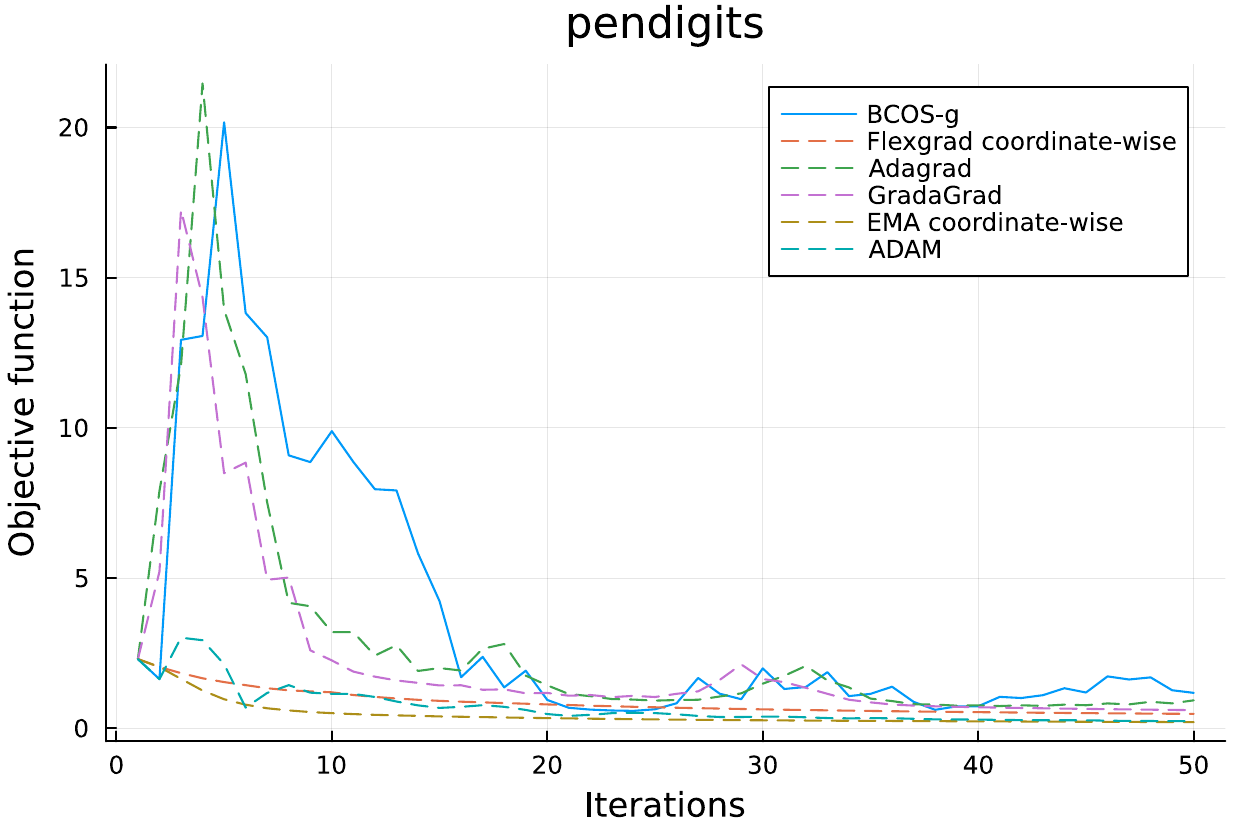}
        & \includegraphics[scale=\figscale]{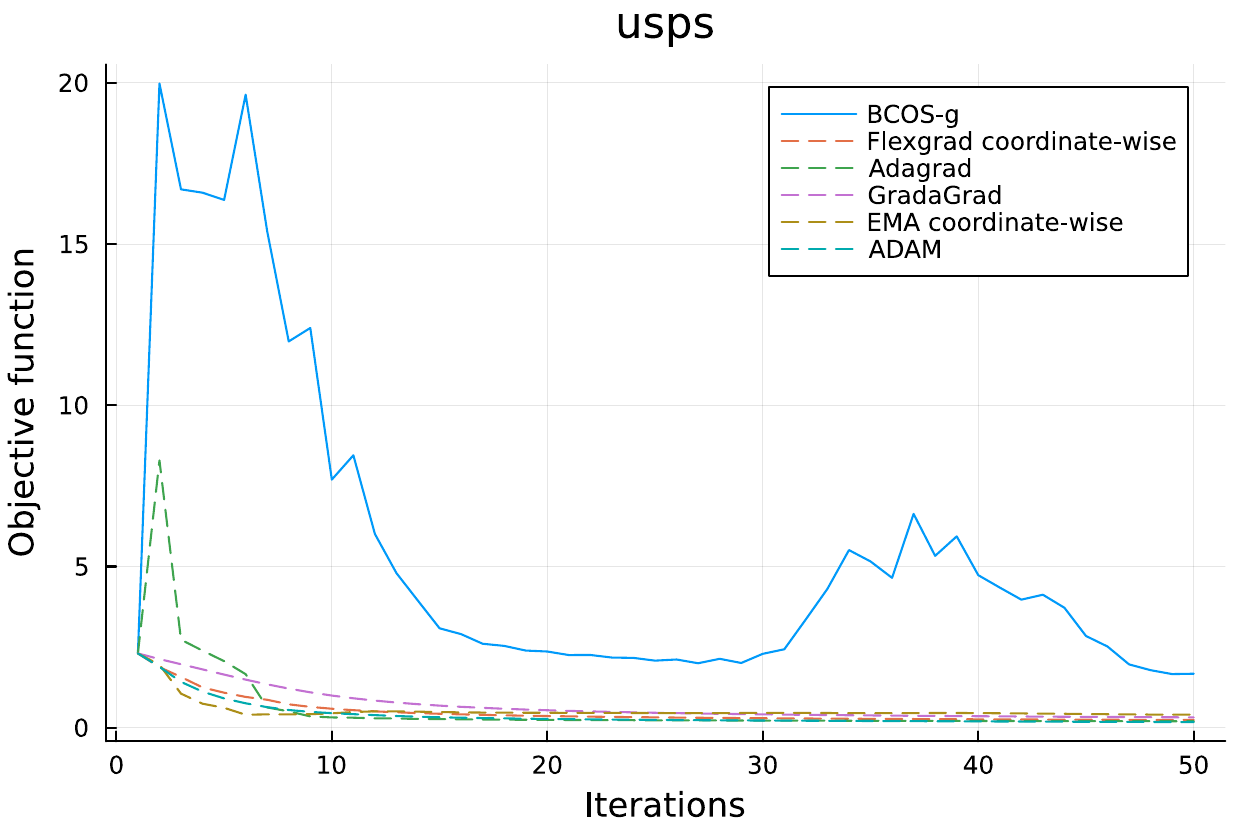}&
        \includegraphics[scale=\figscale]{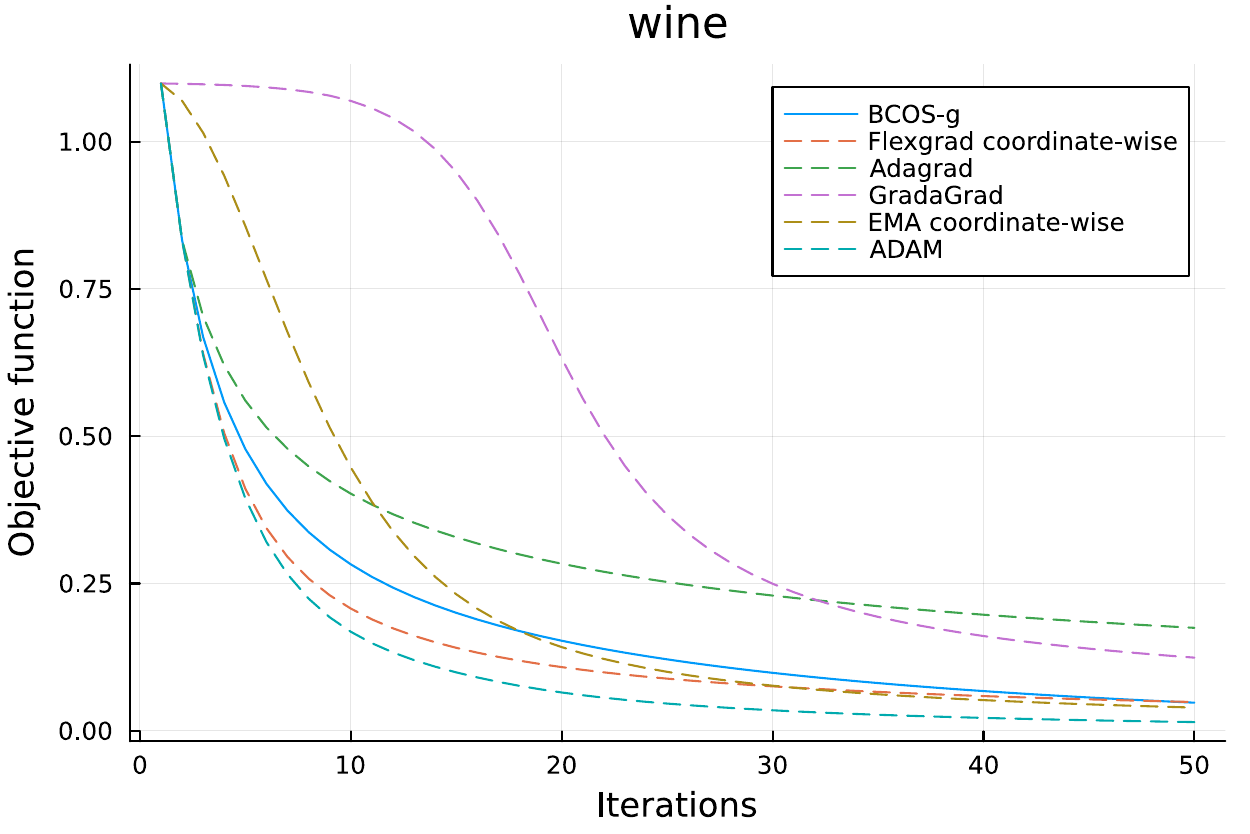}
     \end{tabular}
 \caption{Evolution of the objective function along iterations on 6 data sets for BCOS, Flexgrad coordinate-wise, Adagrad, Gradagrad, EMA coordinate-wise, ADAM.}
\label{empprobmodelo6}
\end{figure}

\begin{figure}[htbp]
    \centering
    \begin{tabular}{ccc}
  \includegraphics[scale=\figscale]{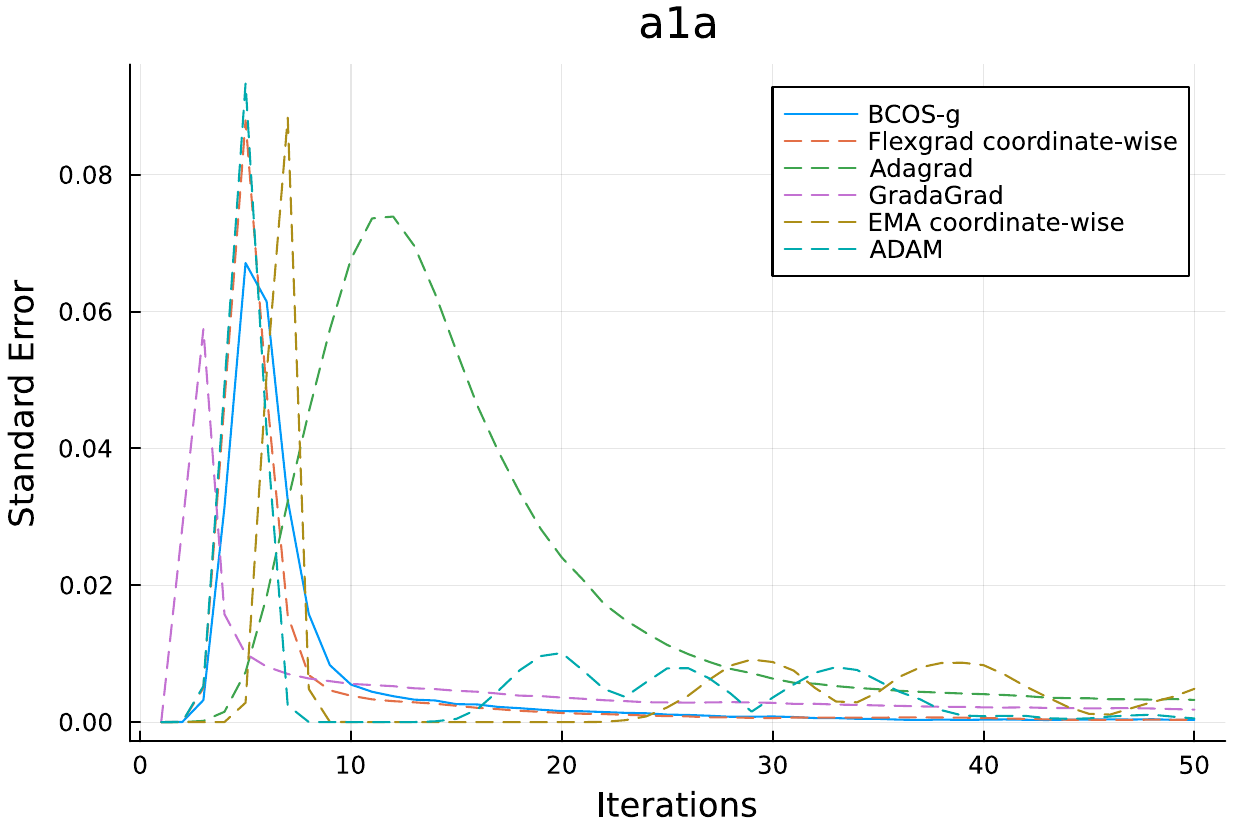}
        & \includegraphics[scale=\figscale]{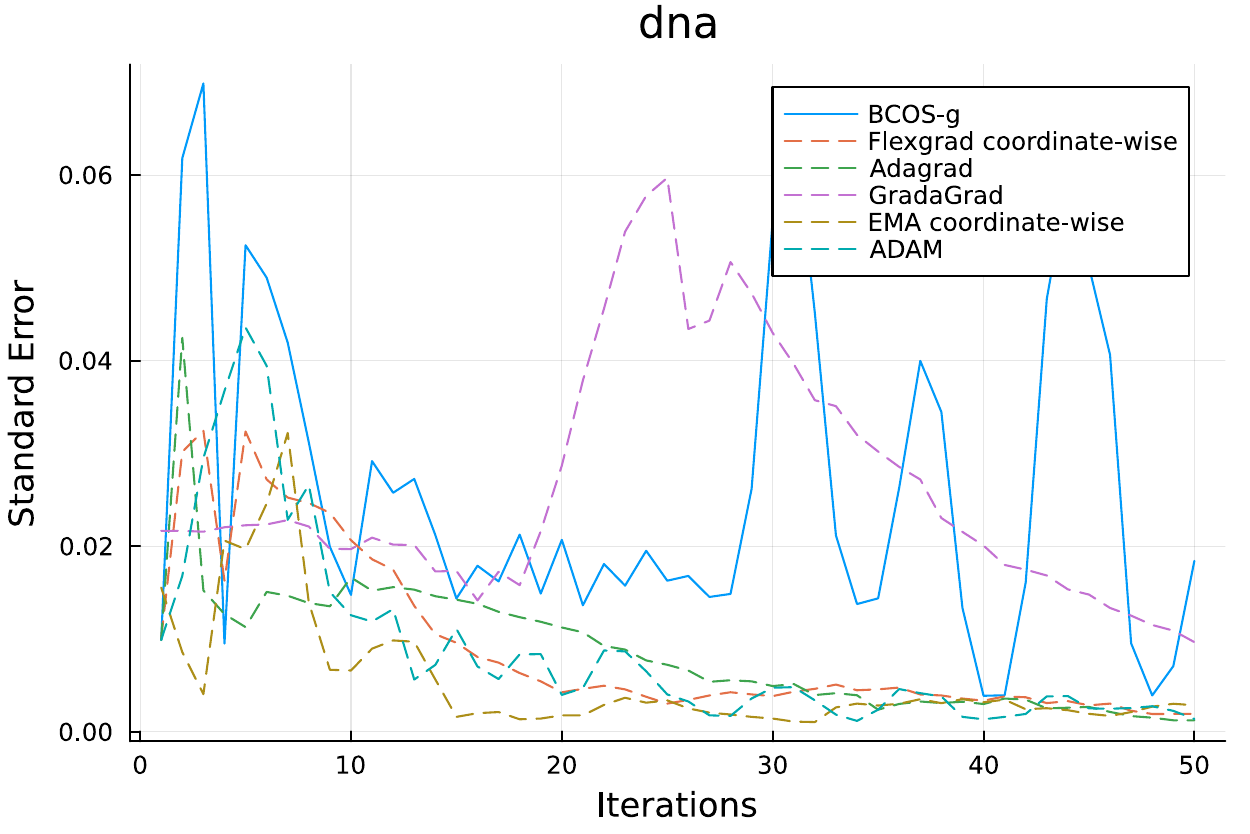}&
        \includegraphics[scale=\figscale]{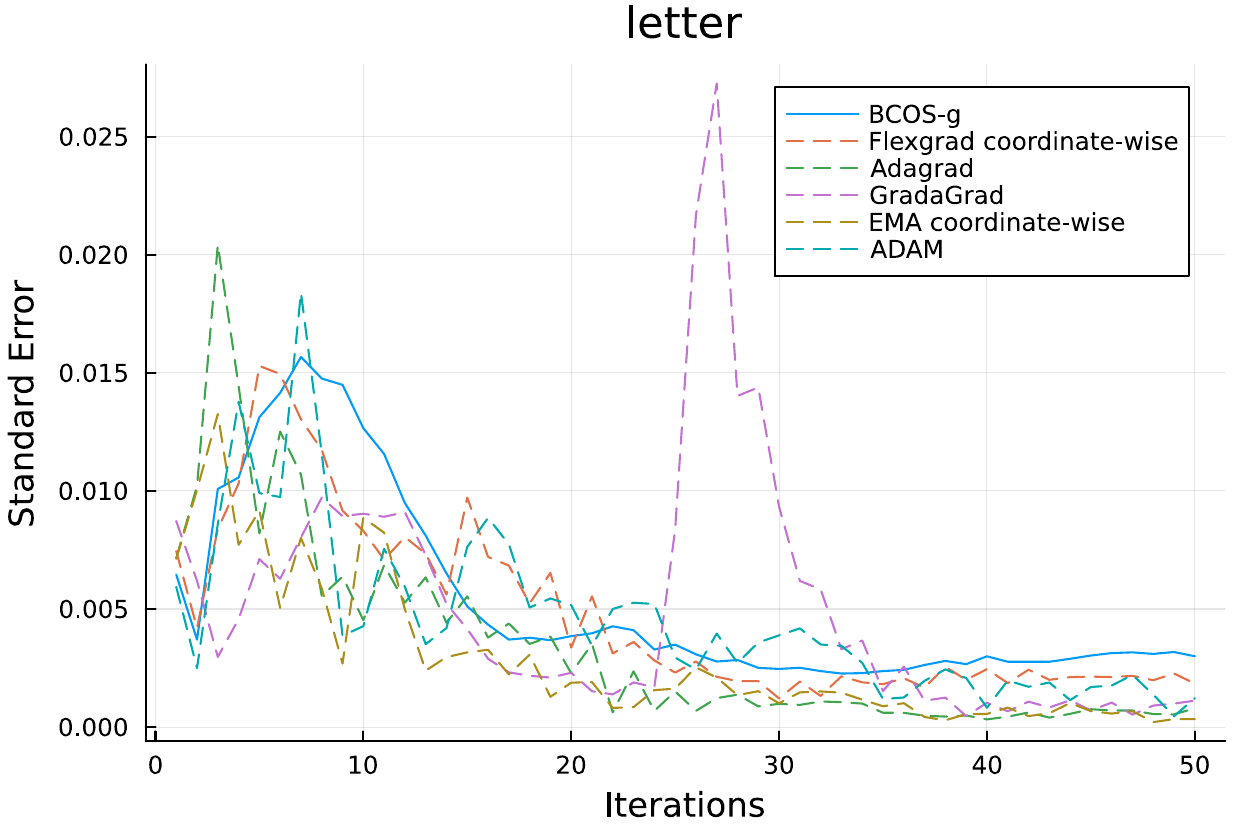}\\ 
        \includegraphics[scale=\figscale]{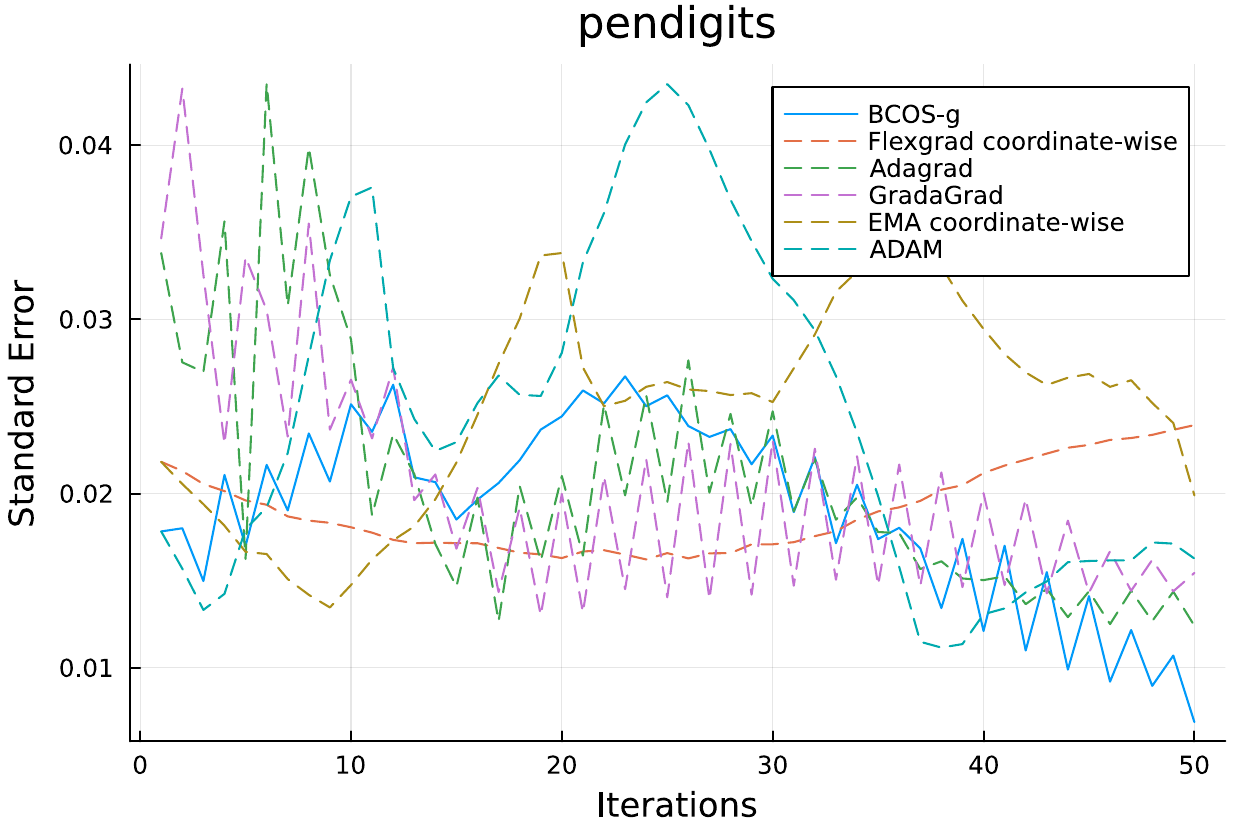}
        & \includegraphics[scale=\figscale]{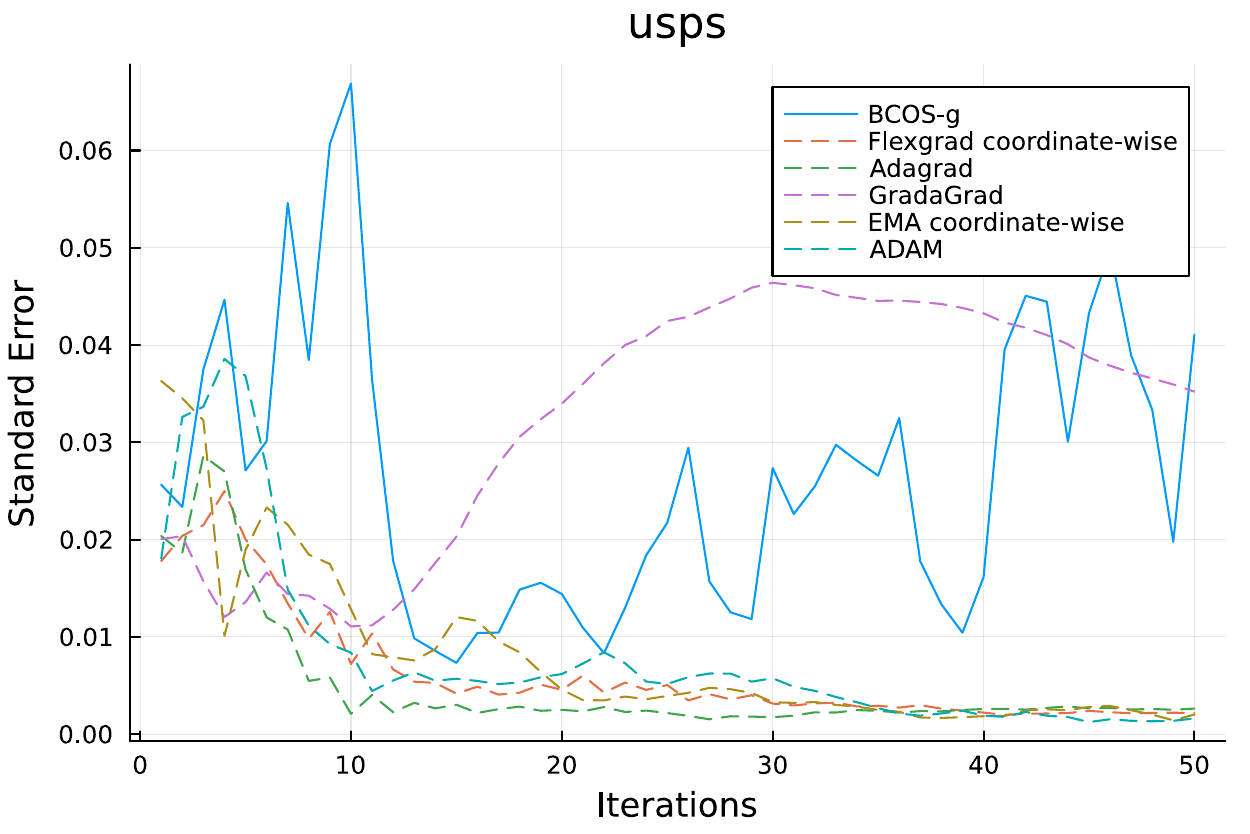}&
        \includegraphics[scale=\figscale]{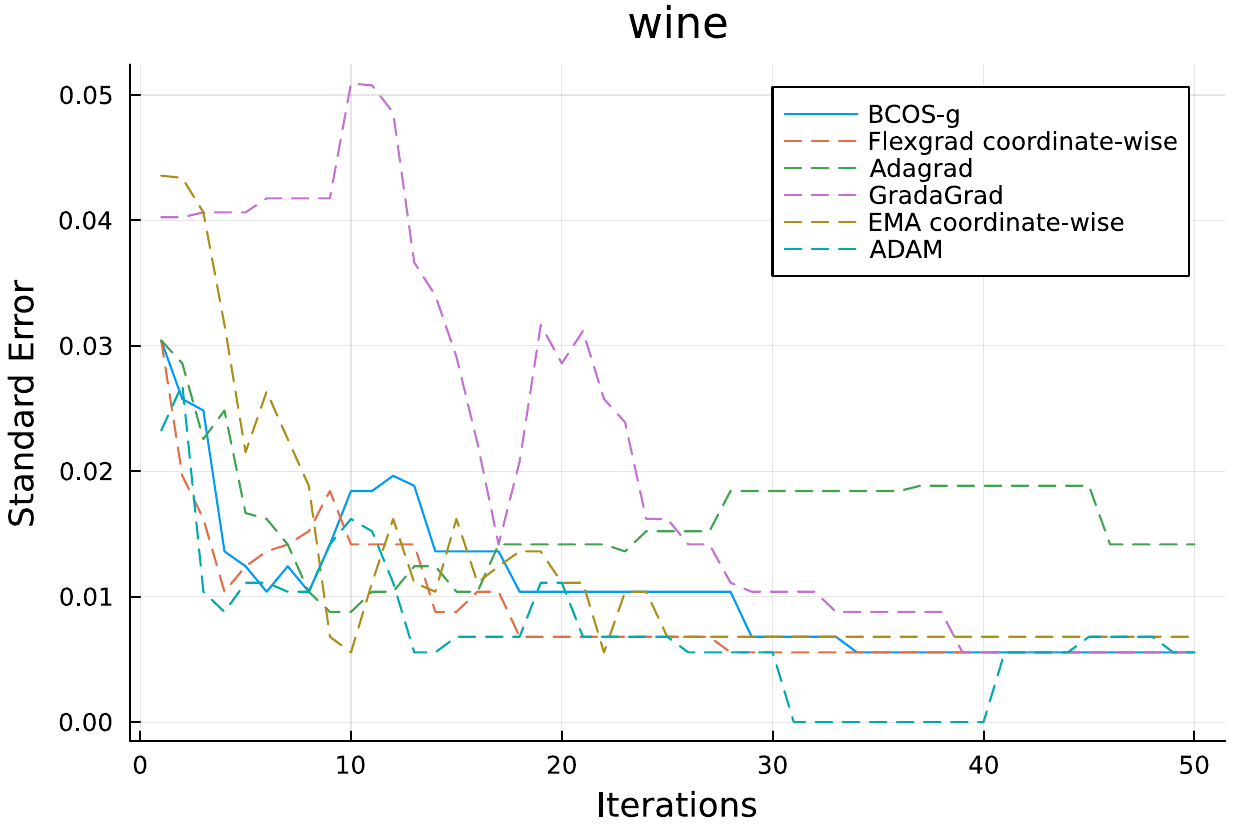}
     \end{tabular}
 \caption{Evolution of the standard deviation of the accuracy along iterations on 6 data sets for BCOS, Flexgrad coordinate-wise, Adagrad, Gradagrad, EMA coordinate-wise, ADAM.}
\label{empprobmodels6}
\end{figure}

\begin{figure}[htbp]
    \centering
    \begin{tabular}{ccc}
        \includegraphics[scale=\figscale]{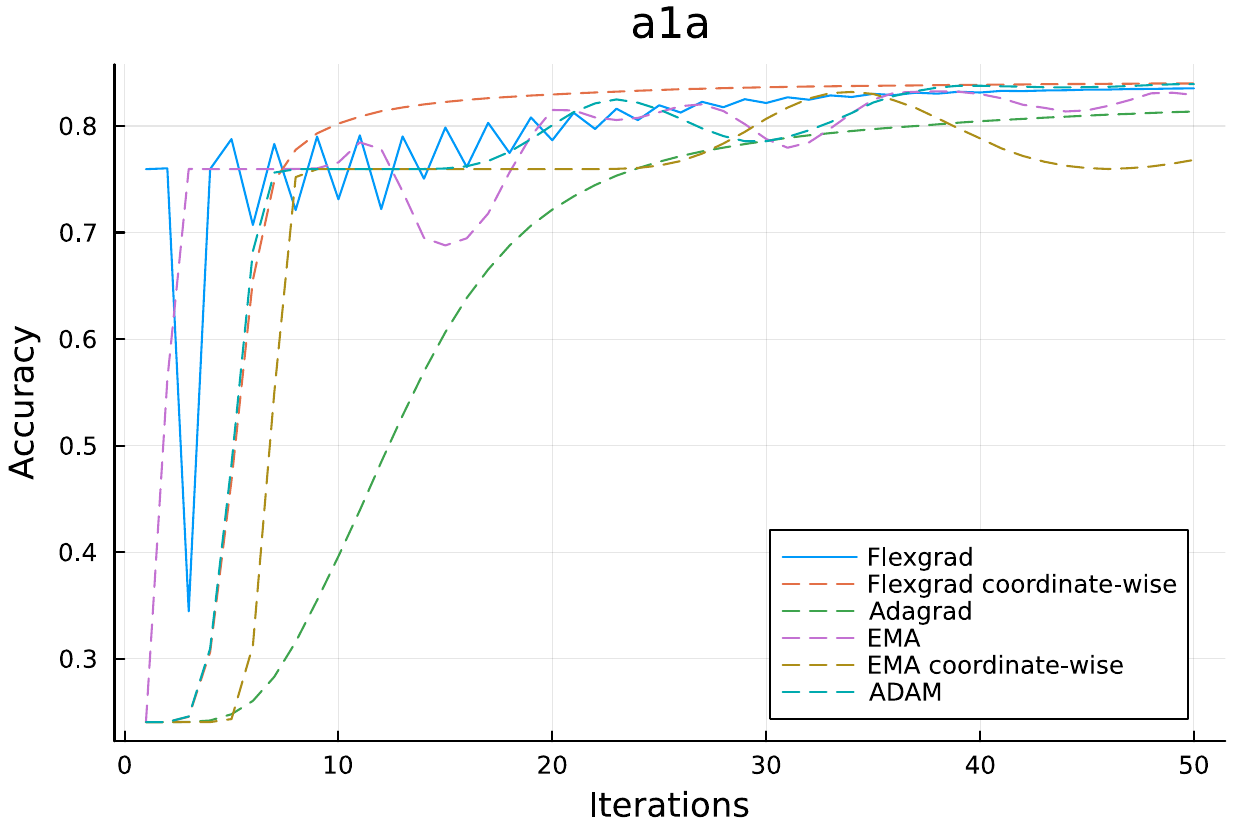}
        & \includegraphics[scale=\figscale]{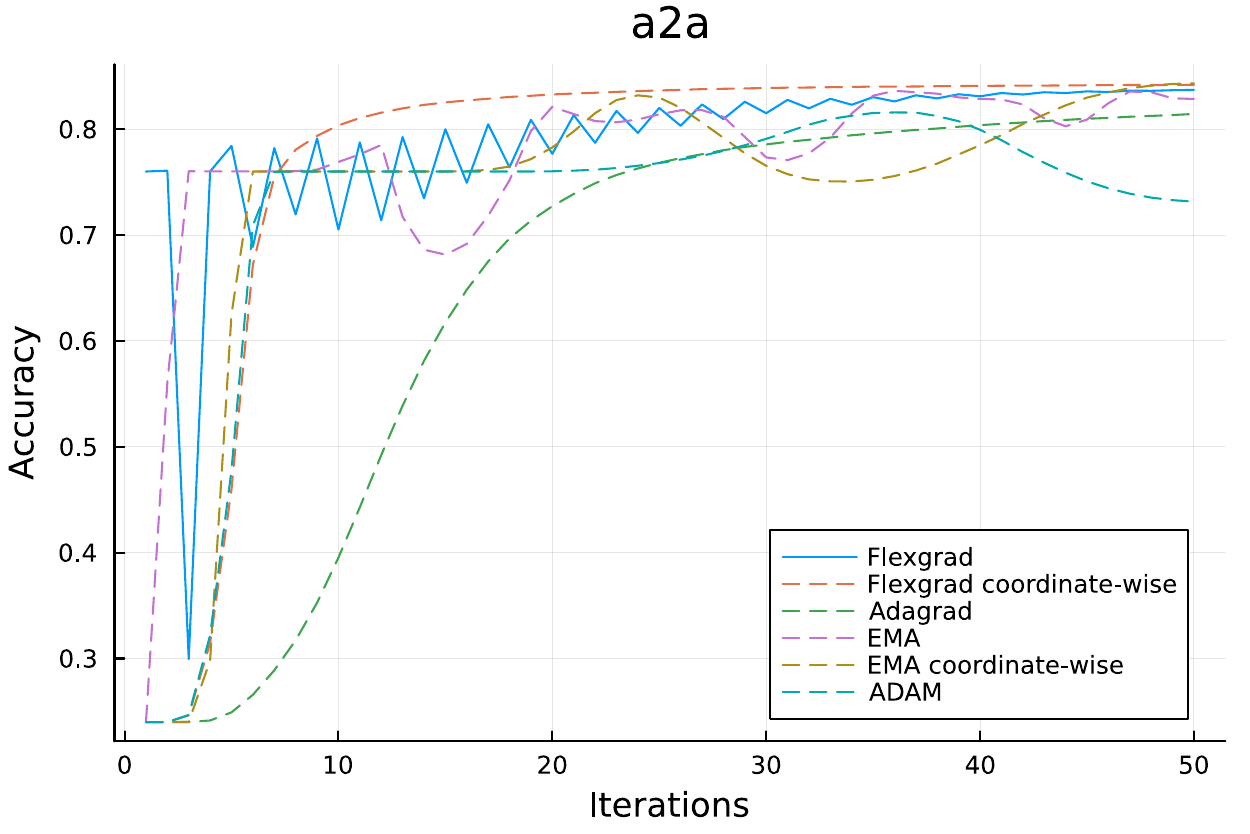}&
        \includegraphics[scale=\figscale]{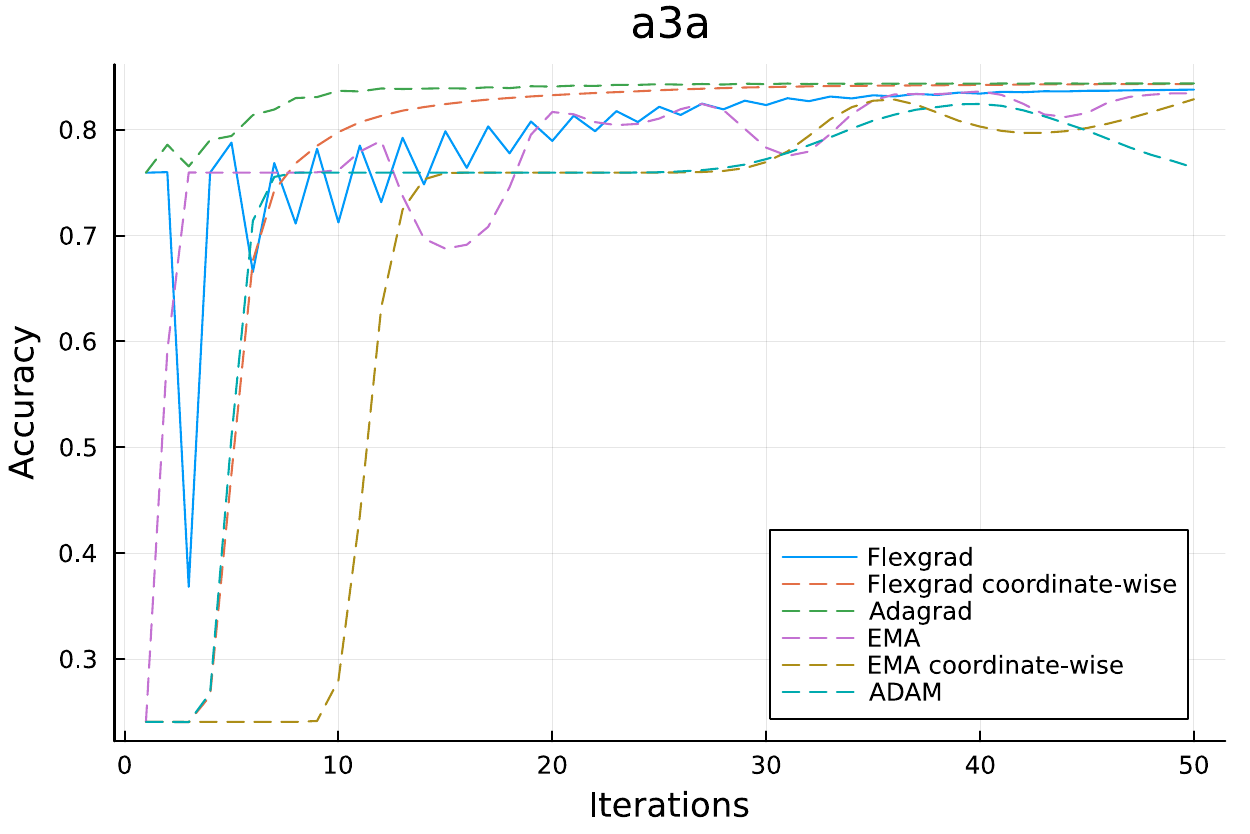}\\ 
        \includegraphics[scale=\figscale]{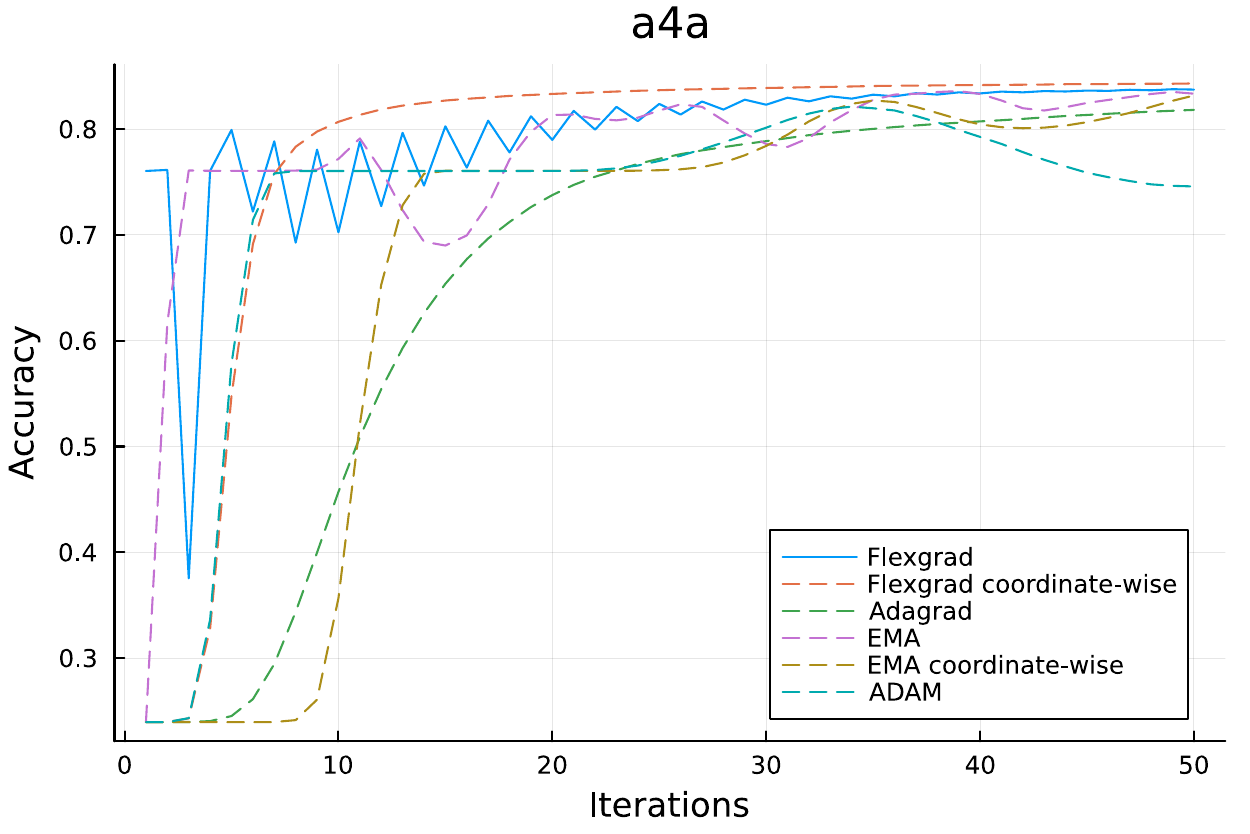}
        & \includegraphics[scale=\figscale]{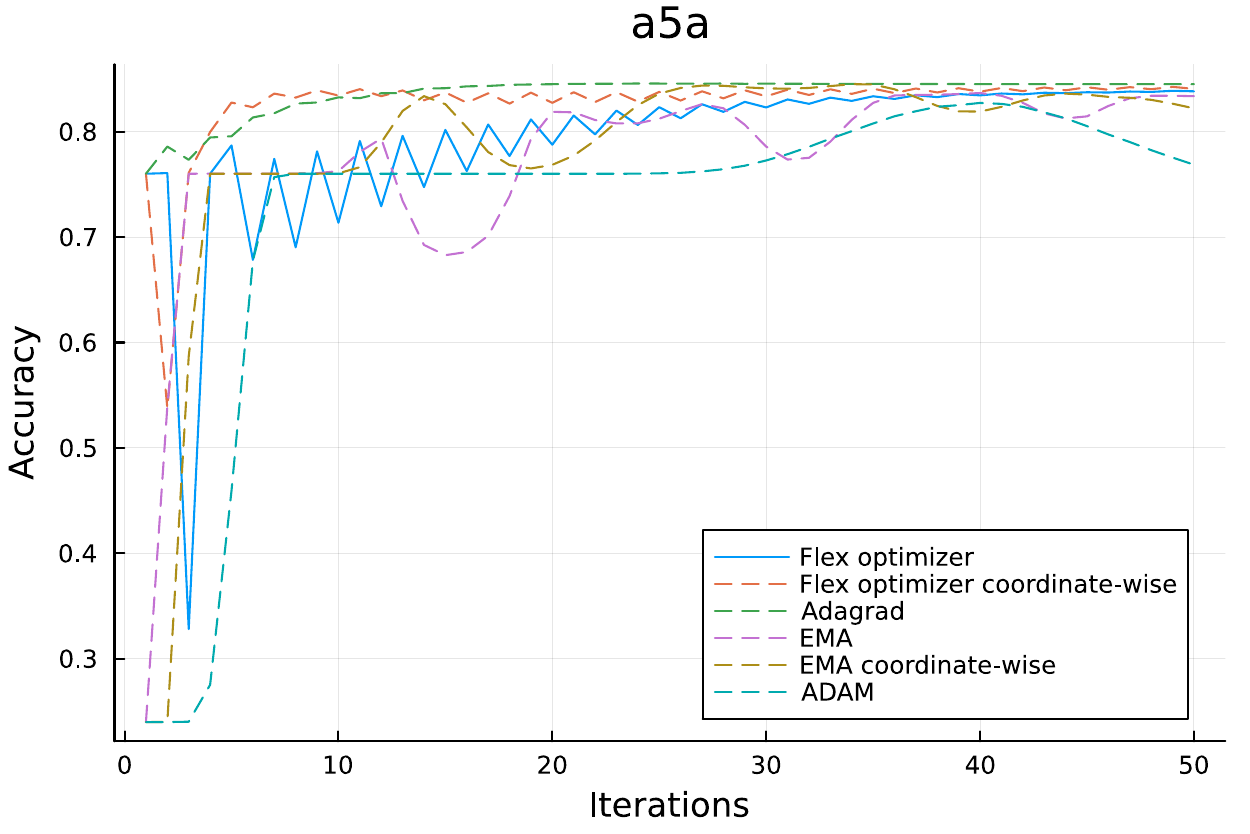}&
        \includegraphics[scale=\figscale]{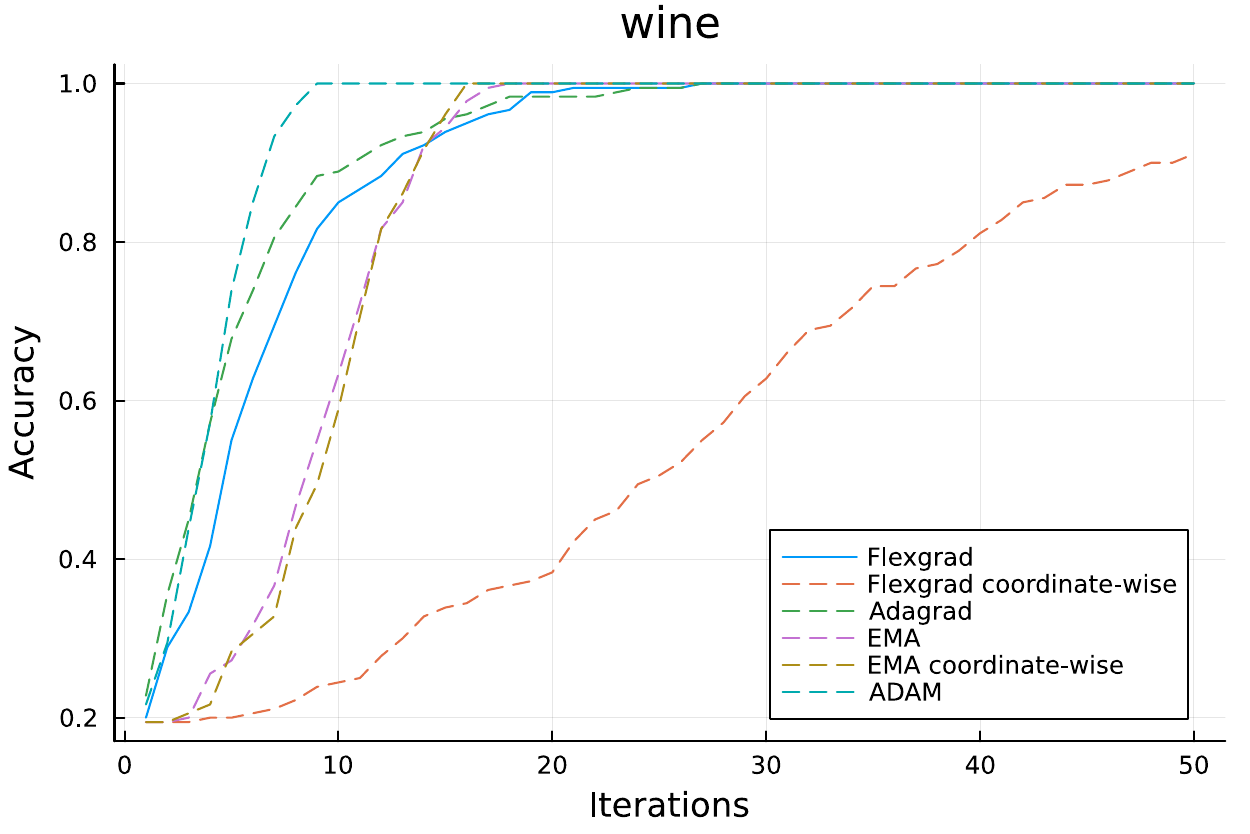}\\ 
        \includegraphics[scale=\figscale]{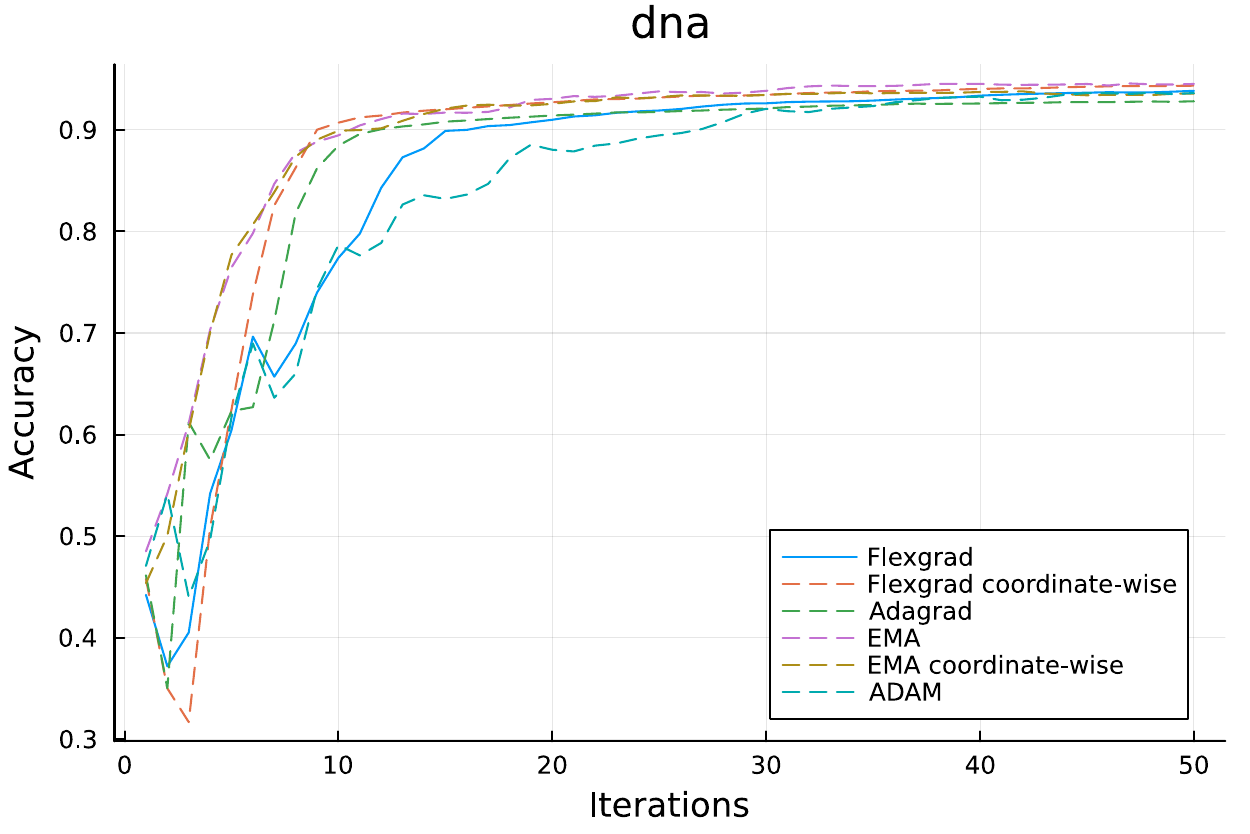}
        & \includegraphics[scale=\figscale]{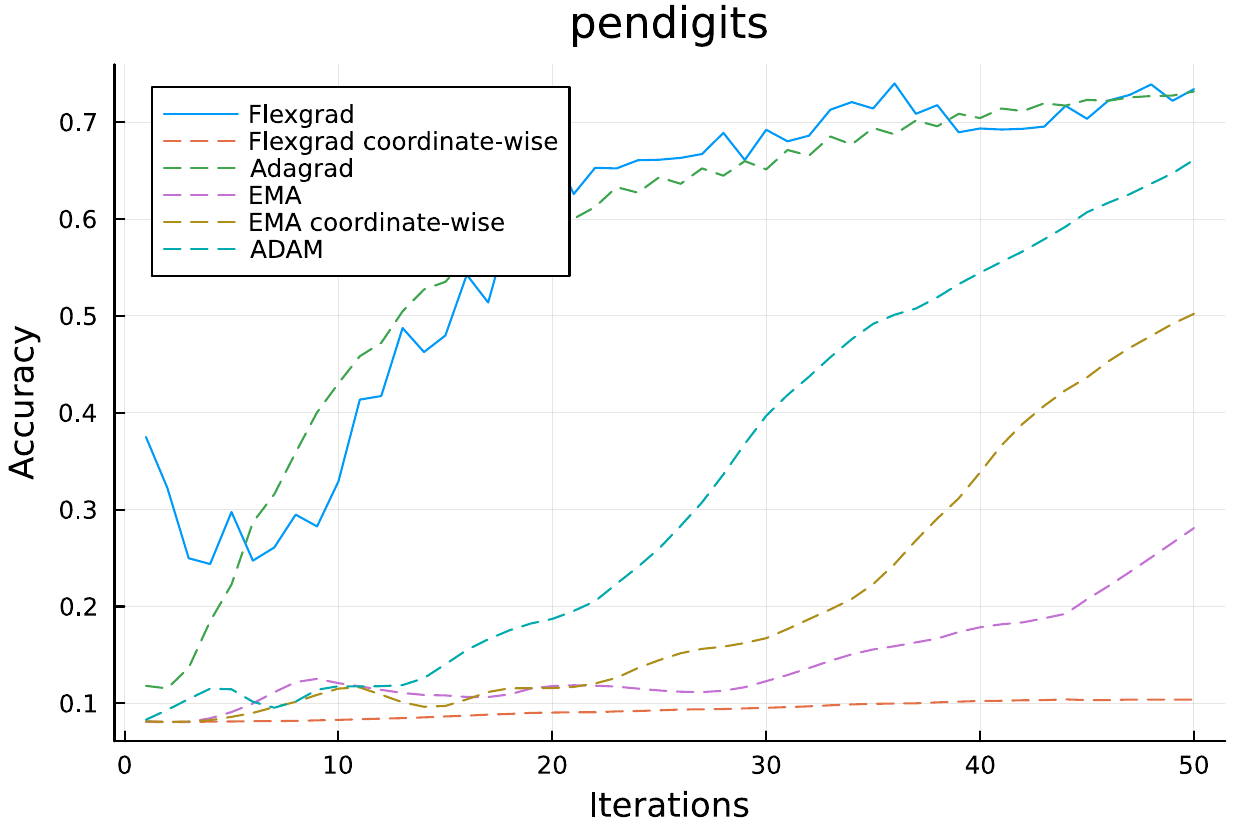}&
        \includegraphics[scale=\figscale]{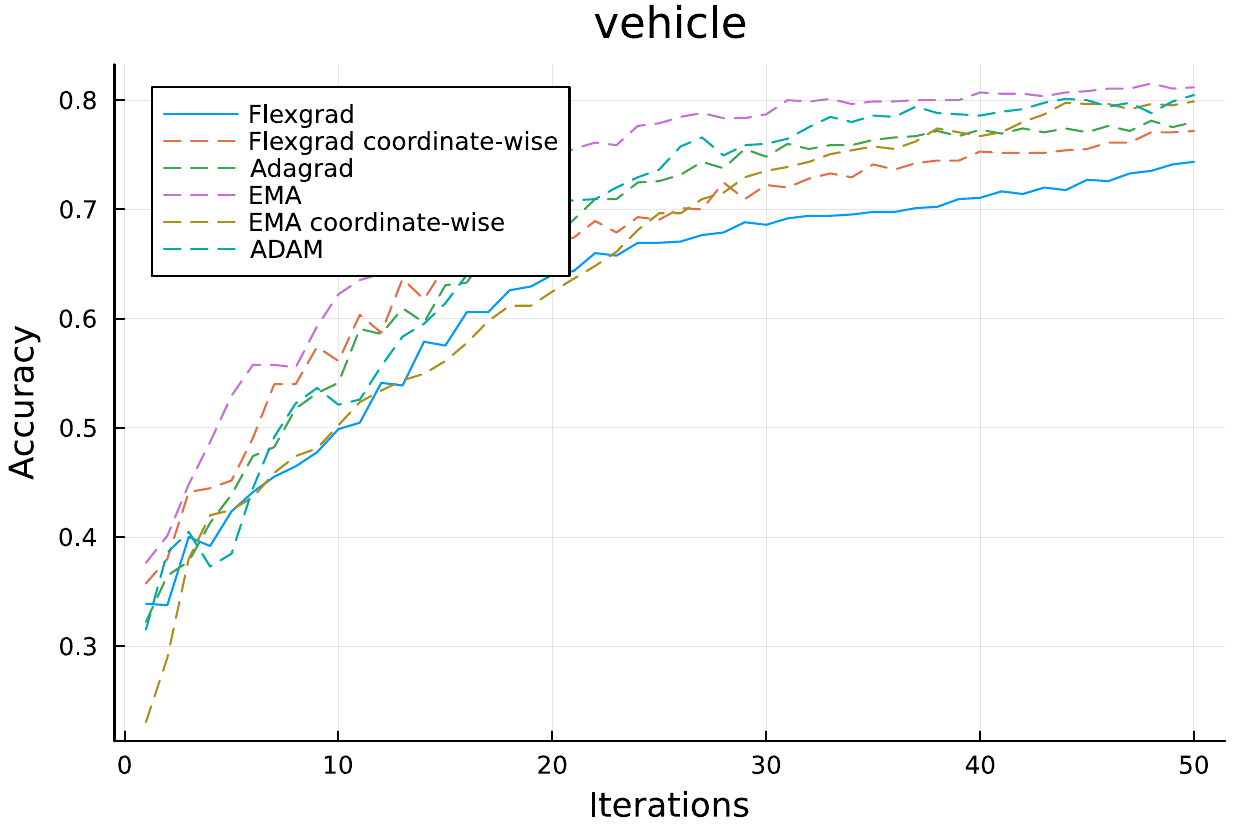}\\
        \includegraphics[scale=\figscale]{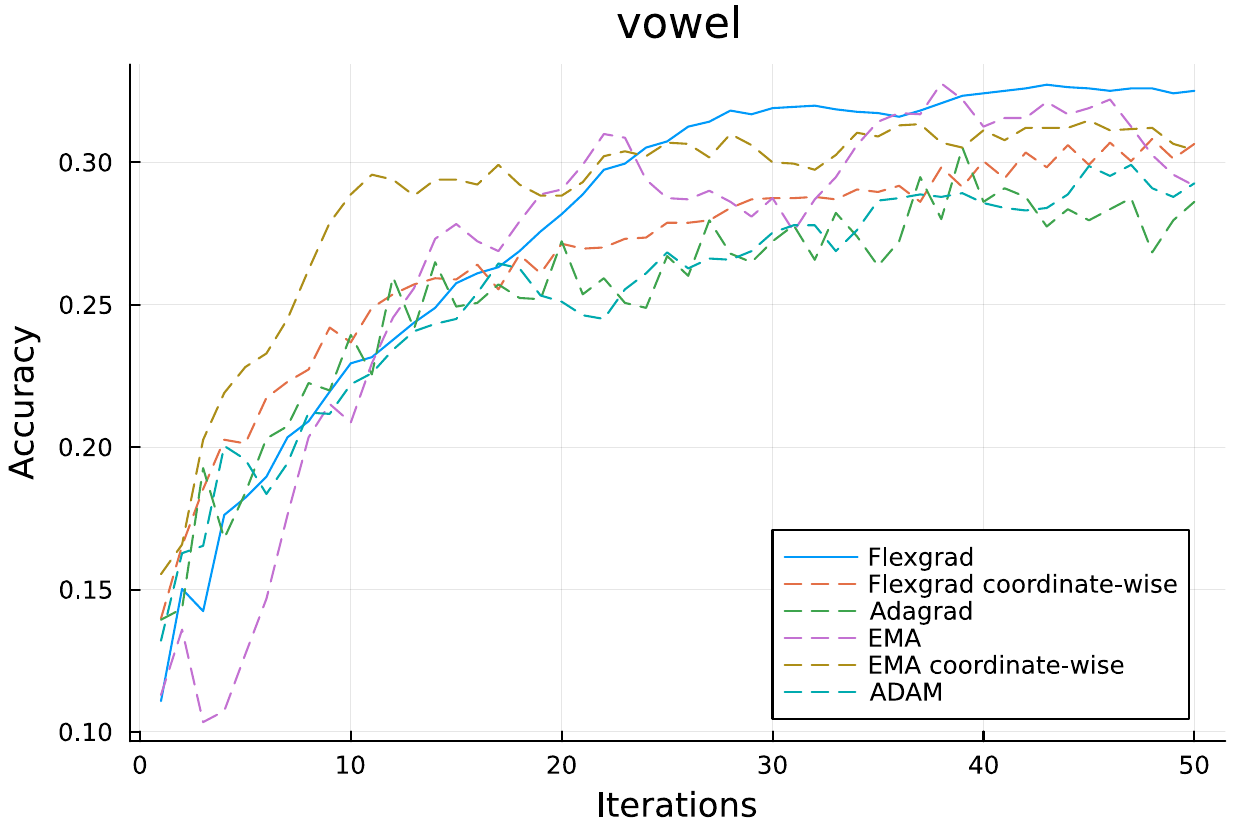}
        & \includegraphics[scale=\figscale]{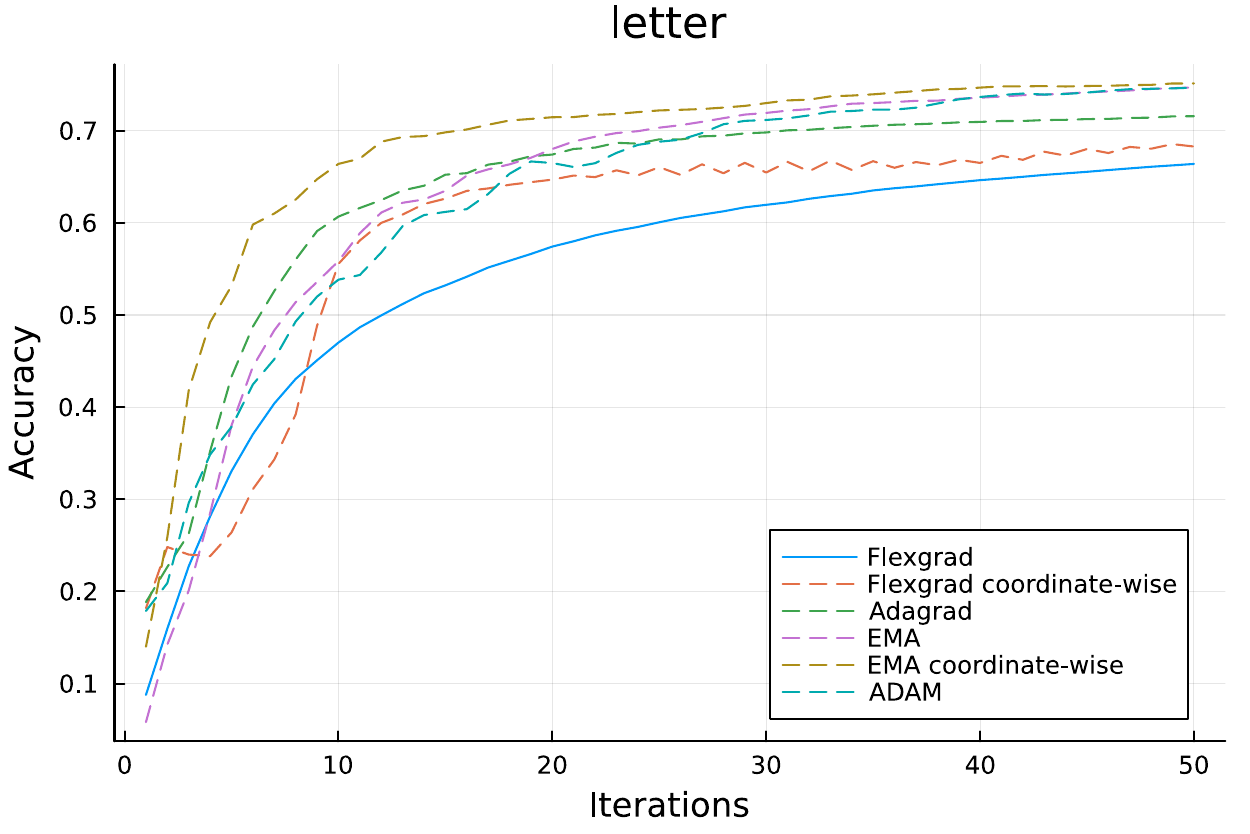}&
        \includegraphics[scale=\figscale]{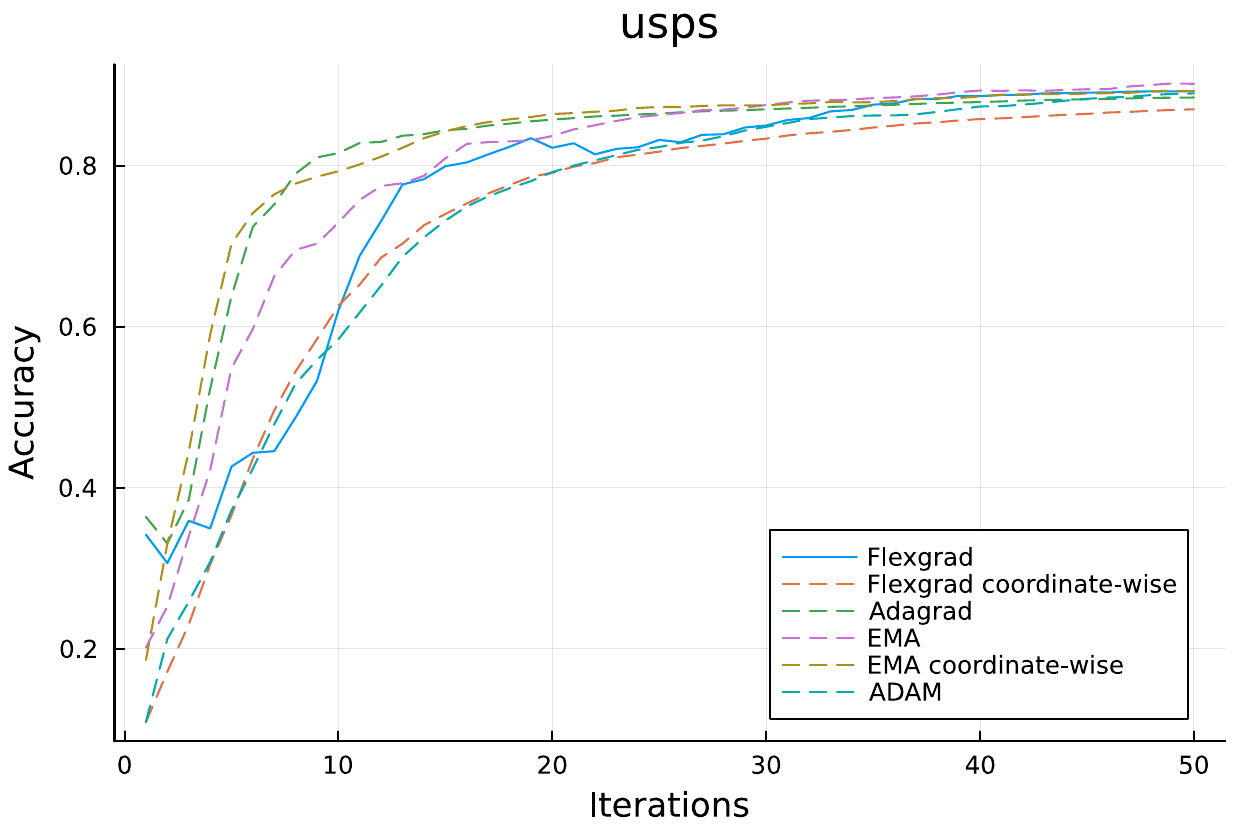}
     \end{tabular}
 \caption{Evolution of the accuracy along iterations on 12 data sets for Flexgrad, Flexgrad coordinate-wise, Adagrad, EMA, EMA coordinate-wise, ADAM.}
\label{empprobmodel2}
\end{figure}

\begin{figure}[htbp]
    \centering
    \begin{tabular}{ccc}
  \includegraphics[scale=\figscale]{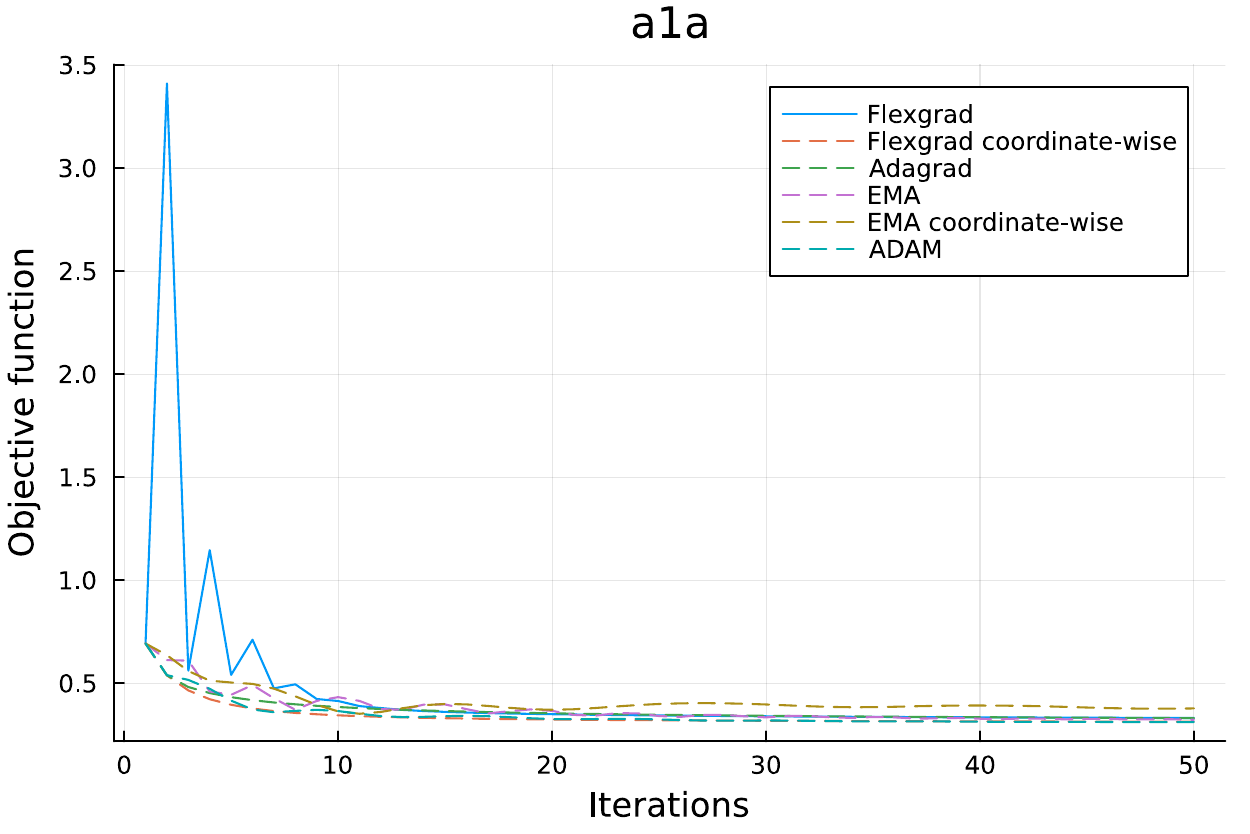}
        & \includegraphics[scale=\figscale]{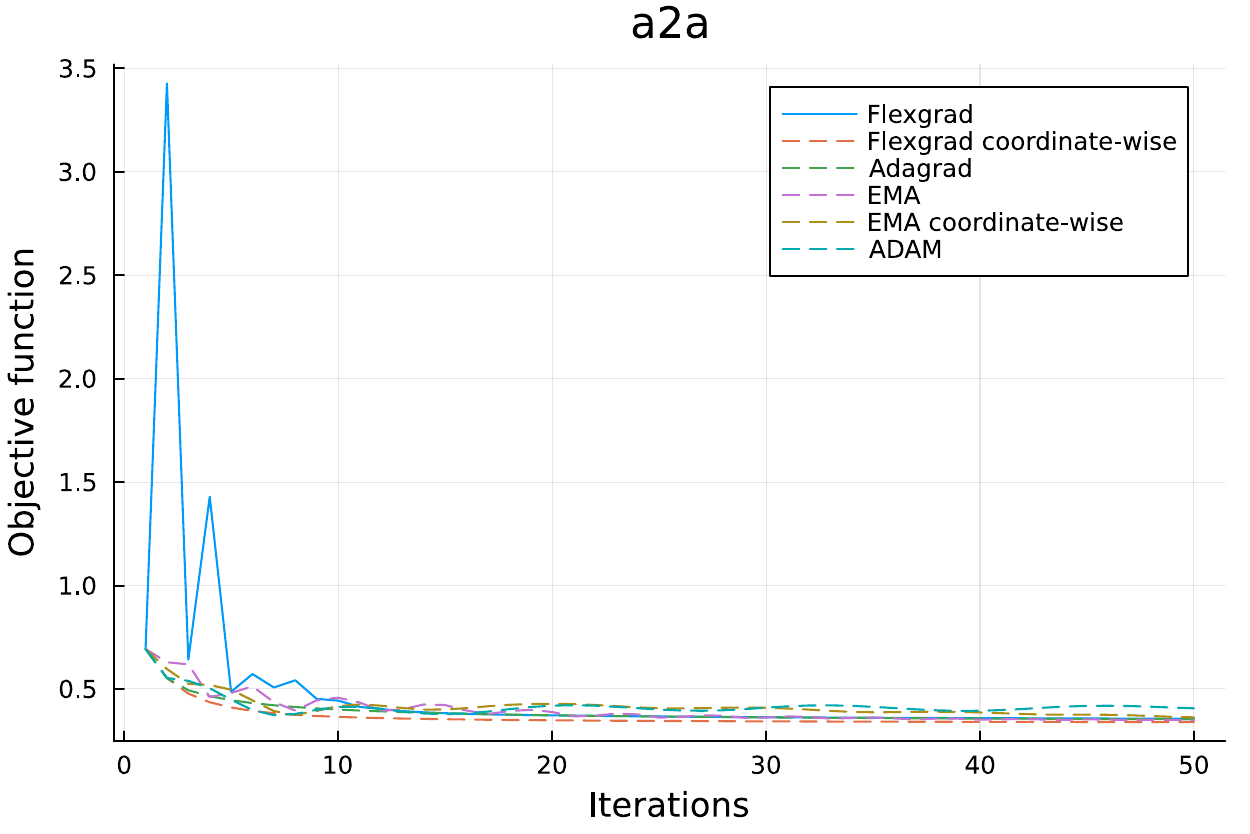}&
        \includegraphics[scale=\figscale]{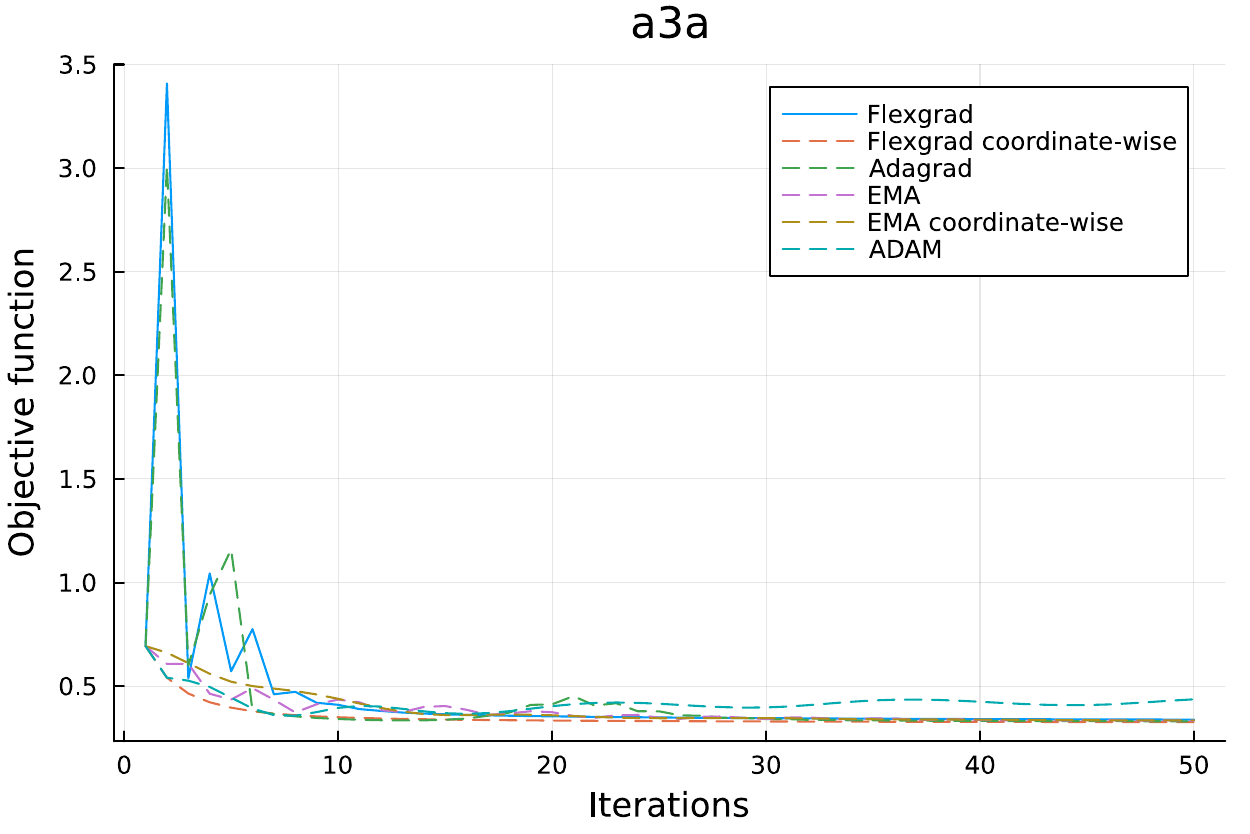}\\ 
        \includegraphics[scale=\figscale]{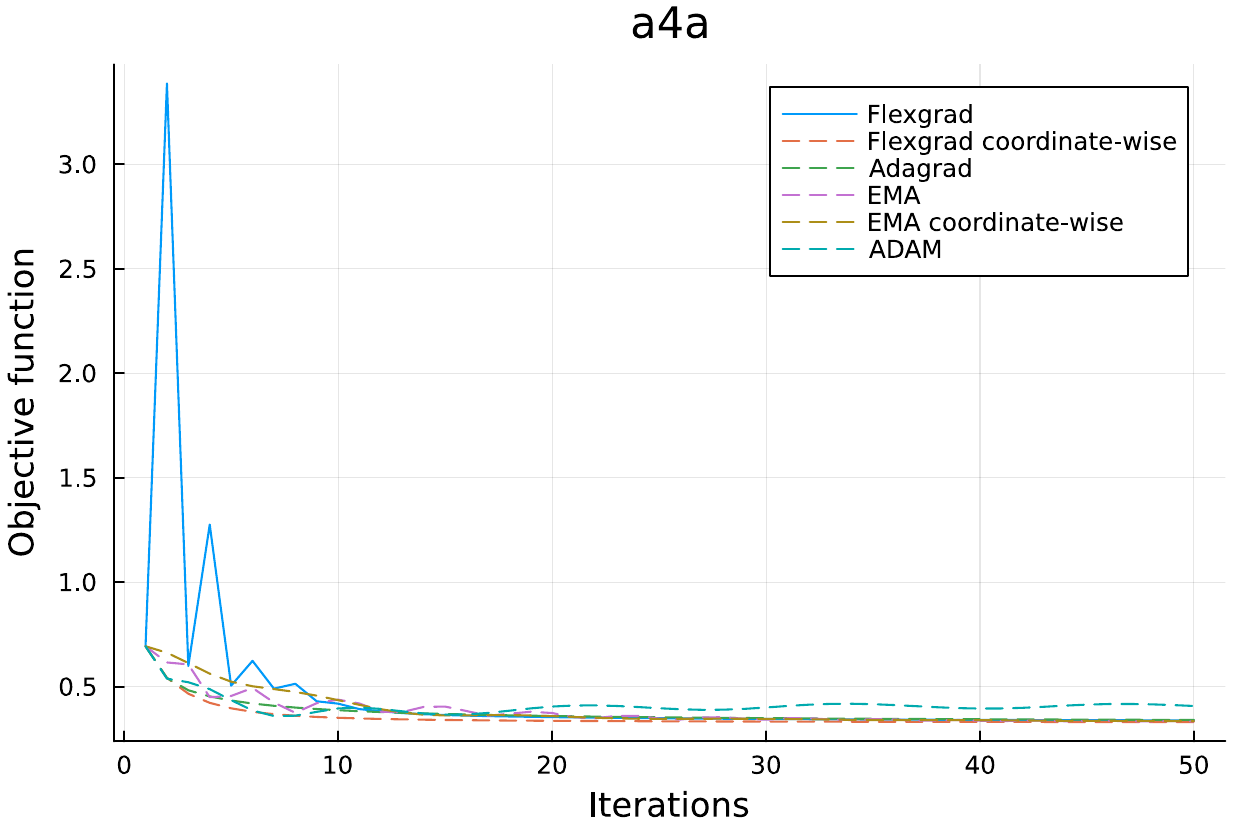}
        & \includegraphics[scale=\figscale]{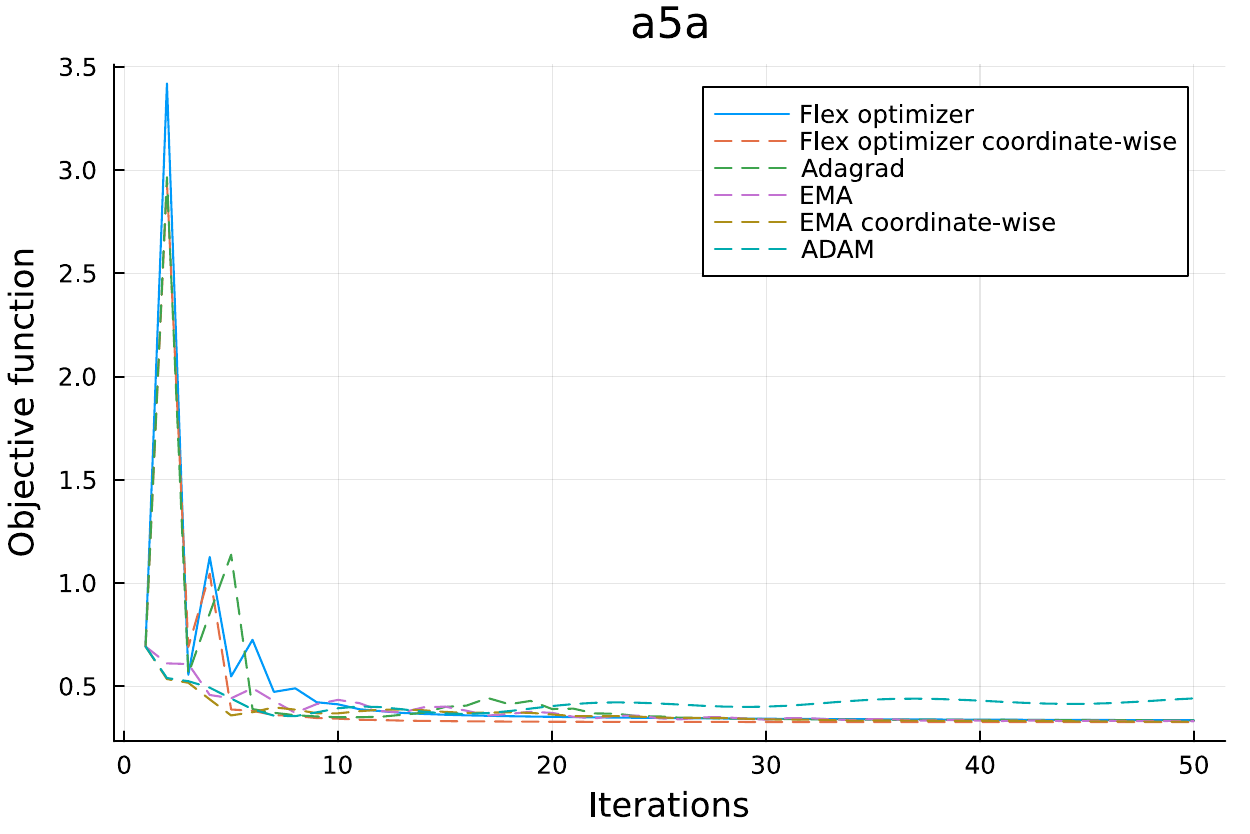}&
        \includegraphics[scale=\figscale]{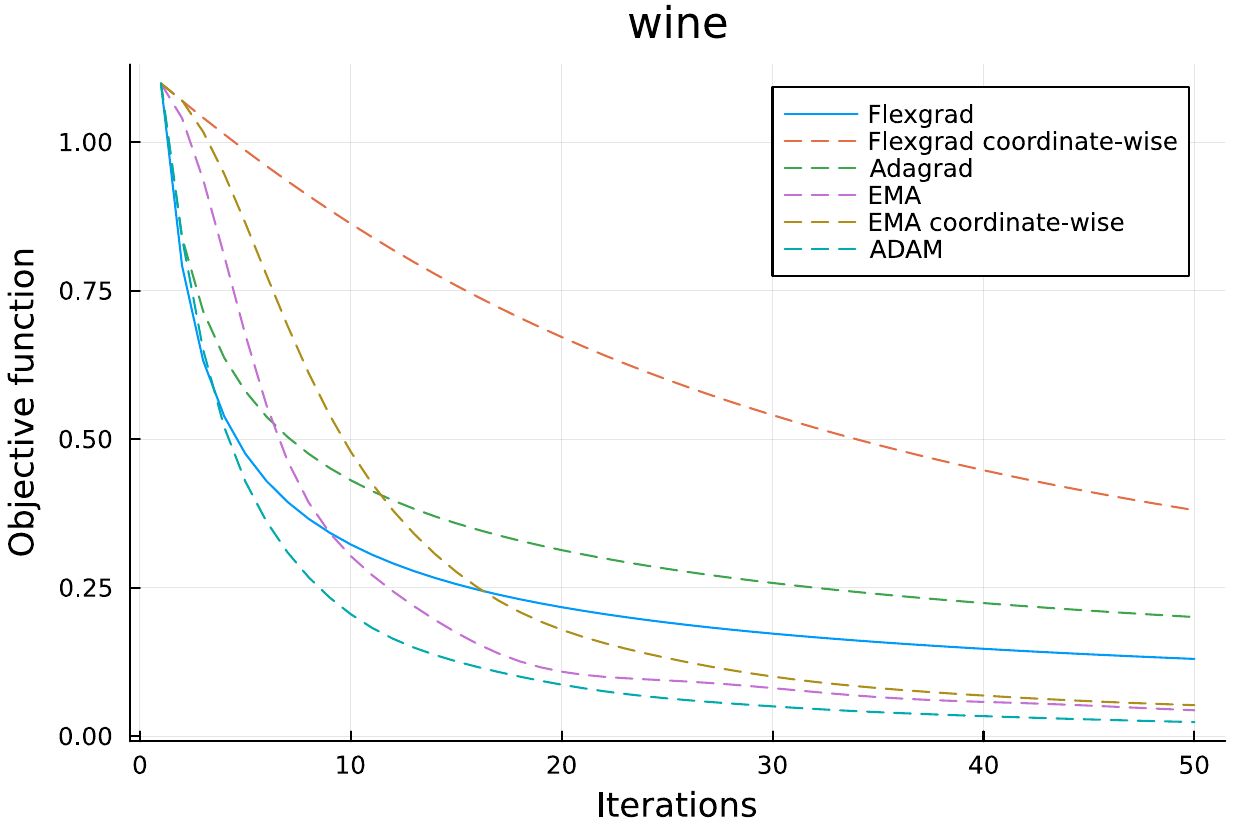}\\ 
        \includegraphics[scale=\figscale]{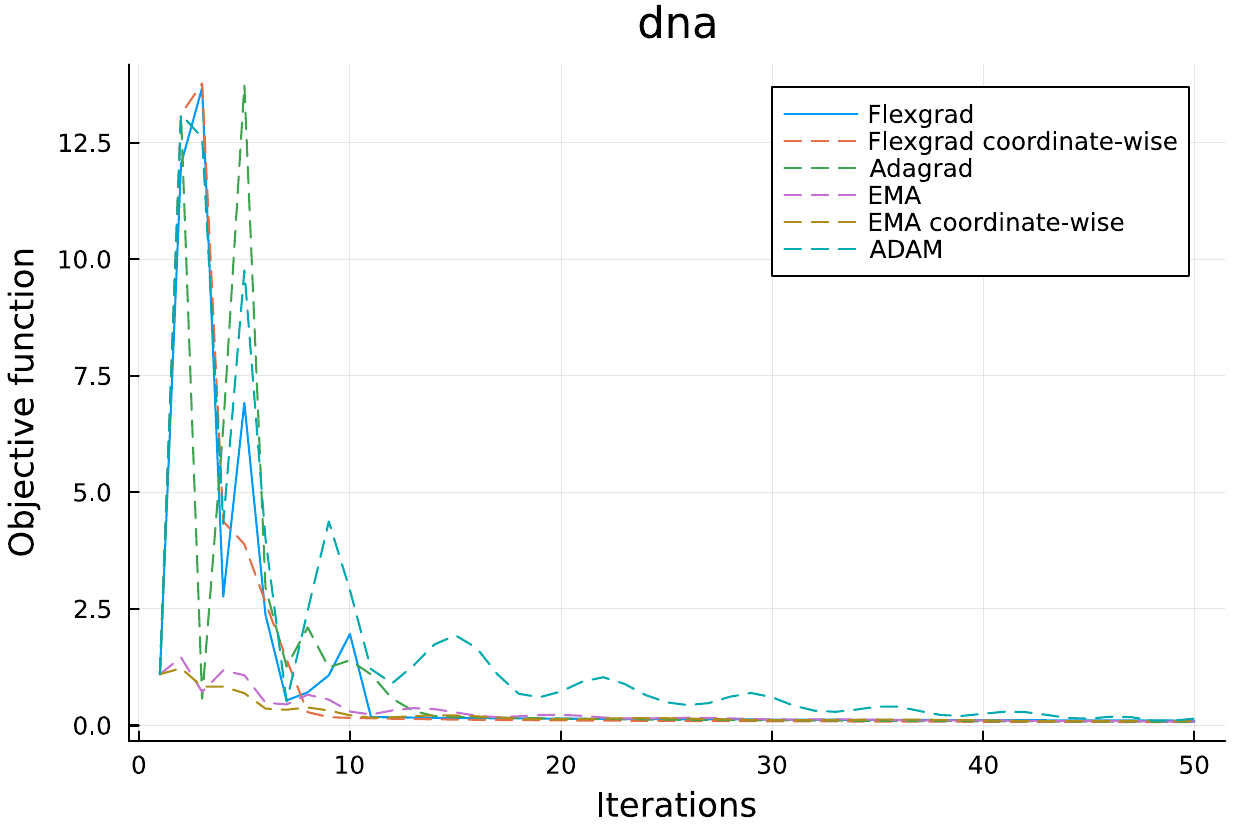}
        & \includegraphics[scale=\figscale]{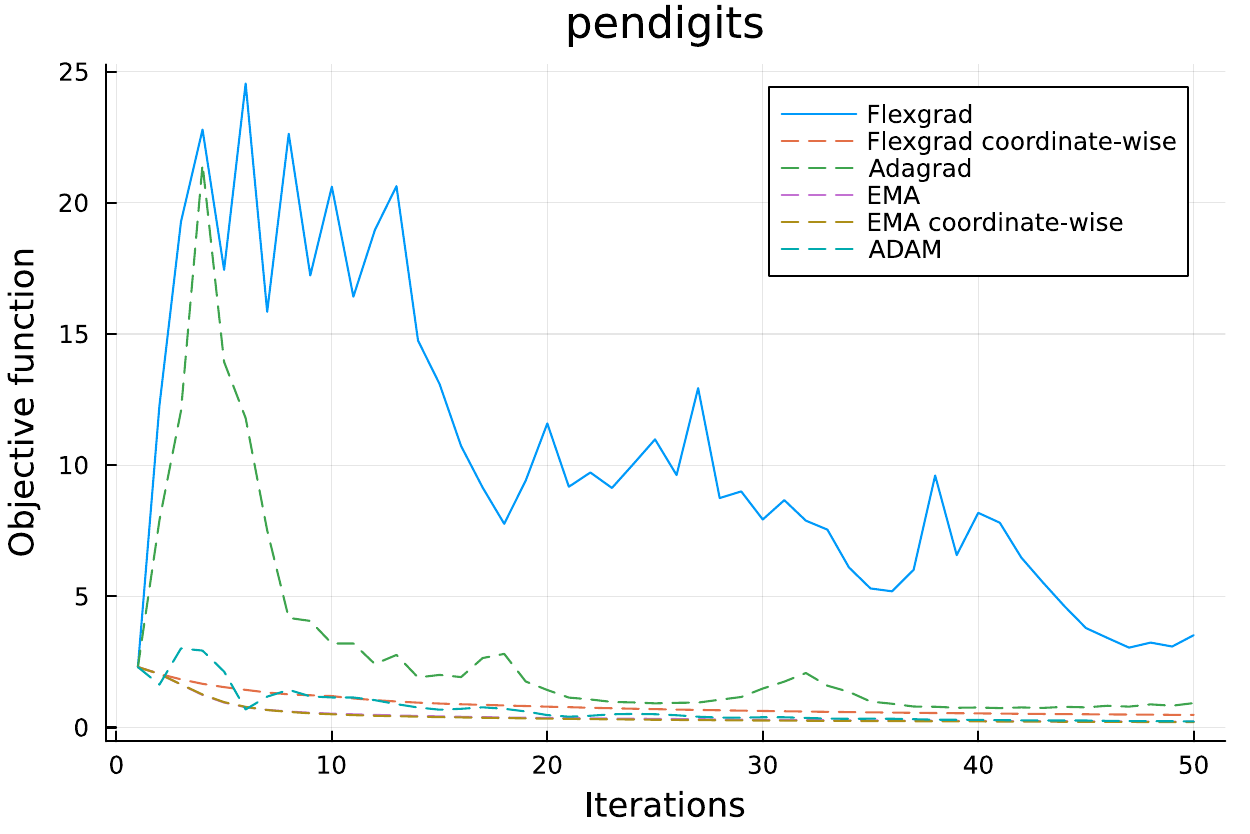}&
        \includegraphics[scale=\figscale]{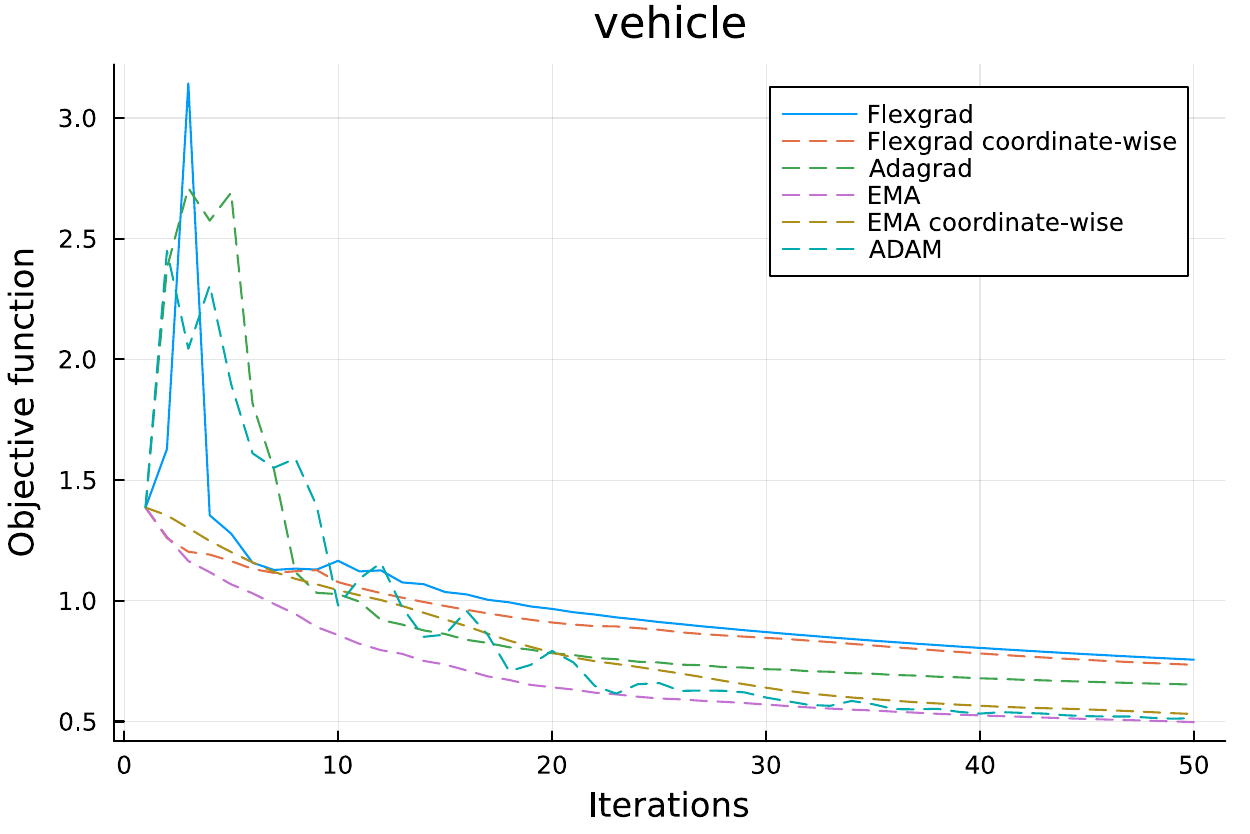}\\
        \includegraphics[scale=\figscale]{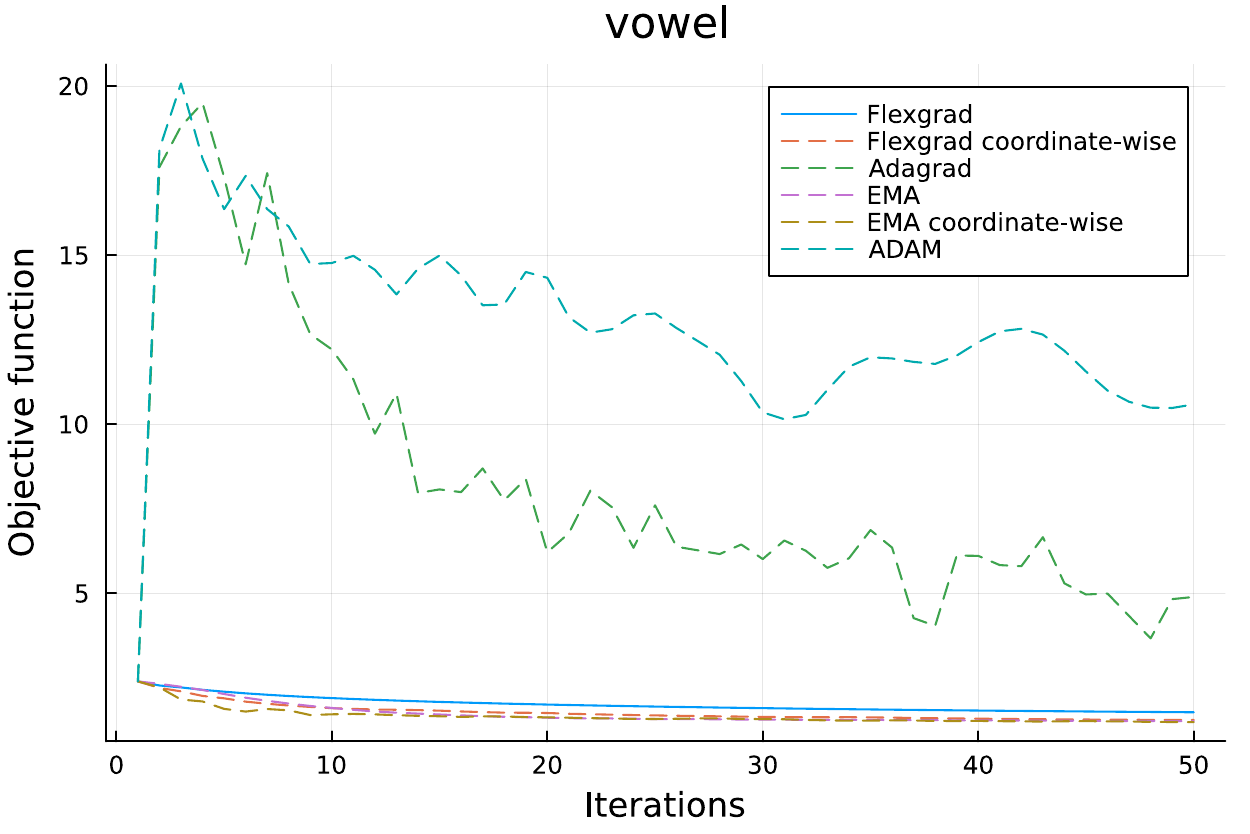}
        & \includegraphics[scale=\figscale]{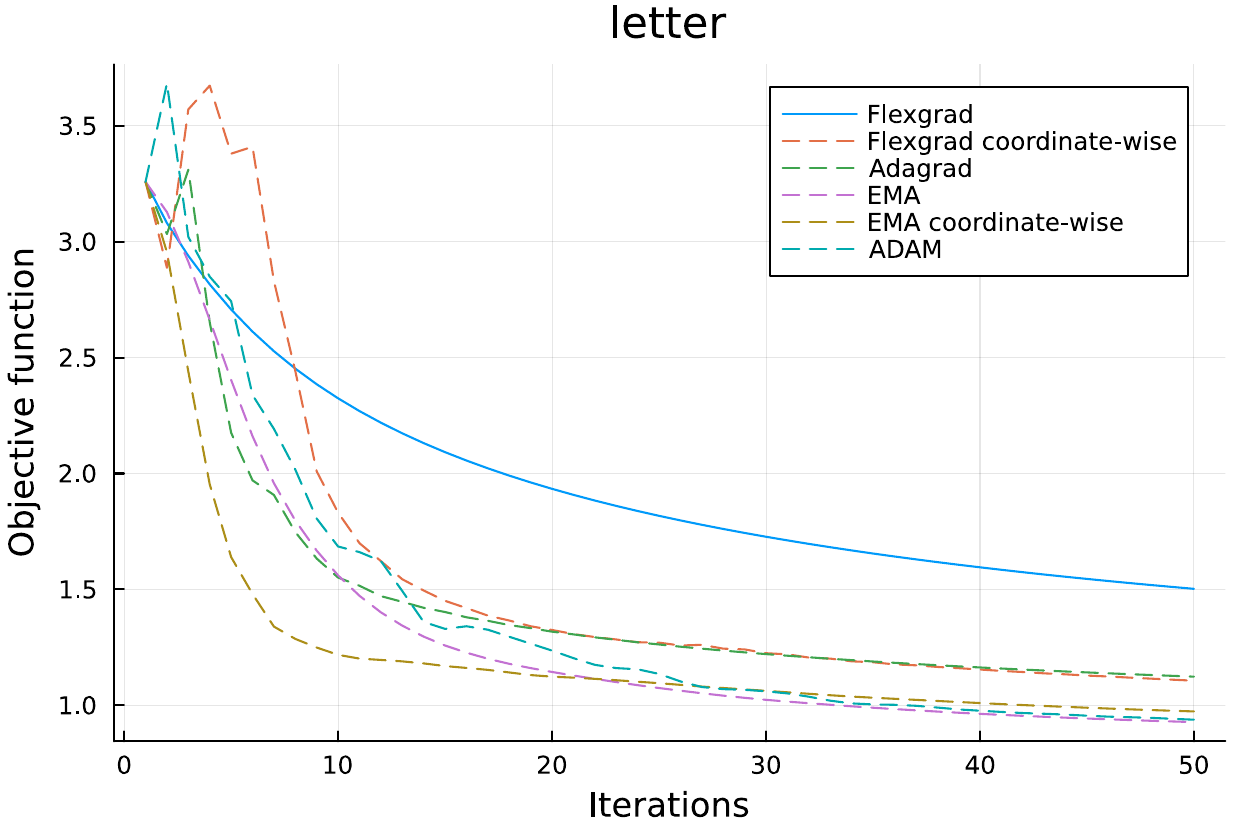}&
        \includegraphics[scale=\figscale]{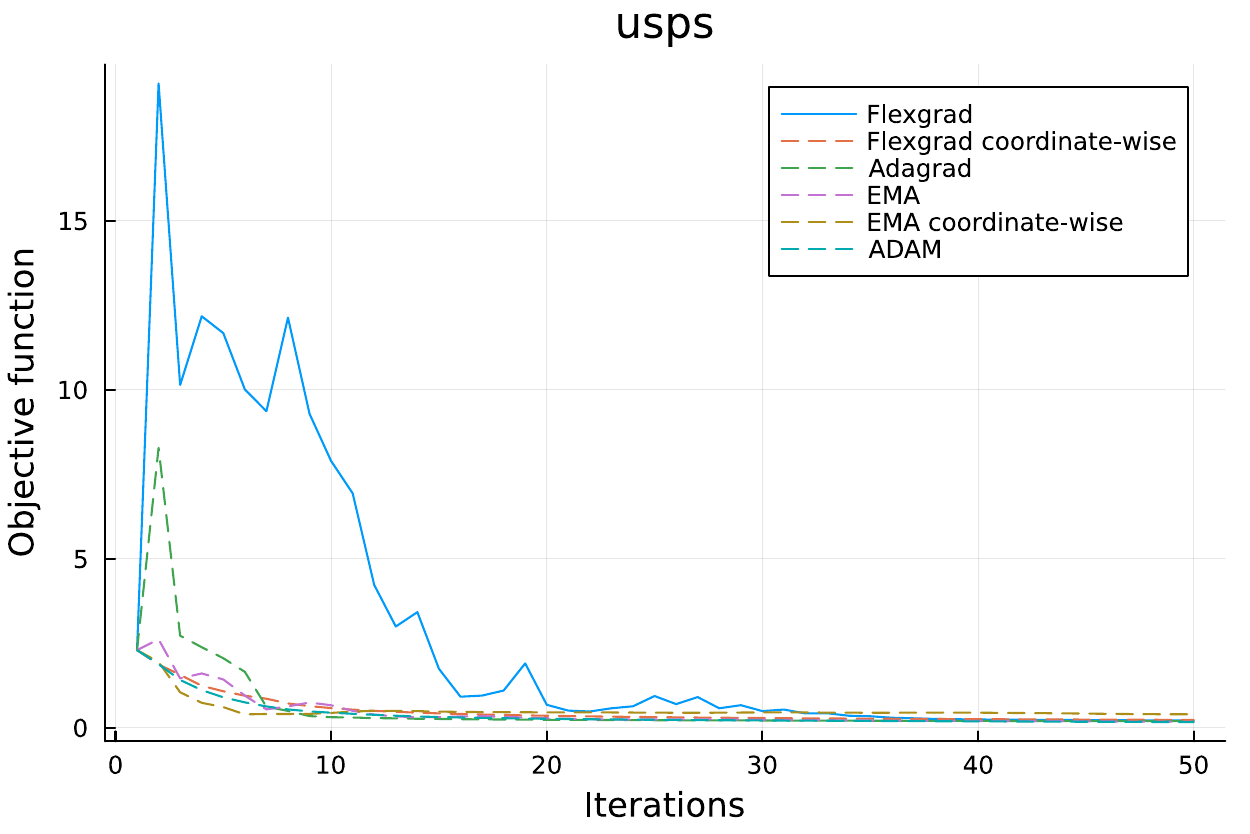}
     \end{tabular}
 \caption{Evolution of the objective function along iterations on 12 data sets for Flexgrad, Flexgrad coordinate-wise, Adagrad, EMA, EMA coordinate-wise, ADAM.}
\label{empprobmodelo}
\end{figure}

\begin{figure}[htbp]
    \centering
    \begin{tabular}{ccc}
  \includegraphics[scale=\figscale]{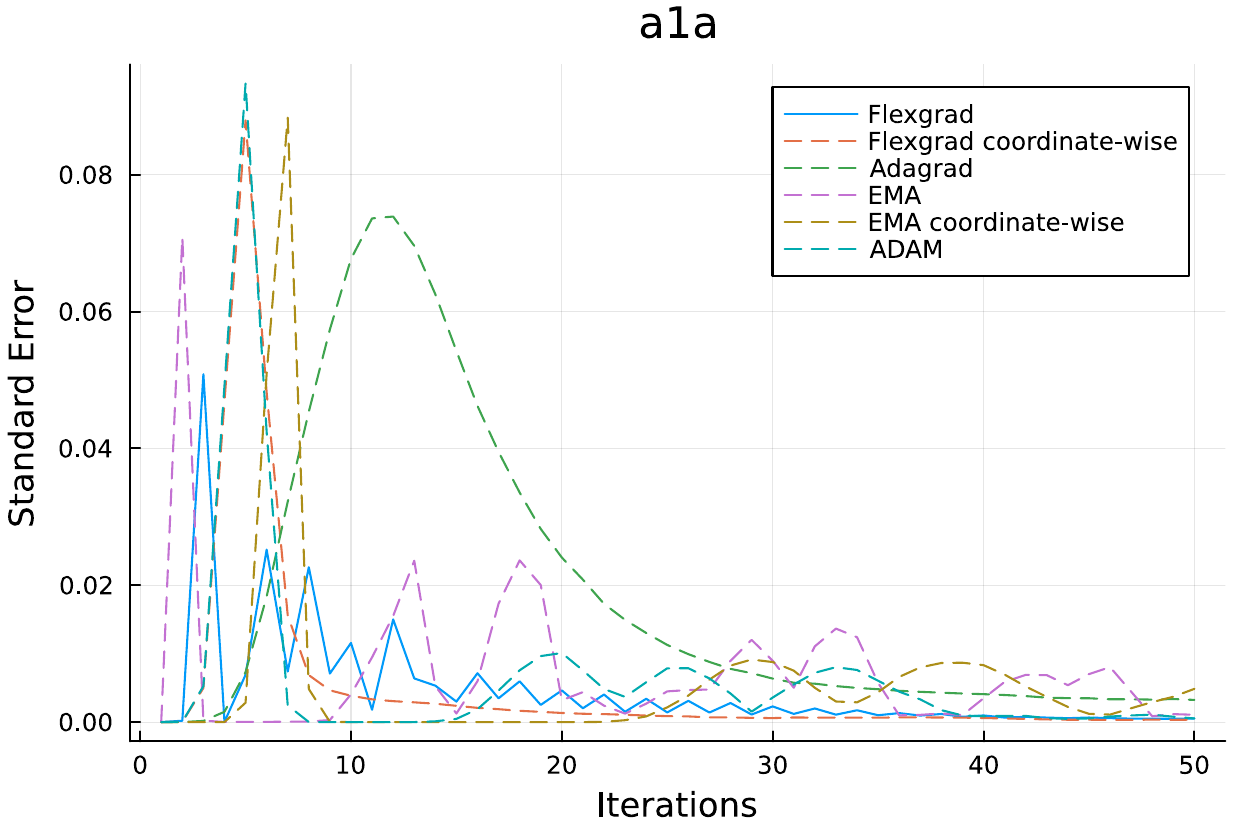}
        & \includegraphics[scale=\figscale]{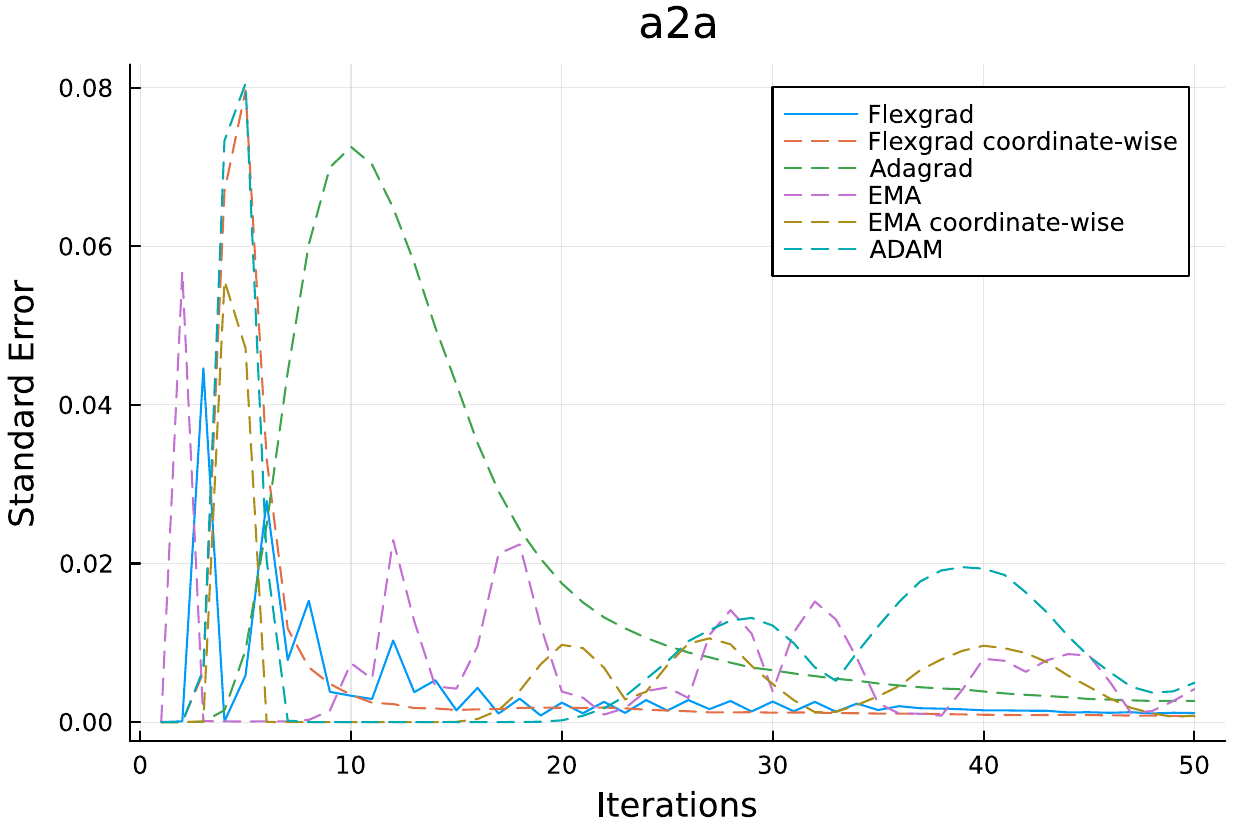}&
        \includegraphics[scale=\figscale]{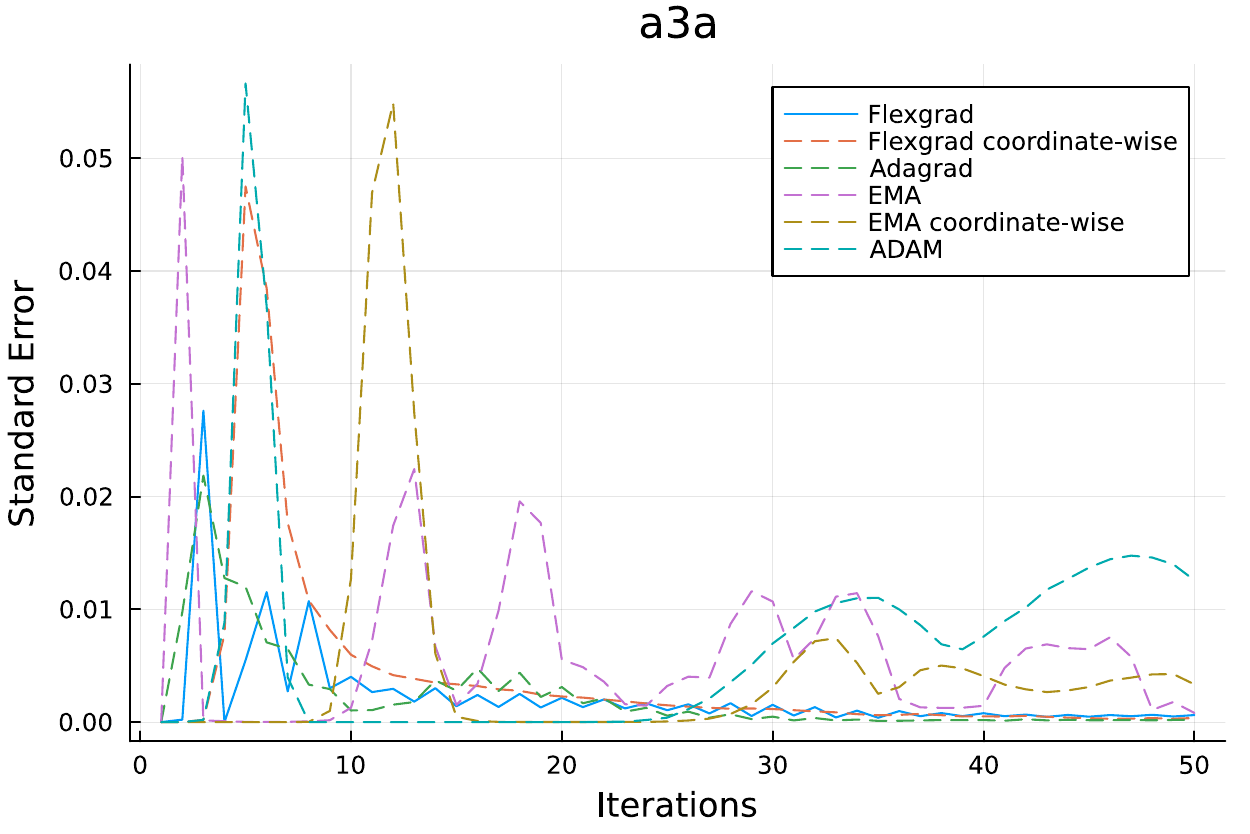}\\ 
        \includegraphics[scale=\figscale]{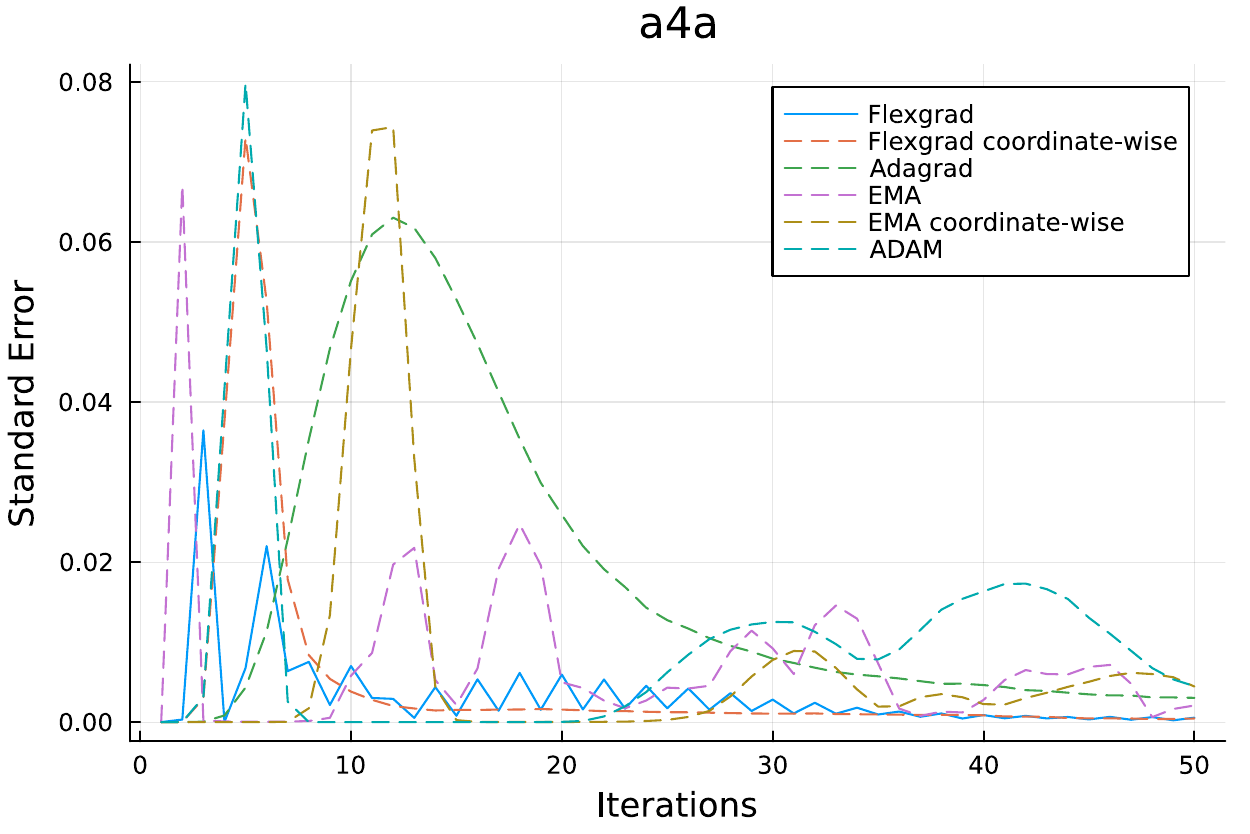}
        & \includegraphics[scale=\figscale]{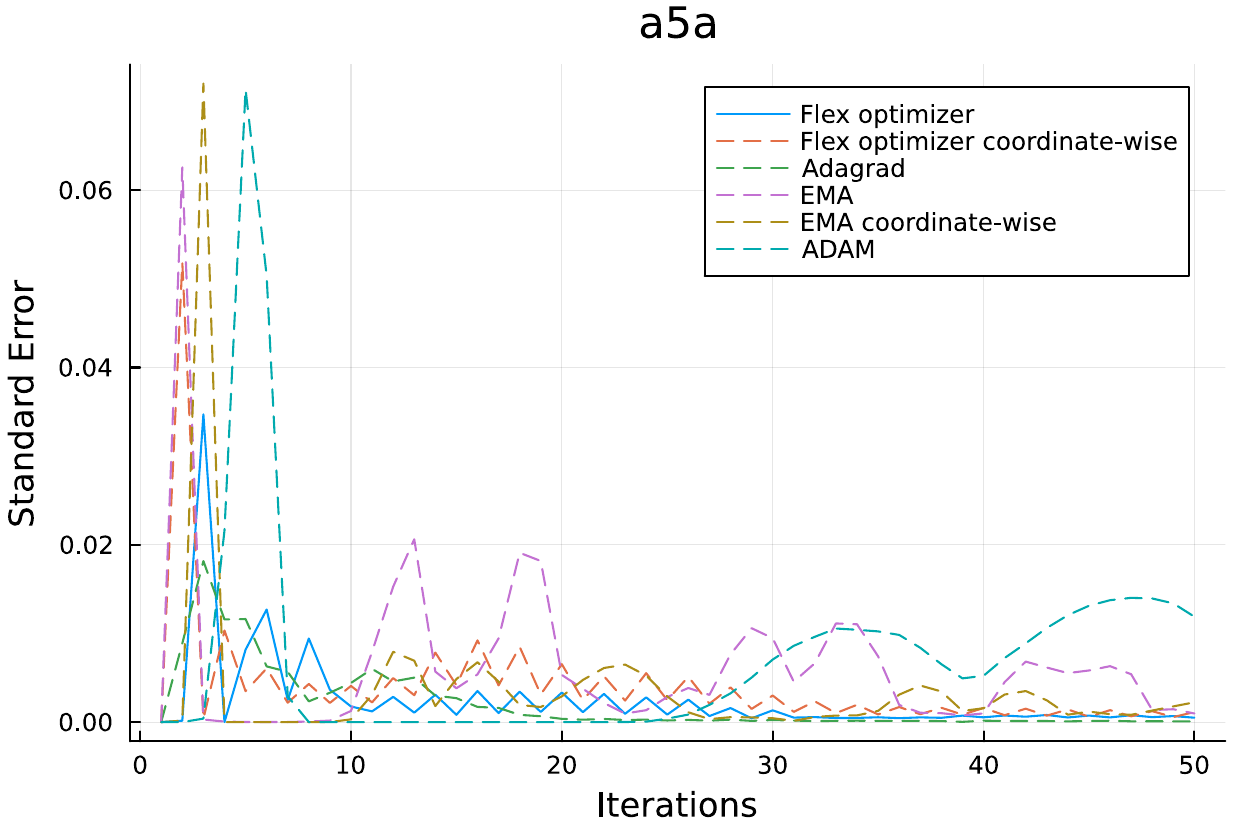}&
        \includegraphics[scale=\figscale]{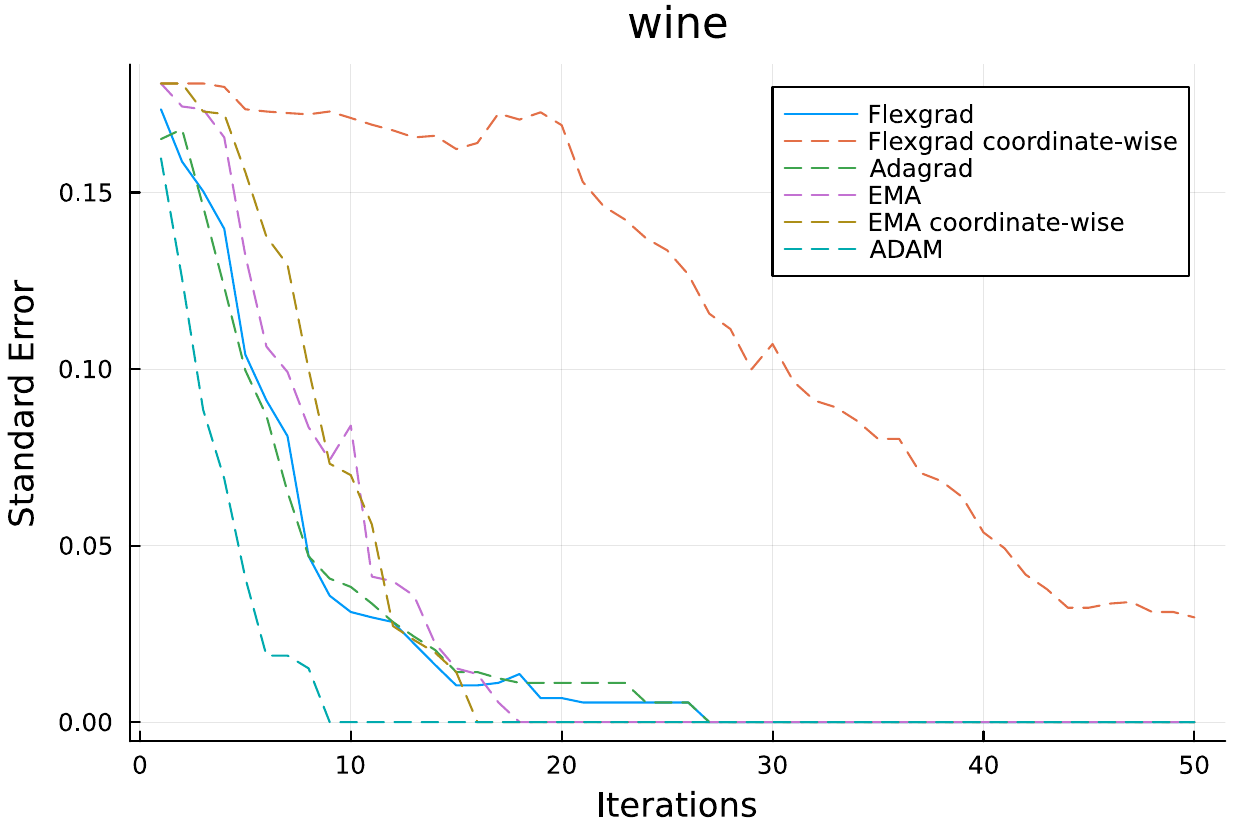}\\ 
        \includegraphics[scale=\figscale]{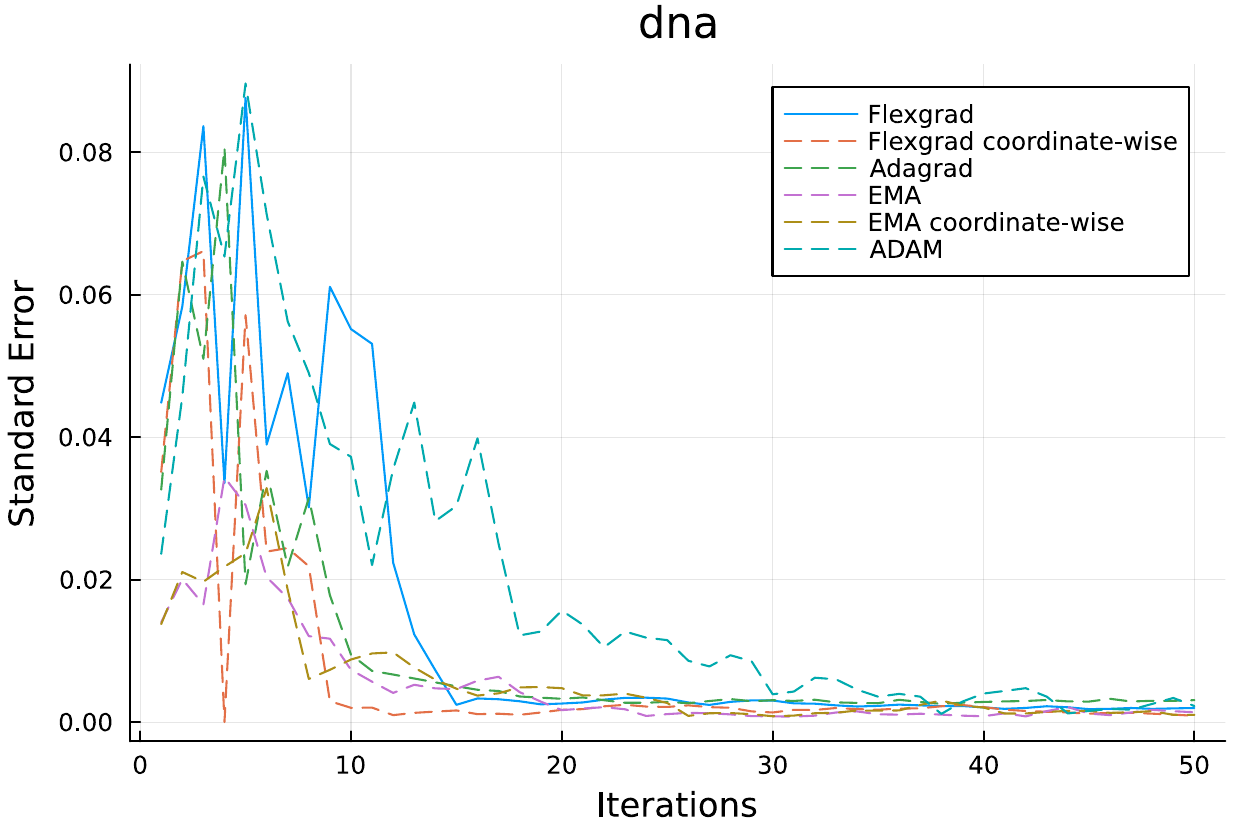}
        & \includegraphics[scale=\figscale]{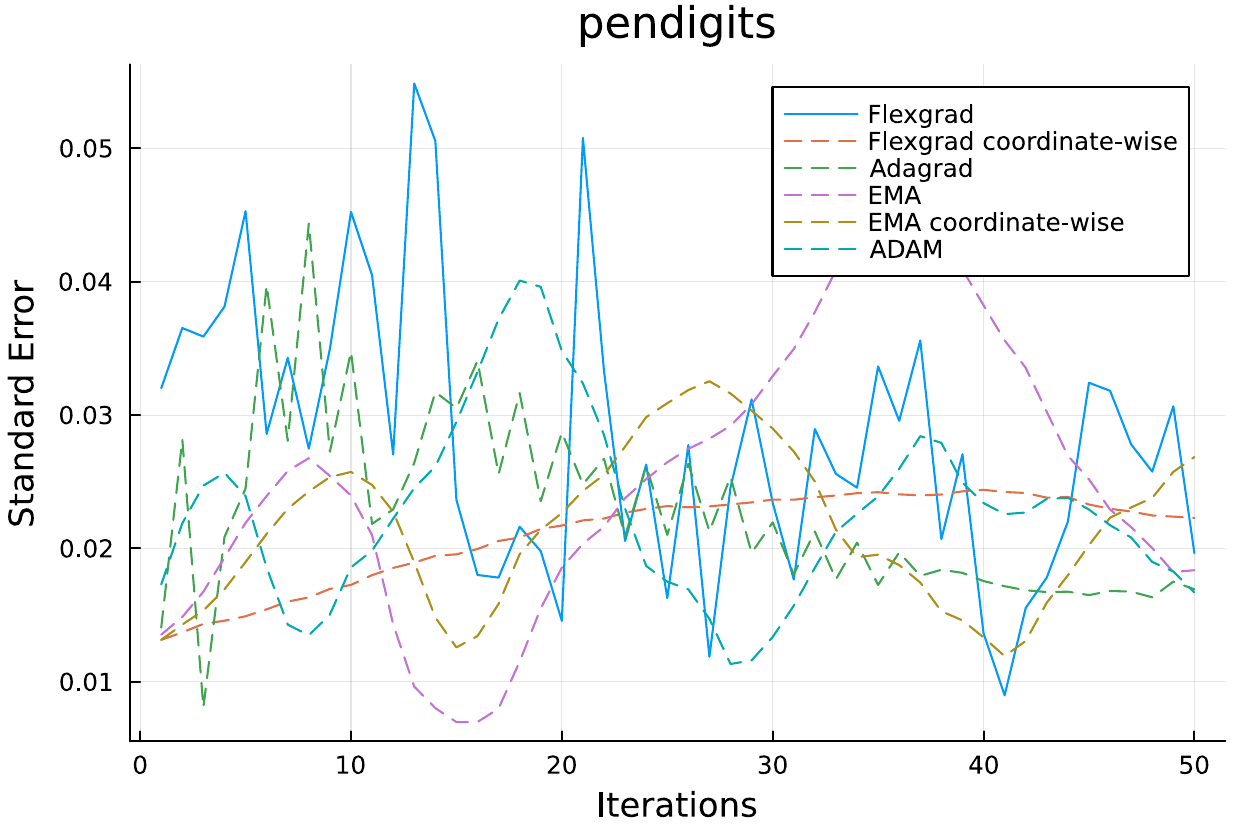}&
        \includegraphics[scale=\figscale]{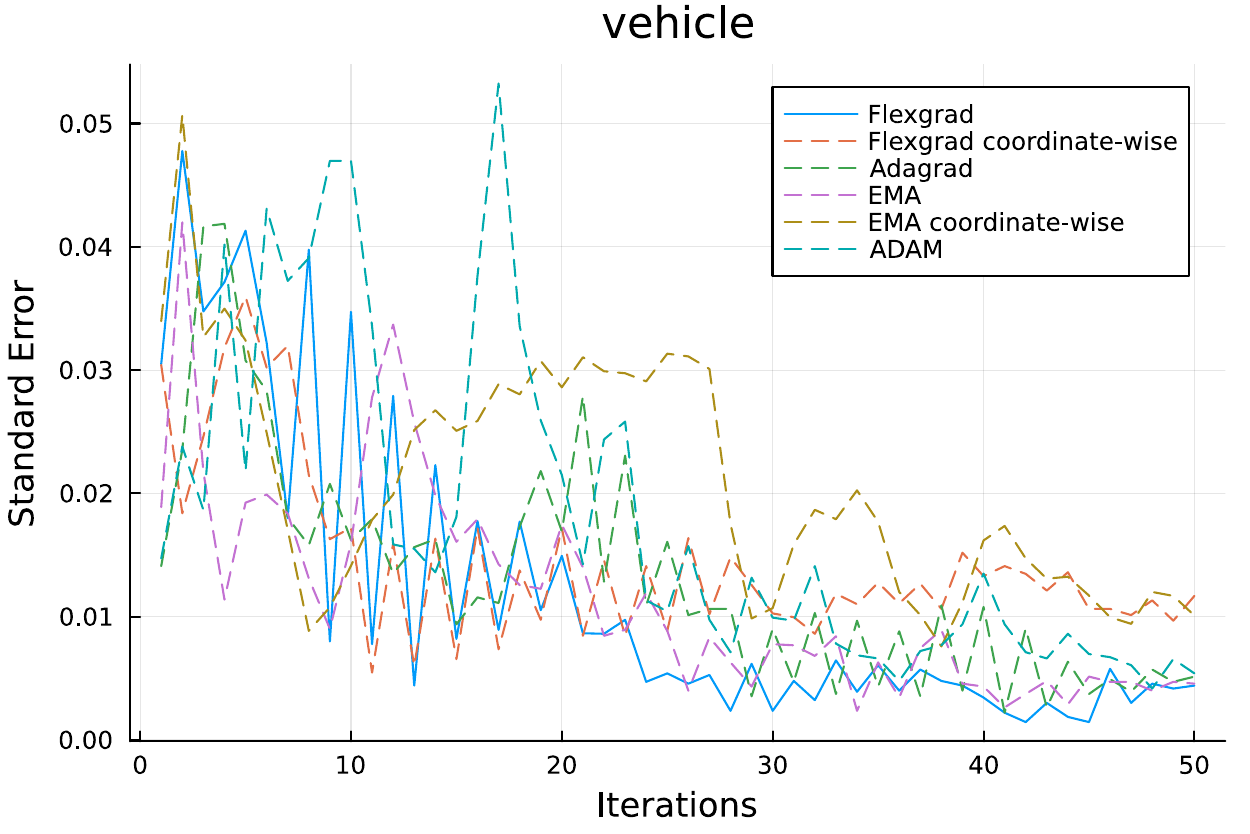}\\
        \includegraphics[scale=\figscale]{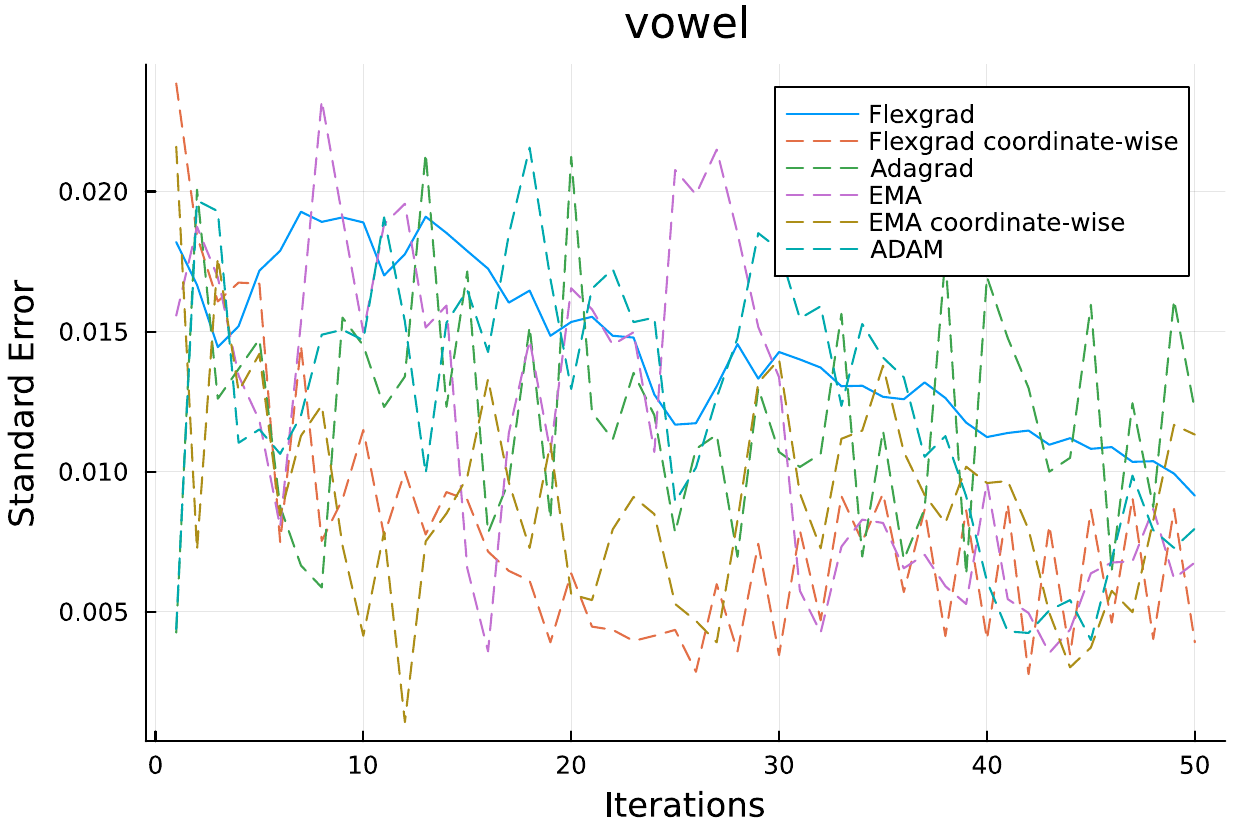}
        & \includegraphics[scale=\figscale]{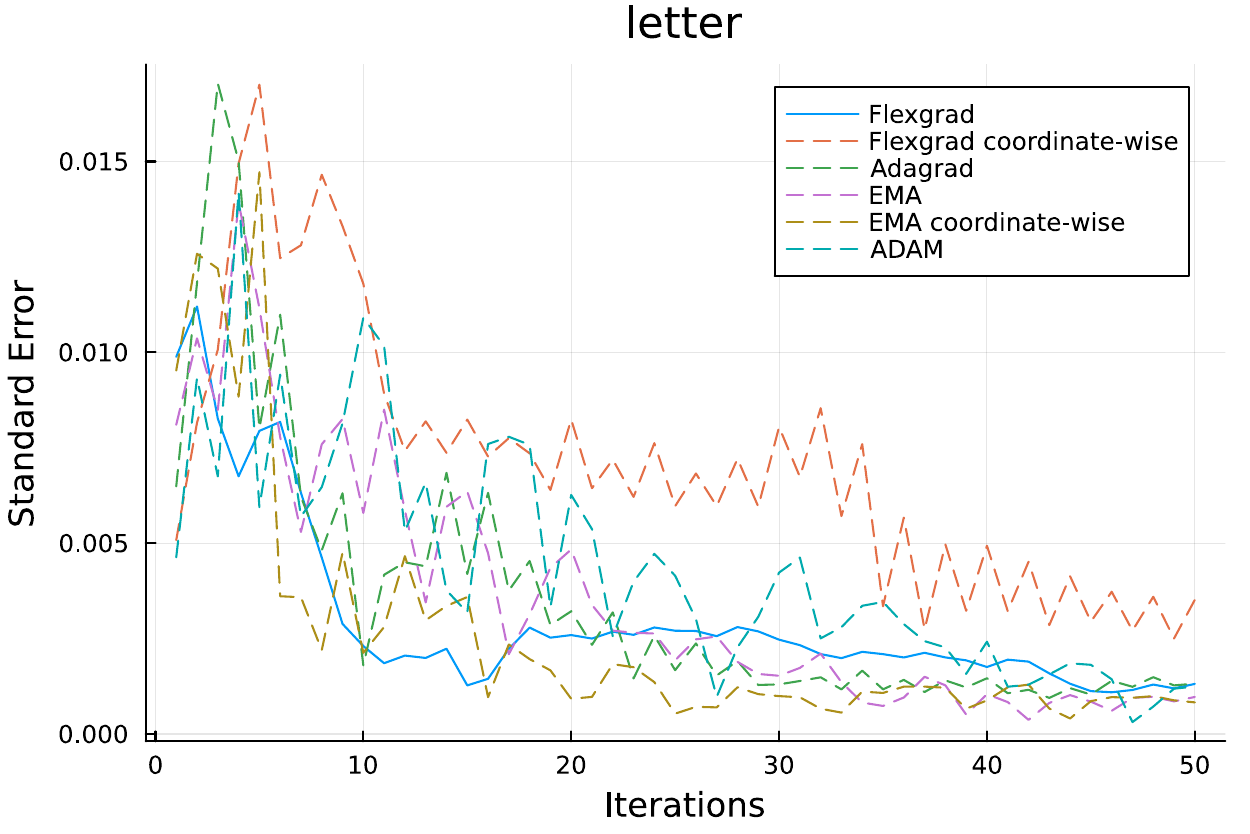}&
        \includegraphics[scale=\figscale]{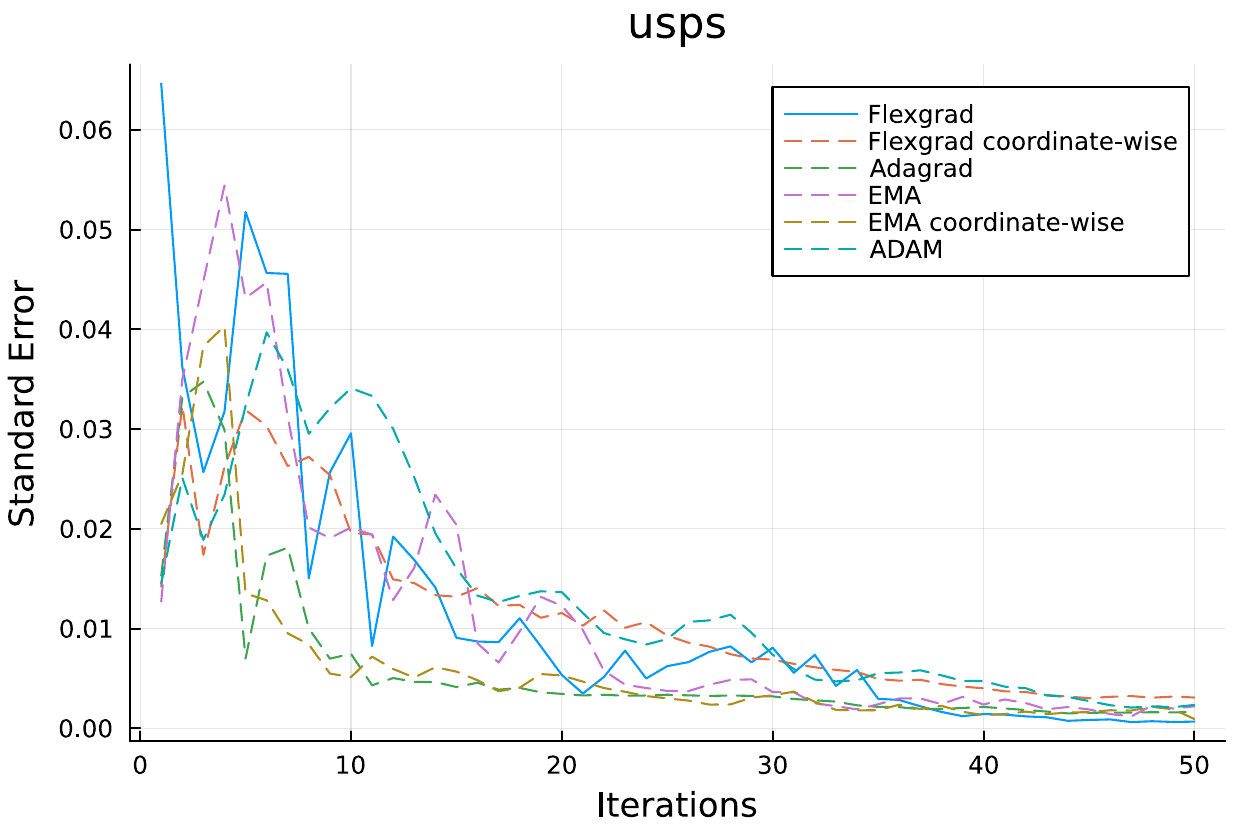}
     \end{tabular}
 \caption{Evolution of the standard deviation of the accuracy along iterations on 12 data sets for Flexgrad, Flexgrad coordinate-wise, Adagrad, EMA, EMA coordinate-wise, ADAM.}
\label{empprobmodels}
\end{figure}



\end{document}